\documentclass[11pt]{article}

\usepackage{hyperref}

\hypersetup{colorlinks = true, % hyperlink color
	    linkcolor = black,     % figure table back
	    urlcolor = black,       % url blue
        citecolor = blue}      % cite blue
\hypersetup{breaklinks=true}

\usepackage{authblk}
\usepackage{fullpage}  % for narrow margins
\usepackage{float}
\usepackage{pb-diagram}  		% for commutative diagrams
\usepackage{esvect}
\usepackage[T1]{fontenc}    % use 8-bit T1 fonts
\usepackage{hyperref}       % hyperlinks
\usepackage{url}            % simple URL typesetting
\usepackage{booktabs}       % professional-quality tables
\usepackage{amsfonts}       % blackboard math symbols
\usepackage{nicefrac}       % compact symbols for 1/2, etc.
\usepackage{microtype}      % microtypography
\usepackage{fullpage}  % for narrow margins
\usepackage{float}
\usepackage{pb-diagram}
\usepackage{makecell}		% for commutative diagrams
\usepackage{url}

\usepackage{algorithm} %//format of the algorithm
\usepackage{algorithmic} %//format of the algorithm

\usepackage{multirow} %//multirow for format of table
\usepackage{hhline}  % for \hhline in tables
\usepackage{amsmath}
\usepackage{xcolor}

\usepackage{cases}
\usepackage{ulem} 
\usepackage{cancel}
\usepackage{soul} 
\usepackage{graphicx}
\usepackage{amscd}
\usepackage{amssymb,mathrsfs}
\input{epsf.sty}
\usepackage{amsthm,amscd}
\usepackage{color}
\usepackage{latexsym}
\usepackage{epic}
\usepackage{appendix}
\usepackage{enumerate}
\usepackage{longtable}
\usepackage{lscape}
\usepackage{extarrows}
\usepackage{epstopdf}
\usepackage{listings}
\newlength\myindent

\newcommand{\Rmnum}[1]{\expandafter\@slowromancap\romannumeral #1@}

\allowdisplaybreaks

\newtheorem{definition}{Definition}[section]
\newtheorem{theorem}{Theorem}[section]
\newtheorem{lemma}{Lemma}[section]

\newtheorem{proposition}{Proposition}[section]
\newtheorem{assumption}{Assumption}[section]

\numberwithin{equation}{section}

\DeclareMathOperator{\T}{\mathrm{T}}
\DeclareMathOperator{\Hess}{\mathrm{Hess}}
\DeclareMathOperator{\grad}{\mathrm{grad}}

\def\M{\mathcal{M}}

\providecommand{\keywords}[1]
{
\small 
\textbf{Keywords:}
}

\begin{document}

\title{An increasing rank Riemannian method for generalized Lyapunov equations 
}
\author[1]{Zhenwei Huang}
\author[2]{Wen Huang}

\affil[1]{ School of Mathematical Sciences, Xiamen University, Xiamen, China.\vspace{.15cm}}
\affil[2]{ Corresponding author, School of Mathematical Sciences, Xiamen University, Xiamen, China. }

\maketitle

\begin{abstract}
In this paper, we consider finding a low-rank approximation to the solution of a large-scale generalized Lyapunov matrix equation in the form of $A X M + M X A = C$, where $A$ and $M$ are symmetric positive definite matrices. An algorithm called an Increasing Rank Riemannian Method for Generalized Lyapunov Equation (IRRLyap) is proposed by merging the increasing rank technique and Riemannian optimization techniques on the quotient manifold $\mathbb{R}_*^{n \times p} / \mathcal{O}_p$. To efficiently solve the optimization problem on $\mathbb{R}_*^{n \times p} / \mathcal{O}_p$, a line-search-based Riemannian inexact Newton's method is developed with its global convergence and local superlinear convergence rate guaranteed. 
Moreover, we derive a preconditioner which takes $M \neq I$ into consideration. 
%\sout{Numerical experiments show that IRRLyap with one of the Riemannian metrics is most efficient and robust in general and is preferable compared to the tested state-of-the-art methods when the lowest rank solution is desired.}
Numerical experiments show that the proposed Riemannian inexact Newton's method and preconditioner has superior performance and IRRLyap is preferable compared to the tested state-of-the-art methods when the lowest rank solution is desired.

\begin{keywords}

Generalized Lyapunov equations, Riemannian optimization, Low-rank approximation, Riemannian truncated Newton's method, Increasing rank method
\end{keywords}

\end{abstract}

\section{Introduction} \label{Intro}

This paper considers the large-scale generalized Lyapunov matrix equation 
\begin{equation}
  AXM + MXA - C = 0, \label{Intro-LyapMatEqu}
\end{equation} 
where $A, M \in \mathbb{R}^{n\times n}$ are large-scale, sparse, and symmetric positive definite matrices and $C \in \mathbb{R}^{n\times n}$ is symmetric positive semidefinite. It has been shown in~\cite{Chu87,penzl_numerical_1998} that Equation~\eqref{Intro-LyapMatEqu} has a unique solution $X^*$ and $X^*$ is symmetric positive semidefinite. 
 Moreover, it is known that under certain circumstances such as $C$ is low rank, the solution $X^{*}$ of the Lyapunov equation \eqref{Intro-LyapMatEqu} can be well approximated by a low-rank matrix, since the eigenvalues of the solution $X^*$ numerically has an exponential decay~\cite{PENZL00,Grasedyck2004ExistenceOA}. 

In recent years, solving Equation (\ref{Intro-LyapMatEqu}) has attracted much attention since it plays a significant role in model reduction~\cite{Moore1981PrincipalCA,Schilders2014ModelOR,Hamadi2021AMR,ISC22}, signal processing~\cite{Goyal2020ImageDR}, systems and control theory~\cite{LMESSC96,Antoulas2005ApproximationOL}, and solutions of PDEs~\cite{DV16}. 
Equation~\eqref{Intro-LyapMatEqu} with a large-scale sparse system matrix $A$ and a large-scale sparse mass matrix $M$ can be yielded by the finite element method semidiscretized in space~\cite{Bart10}.

%For example, consider the following generalized linear invariant-time system 
%\[	
%	\Sigma(E,A,C):\begin{cases}
%		E\dot{x}(t) = Ax(t), \\ 
%		\;\;\;y(t) = Cx(t),
%	\end{cases}
%\]
%	where $x(t)$ and $y(t)$ are the state vector and output vector respectively. \sout{Then} \zwhcomm{The solution of }{} the following generalized Lyapunov matrix equation 
%	\[
%		A^TXE + E^TXA + E^TC^TCE=0
%	\] 
%	captures the stability of $(E,A)$ \cite{Ishihara02}.
	
\subsection{Contributions}

In this paper, Problem~\eqref{Intro-LyapMatEqu} is formulated as an optimization problem on a quotient manifold $\mathbb{R}_*^{n \times p} / \mathcal{O}_p$ (see definition in Section~\ref{PrelNotat}). A line-search-based Riemannian truncated Newton's method is developed by using the truncated Newton in~\cite{dembo_truncated-newton_1983} and the line search conditions in~\cite{Byrd1989ATF}. The global convergence and local superlinear convergence rates are established. As far as we know, this is the first line-search-based generic Riemannian inexact Newton's method for optimization and guarantees global and local superlinear convergence. In addition, the line search conditions that we used allow conditions other than the Wolfe conditions in~\cite{dembo_truncated-newton_1983}.
 We equip the quotient manifold $\mathbb{R}_*^{n \times p} / \mathcal{O}_p$ with a Riemannian metric in~\cite{Zheng2022RiemannianOU} and explore the performance of the Riemannian truncated Newton's method for solving~\eqref{Intro-LyapMatEqu}. 
 Moreover, we propose an new preconditioner which is more effective compared to the one in~\cite{Bart10} in the sense that the new one takes $M \neq I$ into consideration whereas the one in~\cite{Bart10} does not. We finally combine the Riemannian truncated Newton's method with the increasing rank algorithm in~\cite{Bart10} and give Algorithm~\ref{FinalAlgPrecond-RLyap-RNewton} for solving~\eqref{Intro-LyapMatEqu}.
Numerically, the proposed Riemannian truncated Newton's method is able to find higher accurate solutions compared to the Riemannian trust-region Newton's method in~\cite{Absil2007TrustRegionMO}, the Riemannian conjugate gradient method~\cite{Boumal2015LowrankMC}, and the Riemannian memory-Limited BFGS method~\cite{HAGH18}; the proposed preconditioner can further reduce the number of actions of the Hessian evaluations when $M \neq I$ compared to the pre-conditioner in~\cite{Bart10}; and the increasing rank method with the Riemannian truncated Newton's method is able to find lower rank solutions compared to the three state-of-the-art low-rank methods K-PIK~\cite{simoncini_new_2007}, RKSM~\cite{DS11,KS20}, and mess\_lradi~\cite{SaaKB21-mmess-2.2} when the residual $\|AXM+MXA-C\|_F$ is roughly the same.

\subsection{Related Work}

%\sout{The classical numerical methods for the standard Lyapunov matrix equation, i.e., setting $M=I$, are the Bartels-Stewart method \cite{BS72}, the Hammarling method \cite{Ham82} and the Hessenberg–Schur method \cite{GN79}, which are based on the Schur decomposition of the system matrix A. }
Classical numerical methods for the standard Lyapunov matrix equations, i.e., setting $M=I$, include the Bartels-Stewart method \cite{BS72}, the Hessenberg–Schur method \cite{GN79}, %which is based on the Schur decomposition of the system matrix $A$, 
and the Hammarling method \cite{Ham82}. These methods are direct methods and the computations involve $O(n^3)$ number of floating point operations, which prevents the use for large-scale problems, where the floating point operation is defined in~\cite{Golub1996MatrixC}. One approach to develop an efficient method for large-scale problems is to explore the low-rank structure of the solution $X^*$.
%Hence, it is important to design a solver for computing large-scale problems.

%compute the matrix $X$ of size $n\times n$ in $O(n^3)$ floating point operations \zwhcomm{according to \cite{Golub1996MatrixC}}{}. This gives rise to that they only are fit for small-scale problems, e.g, $n\le 10^4$. It, hence, is especially important to design a solver for computing large-scale problems.

% \whcomm{ZHTODO}{}

\paragraph{Low-rank methods for Lyapunov equations.} In practice, there is a class of applications of generalized Lyapunov equation (\ref{Intro-LyapMatEqu}), whose solution %\sout{$X^*$ has rank $p\ll n$} 
 can be approximated with a low-rank matrix of rank $p\ll n$. In this case, the number of unknowns is reduced to $np$, which is greatly less than $n^2$. If we can compute a low-rank approximation in $O(np^c)$ operations with constant $c$, for large problems, then the complexity is significantly reduced.
 % the benefits of doing so are very obvious. 
 Based on this idea, very diverse low-rank methods have been proposed. 
Majority of these methods are based on Smith method~\cite{Gugercin2003}, low-rank alternating direction implicit iterative (LR-ADI)~\cite{low-rank_2004,Benner2009OnTA,Benner2014ComputingRL,PDJ22}, sign function methods~\cite{Baur2006FactorizedSO,Baur2008LowRS}, Krylov subspace techniques~\cite{simoncini_new_2007,DS11,Hamadi2021AMR}, and Riemannain optimization~\cite{Bart10}. It is worth noting that there exist several low-rank methods for algebraic Riccati equations~\cite{Mishra2013ARA,Benner2021ALS}, which, in some cases, are also suited for Lyapunov equations.

	% \sout{A kind of methods are focused on Smith method whose typical representative is the modified LR-Smith(l) method~\cite{Gugercin2003}. The work in~\cite{low-rank_2004} presented alternating direction implicit (ADI) based low-rank solvers which is the Cholesky factor–alternating direction implicit (CF–ADI) algorithm. For the low-rank ADI method, one of the most computationally expensive steps is solving the shifted systems and thus Benner et al.~\cite{PDJ22} proposed a new scheme that use the extended Krylov subspace method for solving shifted systems. Simoncini~\cite{simoncini_new_2007} concentrated on methods based on Krylov subspaces and proposed the Krylov-plus-inverted-Krylov (KPIK) method. The work in~\cite{LW07} is based on multigrid algorithms.}

		The low-rank methods listed above in addition to~\cite{Bart10,Mishra2013ARA} all start from well-known iterative methods and skillfully rewrite the problem so that the problem can be transformed into a low-rank setting. Although the calculations of these methods are cheap at each step since they work on a factor $Y$ of iterate $X=YY^T$, the convergence rate of these methods is usually unknown or at most linear and no good acceleration strategy has been given. As a consequence, Bart et al. \cite{Bart10} reformulated Equation (\ref{Intro-LyapMatEqu}) into an optimization problem defined on the set of fixed rank symmetric positive semidefinite matrices. %, a manifold, by minimizing the energy norm of the error. 
  In this paper, we follow this idea but from the perspective of quotient manifold and %\sout{use a new} \zwhcomm{
  propose a
  %}{} 
  Riemannian Newton's method to solve %\sout{our} \zwhcomm{
  the
  %}{} 
  problem. Moreover, it is empirically shown in Section~\ref{sec:ComAlg1} that the quotient manifold geometry is superior to the embedded submanifold geometry and the intuition is also discussed therein.

\paragraph{Riemannian Newton's methods.} Newton's method is a powerful tool for finding the minimizer of nonlinear functions in Euclidean spaces. Although Newton’s method has a fast local convergence rate, it is highly sensitive to the initial iterate (i.e., it is not globally convergent); it is not defined at points where the Hessian is singular; and for non-convex problems, it does not necessarily generate a sequence of descent directions. In order to overcome these issues, Dembo et al. \cite{dembo_truncated-newton_1983} proposed a truncated-Newton's method. The truncated-Newton's method contains two parts. The outer iteration executes Newton's method based on line search, and the inner iteration solves the Newton equation inexactly by a truncated conjugate gradient method. % solved once in each major iteration. 
% \sout{so} \zwhcomm{
% To this end, %}{} 
% conjugate gradient method is chosen to solve Newton's equation in the minor iteration. 
Absil et al. \cite{Absil2007TrustRegionMO} proposed a Riemannian trust-region Newton's method similar to this idea which is based on the trust-region method. The proposed method is a generalization of the one in \cite{dembo_truncated-newton_1983} to manifolds and further relaxes the line search condition used in~\cite{dembo_truncated-newton_1983}.  %
It is worth mentioning that there are some Riemannian Newton's methods based on line search; see, e.g.,~\cite{Bortoloti2018DampedNM,Zhao2018ARI, Wang2020RiemannianNM, Bortoloti2021AnED}. 
% \sout{They are different from our proposed method, such as the choice of search direction and line search conditions} 
As far as we know, all the existing Riemannian Newton's methods aim to find a root of a vector field, whereas the proposed Riemannian truncated-Newton's method is used to optimize a sufficiently smooth function.

% \zwhcomm{In essence, they are different from our proposed method which is to find a minimizer of a sufficiently smooth real-valued function, since they are designed to find a singularity of a differentiable vector field. Besides, in~\cite{Bortoloti2018DampedNM, Bortoloti2021AnED}, a descent direction is set as the negative gradient direction of a metric function if no Newton's direction exists.}{}

\subsection{Outline}

This paper is stated as follows. In Section \ref{PrelNotat}, we introduce preliminaries on algebra and manifold. In Section \ref{ProbStat}, we reformulate Equation (\ref{Intro-LyapMatEqu}) as an optimization problem on quotient manifold $\mathbb{R}_*^{n\times p}/\mathcal{O}_p$. The classical truncated Newton's method is generalized to Riemannian setting in Section \ref{RieNew} and it is proven that the proposed Riemannian truncated Newton's method converges to a stationary point from any starting point and has a local superlinear convergence rate. In Section~\ref{QuotMani}, we give ingredients for optimization on Riemannian quotient manifolds with three metrics. In Section \ref{FinalAlgPrecond}, the proposed preconditioner  is discussed, and the increasing rank algorithm for solving Equation (\ref{Intro-LyapMatEqu}) is given. Subsequently, in Section \ref{NumExp}, we demonstrate the performance of the proposed Riemannian truncated-Newton's method preconditioner, and compare the proposed method with other existing low-rank methods for Lyapunov equations. Finally, concluding remarks are given in Section~\ref{Concl}.

\section{Preliminaries and Notation} \label{PrelNotat}

Throughout this paper, we denote the real matrices space of size $m$-by-$n$ by $\mathbb{R}^{m\times n}$. For $A\in\mathbb{R}^{m\times n}$, the matrix $A^T\in\mathbb{R}^{n\times m}$ denotes the transpose of $A$. Define $\mathrm{vec}(\cdot): \mathbb{R}^{m\times n} \rightarrow \mathbb{R}^{mn}$ as an isomorphism operator which transforms a matrix into a vector by column-wise stacking, that is, for any $A=[a_1,a_2,\cdots,a_n]\in\mathbb{R}^{m\times n}$. We have $\mathrm{vec}(A)=[a_1^T,a_2^T,\dots,a_n^T]\in\mathbb{R}^{mn}$.
Let $\otimes$ denote the Kronecker product, i.e., for $A=[a_{ij}]\in \mathbb{R}^{m\times n}$ and $B=[b_{ij}]\in\mathbb{R}^{p\times q}$, it holds that $A\otimes B=[a_{ij}B]\in\mathbb{R}^{mp\times nq}$. 
%Then with the identity $\mathrm{vec}(ABC)=(C^T\otimes A)\mathrm{vec}(B)$ for $A, B,C\in \mathbb{R}^{n\times n}$, we have that Equation (\ref{Intro-LyapMatEqu}) is equivalent to 
%
%\begin{equation}
%  \label{Pre_Not-Euiv_LinearEqu}
%  L\mathrm{vec}(X)=\mathrm{vec}(C),
%\end{equation}
%where $L=A\otimes M+M\otimes A$. By our assumptions, i.e., $A$ and $M$ are symmetric positive definite, the matrix $L$ is symmetric positive definite \cite[Proposition 3.1]{Bart10}. Because $L$ is positive definite, it is not difficult to verify that the function $\|\cdot\|_L=\sqrt{\left<\cdot,\cdot\right>_L}$ with $\left<x,x\right>_L=\left<x,Lx\right>=x^TLx$ defines a norm over $\mathbb{R}^n$ called $L$-norm. 

\paragraph{Riemannian geometry.} The Riemannian geometry and optimization used in this paper can be found in the standard literatures, e.g.,~\cite{Boothby1975AnIT, AbsMahSep2008}, and the notation below follows~\cite{AbsMahSep2008}. Denote a manifold by $\mathcal{M}$. For any $x\in \mathcal{M}$, $\mathrm{T}_x\mathcal{M}$ is the tangent space of $\mathcal{M}$ at $x$ and elements in $\mathrm{T}_x\mathcal{M}$ are call tangent vectors of $\mathcal{M}$. Tangent bundle, $\mathrm{T}\mathcal{M} $, of manifold $\mathcal{M}$ is the union set of all tangent spaces. Mapping $\eta:\mathcal{M}\rightarrow \mathrm{T}\mathcal{M}$ maps a point $x\in\mathcal{M}$ into a tangent vector $\eta_x\in\mathrm{T}_x\mathcal{M}$ and is called a vector field. 
A metric on a manifold $\mathcal{M}$ is defined as $g(\cdot,\cdot): \mathrm{T}\mathcal{M} \oplus \mathrm{T}\mathcal{M} \rightarrow \mathbb{R}$, where $\mathrm{T}\mathcal{M} \oplus \mathrm{T}\mathcal{M}=\{(\eta_x,\xi_x):x\in \mathcal{M},\eta_x\in\mathcal{T}_x\mathcal{M},\xi_x\in\mathrm{T}_x\mathcal{M}\}$ is the Whitney sum of tangent bundle. Given $\eta_x \in \T_x \mathcal{M}$, its induced norm is defined by $\|\eta_x\| = \sqrt{g(\eta_x, \eta_x)}$.
If $g$ is smooth in the sense that for any two smooth vector fields $\xi$ and $\eta$, function $g(\xi,\eta):\mathcal{M}\rightarrow \mathbb{R}$ is smooth, then $\mathcal{M}$ equipped with $g$ is a Riemannian manifold and $g$ is called a Riemannian metric. 
If $\mathcal{M}$ is a Riemannian manifold, then the tangent space $\mathrm{T}_x\mathcal{M}$ is a Euclidean space with the Euclidean metric $g_x(\cdot,\cdot)=g(\cdot,\cdot)|_{\mathrm{T}_x\mathcal{M}}$. For a smooth function $f:\mathcal{M}\rightarrow \mathbb{R}$, notations $\mathrm{grad}f(x)$ and $\mathrm{Hess}f(x)$ denote the Riemannian gradient and Hessian of $f$ at $x$ respectively, and the action of $\mathrm{Hess}f(x)$ on a tangent vector $\eta_x\in\mathrm{T}_x\mathcal{M}$ is denoted by $\mathrm{Hess}f(x)[\eta_x]$. 
Let $\mathcal{H}_x$ be a linear operator on $\T_x\M$. Its adjoint operator, denoted by $\mathcal{H}_x^*$, satisfies $g_x(\eta_x,\mathcal{H}_x\xi_x)=g_x(\mathcal{H}^*_x\eta_x,\xi_x)$ for all $\eta_x,\xi_x\in \T_x\M$. With $\mathcal{H}_x=\mathcal{H}_x^*$, we call $\mathcal{H}_x$ self-adjoint with respect to $g$.

Let $\gamma:\mathcal{I} \rightarrow \mathcal{M}$ denote a smooth curve and $\dot{\gamma}(t)\in\mathrm{T}_{\gamma(t)}\mathcal{M}$ denotes its velocity at $t$, where $\mathcal{I} \subset \mathbb{R}$ is open and $[0, 1] \subset \mathcal{I}$. The distance between $x=\gamma(0)$ and $y=\gamma(1)$ is defined by $\mathrm{dist}(x,y)= \inf_{\{\gamma:\gamma(0)=x,\gamma(1)=y\}}\int_0^1\|\dot{\gamma}(t)\|\mathrm{d}t$. A smooth curve $\gamma:[0,1]\rightarrow \mathcal{M}$ is called the geodesic if it locally minimizes the distance between $\gamma(0)=x$ and $\gamma(1)=y$. %\zwhcomm{
Notation $B_{\mu}(0_x)=\{\xi_x\in \mathrm{T}_x\mathcal{M}: \|\eta_x\| \le \mu\}$ is used to denote the open ball in $\mathrm{T}_x \mathcal{M}$ of radius $\mu$ centered at $0_x$ and $B_{\mu}(x)=\{y\in\mathcal{M}:\mathrm{dist}(x,y)\le \mu\}$ is an open ball in $\mathcal{M}$ of radius $\mu$ centered at $x$. A retraction $R:\T\M \rightarrow \M$ is a map satisfying (i) $R(0_x)=x$ for all $x\in \M$ and (ii) for each fixed $x\in \M$, $\frac{\mathrm{d}}{\mathrm{d} t}R(t\eta_x)\big|_{t=0}=\eta_x$ for all $\eta_x\in\T_x\M$. 
%}{}
%For any $\eta_x\in\mathrm{T}_x\mathcal{M}$, the exponential map is defined as $\mathrm{Exp}_{x}(\eta_x)=\gamma(1)$, where $\gamma(0)=x$ and $\dot{\gamma}(0)=\eta_x$ is the geodesic. \zwhcomm{The injectivity radius of $\mathcal{M}$ at $x$, denoted by $\mathrm{inj}(x)$ is the supremum over radii $r>0$ such that $\mathrm{Exp}_x$ is defined and is diffeomorphism on the open ball $B_r(x)=\{u\in\mathrm{T}_x\mathcal{M}:\|u\|_x\le r\}$.}{} Parallel transport on $\mathcal{M}$ defined as $\mathrm{P}_{0\rightarrow1}^{\gamma}:\mathrm{T}_{\gamma(0)}\mathcal{M}\rightarrow\mathrm{T}_{\gamma(1)}\mathcal{M}$ is a linear isometry along the geodesic $\gamma$.   
%\zwhcomm{For details, please refer to \cite{AbsMahSep2008} or \cite{Boumal:300281}.}{}

\paragraph{The considered manifold.} The set of symmetric positive semidefinite matrices of size $n$ with fixed rank $p$, denoted by $\mathcal{S}_+(p,n)$, is a Riemannian manifold with dimension $np-\frac{1}{2}p(p-1)$ (see, e.g., \cite{Helmke1995CriticalPO}). For any $X\in\mathcal{S}_+(p,n)$, there exists a matrix $Y\in \mathbb{R}^{n\times p}_*$ such that $X=YY^T$, where $\mathbb{R}^{n\times p}_* = \{Y \in \mathbb{R}^{n \times p} : \mathrm{rank}(Y) = p\}$ is called a noncompact Stiefel manifold \cite[Chapter 3]{AbsMahSep2008}. The orthogonal group $\mathcal{O}_p = \{O \in \mathbb{R}^{p \times p}: O^T O = I_p\}$ is a Lie group with the group operator given by the matrix product; it can be equipped with the smooth structure as an embedded submanifold of $\mathbb{R}^{p \times p}$; and its identity is the identity matrix $I_p$. Define a right group action of $\mathcal{O}_p$ on $\mathbb{R}^{n\times p}_*$ as: 

\[
  \theta :\mathbb{R}^{n\times p}_* \times \mathcal{O}_p \rightarrow \mathbb{R}^{n\times p}_*:\theta(Y,Q)=YQ.
  \]
Obviously, this group action satisfies the identity and compatibility conditions~\cite[Definition 9.11]{Boumal:300281}, and therefore it induces an equivalent relation $\sim$ on $\mathbb{R}^{n\times p}_*$: 

\[
  Z \sim Y \Leftrightarrow Y=\theta(Z,Q) = ZQ \text{ for some } Q\in \mathcal{O}_p.
  \]
The orbit of $Y$ forms an equivalence class, namely, $[Y]=\{YQ:Q\in \mathcal{O}_p\}$. We denote the quotient space $\mathbb{R}^{n\times p}_*/\sim$ as $\mathbb{R}_*^{n\times p}/\mathcal{O}_p$. The natural projection between $\mathbb{R}^{n\times p}_*$ and $\mathbb{R}_*^{n\times p}/\mathcal{O}_p$ is given by 

\[
  \pi: \mathbb{R}^{n\times p}_* \rightarrow \mathbb{R}^{n\times p}_* / \mathcal{O}_p : Y \mapsto \pi(Y)=[Y].
  \]
From now on, out of clarity, we denote $[Y]$ by $\pi(Y)$ as an element of the quotient manifold $\mathbb{R}_*^{n\times p}/\mathcal{O}_p$ and use $\pi^{-1}(\pi(Y))$ to denote $[Y]$ when it is regarded as a subset of $\mathbb{R}_*^{n\times p}$. 

Since the Lie group $\mathcal{O}_p$ acts smoothly, freely and properly with the group action $\theta$ on the smooth manifold $\mathbb{R}^{n\times p}_*$, by~\cite[Theorem 9.17]{Boumal:300281}, the quotient space $\mathbb{R}^{n\times p}_*/ \mathcal{O}_p$ is a manifold called a quotient manifold of $\mathbb{R}^{n\times p}_*$. The manifold $\mathbb{R}_*^{n \times p}$ is therefore called the total space of $\mathbb{R}^{n\times p}_*/ \mathcal{O}_p$. Since $\mathbb{R}^{n\times p}_*/\mathcal{O}_p$ is diffeomorphic to $\mathcal{S}_+(p,n)$, the quotient manifold $\mathbb{R}^{n\times p}_*/\mathcal{O}_p$ can be viewed as a representation of the manifold $\mathcal{S}_+(p,n)$. 

%Indeed, first of all, $\beta$ must be surjection; besides let $\xcancel{[Y],[Z]\in \mathbb{R}^{n\times p}_*} \zwhcomm{\pi(Y),\pi(Z)\in\mathbb{R}^{n\times p}_*/\mathcal{O}_p}{}$ such that \sout{$\beta([Y])\not=\beta([Z])$} \zwhcomm{$\beta(\pi(Y))\not=\beta(\pi(Z))$}{}, i.e., $YY^T\not=ZZ^T$, which implies that there exists no $Q\in \mathcal{O}_p$ such that $Y=ZQ$, and thus \sout{$[Y]\not=[Z]$} \zwhcomm{$\pi(Y)\not=\pi(Z)$}{}.

\paragraph{Function classes.} 
% {\color{red}
For the convergence analysis, we need notions of a pullback of a real-valued function $f$ with respect to a retraction $R$ and the radially Lipschitz continuous differentiability of a pullback. 
\begin{definition} \label{PrelNotat-Pullback}
	Let $f: \mathcal{M} \rightarrow \mathbb{R}^n$ be a real-valued function on manifold $\mathcal{M}$ with a retraction $R$. We call $\hat{f}:\mathrm{T}\mathcal{M} \rightarrow \mathbb{R}: \eta_x \mapsto \hat{f}(\eta_x)=f(R(\eta_x))$ the pullback of $f$ with respect to $R$. When restricted on $\mathrm{T}_x\mathcal{M}$, we denote $\hat{f}\big|_{\mathrm{T}_x\mathcal{M}}$ by $\hat{f}_x$.
\end{definition}
% }
\begin{definition} (\cite[Definition 7.4.1]{AbsMahSep2008}) 
\label{PrelNotat-RLipConti}
  Let $\hat{f}$ be a pullback of $f$ with respect to a retraction $R$. $\hat{f}$ is referred to as a radially $L$-$C^1$ function for all $x\in \mathcal{M}$ if there exists a positive constant $L$ such that for all $x\in \mathcal{M}$ and all $\eta_x\in \mathrm{T}_x\mathcal{M}$, it holds that 
  \begin{equation}
    \begin{aligned}
      \bigg| \frac{\mathrm{d}}{\mathrm{d}\tau}\hat{f}(t\eta_x)\big|_{t=\tau} - \frac{\mathrm{d}}{\mathrm{d}t}\hat{f}(t\eta_x)\big|_{t=0} \bigg|\le L \tau\|\eta\|^2,
    \end{aligned}
  \end{equation}
  where $\tau$ and $\eta_x$ satisfy that $\mathrm{R}_x(\tau\eta_x)\in \mathcal{M}$.
\end{definition}

%\begin{definition}(\cite[Definition 10.4.2]{Boumal:300281})
%	\label{PrelNotat-VecFieldLipConti}
%	A vector field $\eta$ on a connected manifold $\mathcal{M}$ is $L$-Lipschitz continuous if, for all $x,y\in\mathcal{M}$ with $\mathrm{dist}(x,y)<\mathrm{inj}(x)$, it holds that 
%	\[
%		\|\mathrm{P}_{0\leftarrow 1}^{\gamma}\eta_y-\eta_x\|\le L\mathrm{dist}(x,y),
%	\]
%	where $\gamma:[0,1]\rightarrow \mathcal{M}$ is the unique minimizing geodesic connecting $x$ to $y$, and $\mathrm{inj}(x)$ is the injectivity radius of $\mathcal{M}$ at $x$.
%\end{definition}

%For $x\in\mathcal{M}$, we also need the notion of Riemannian Lipshitz continuity for a linear operator $H(x):\mathrm{T}_x\mathcal{M} \rightarrow \mathrm{T}_x\mathcal{M}$. 
%\begin{definition}
%	Suppose that $\mathcal{M}$ is connected. We say $H$ is L-Lipschitz continuous if for all $x,y\in\mathcal{M}$ such that $\mathrm{dist}(x,y)<\mathrm{inj}(x)$ we have 
%	\[
%	\|\mathrm{P}_{0\leftarrow 1}^{\gamma}\circ H(y) \circ \mathrm{P}_{1\leftarrow0}^{\gamma}-H(x)\| \le L\mathrm{dist}(x,y),
%	\]
%	where $\|\cdot\|$ denotes the operator norm with respect to the Riemannian metric and $\gamma:[0,1]\rightarrow \mathcal{M}$ is the unique minimizing geodesic connecting $x$ to $y$.
%\end{definition}

\paragraph{Notations.} %At the end of this section, we list some notations that will be used later. 
For any $Y\in \mathbb{R}_*^{n\times p}$ with $p<n$, the matrix $Y_\perp$ of size $n\times(n-p)$ denotes the normalized orthogonal completment of $Y$, i.e., $Y_\perp^TY_\perp=I_{n-p}$ and $Y_\perp^TY=0$. Furthermore, $Y_{\perp_M}$ is the normalized orthogonal complement of $MY$. The sysmmetric matrices space of size $n\times n$ is denoted by $\mathcal{S}_n^{\mathrm{sym}}$. For any matrix $X\in\mathcal{S}_n^{\mathrm{smy}}$, the symbols $X \succ0$ and $X\succeq0$ mean that $X$ is positive definite and positive semidefinite respectively; notation $\lambda(X)$ denotes any eigenvalue of $X$, whereas $\lambda_{\min}(X)$ and $\lambda_{\max}(X)$ denote the minimum eigenvalue and the maximum eigenvalue of $X$. For a square matrix $X$, $\mathrm{tr}(X)$ denotes the sum of all diagonal elements. For $Y\in\mathbb{R}^{n\times p}$, let $\mathrm{sym}(Y)=(Y+Y^T)/2$ and $\mathrm{skew}(Y)=(Y-Y^T)/2$.

\section{Problem Statement} \label{ProbStat}

% \zwhcomm{
% Apply $\mathrm{vec(\cdot)}$ to the left and right sides of Eqiation~\eqref{Intro-LyapMatEqu} to get $L\mathrm{vec}(X)=\mathrm{vec}(C)$, where $L=A\otimes M + M\otimes A$ is symmetric positive definite by our assumptions~\cite[Proposition 3.1]{Bart10}. Thus, the function $\|\cdot\|_L: v \mapsto \sqrt{v^TLv}$ is a norm called $L$-norm. Let $X$ and $X^*$ denote the approximation of rank $p$ and the exact sulotuon of Equation (\ref{Intro-LyapMatEqu}) respectively. From the perspective of optimization, we can solve the Lyapunov equation (\ref{Intro-LyapMatEqu}) by minimizing the $L$-norm of the error $E=X-X^*$. This idea has been used in \cite{Bart10}, in which the following problem firstly is obtained 

We consider a fixed-rank optimization formulation of~\eqref{Intro-LyapMatEqu} in~\cite{Bart10}
% \begin{align*}
% 	&\min_{X\in \mathbb{R}^{n\times n}}  h(X):=\mathrm{trace}(XAXM) - \mathrm{trace}(XC), \\ 
% 	&\;\; \text{ s.t. } X\in \tilde{\mathcal{S}}_+(p,n)=\{X\in\mathbb{R}^{n\times n}:X\in\mathcal{S}_n^{\mathrm{sym}},\mathrm{rank}(X)\le p\}.
% \end{align*}
% Noting that $\tilde{\mathcal{S}}_+(p,n)$ is not a manifold, if the solution of Equation (\ref{Intro-LyapMatEqu}), $X^*$, has rank not less than $p$, by our assumption, then the minimizer of $h(X)$ over $\widetilde{\mathcal{S}}(p,n)$ has rank $p$ \cite[Proposition 3.2]{Bart10}. The optimization problem, hence, is relaxed as:
\begin{equation} 
  \begin{aligned} \label{Pro_stat-OptProb_Bart}
   & \min_{X\in \mathbb{R}^{n\times n}} h(X) :=\mathrm{tr}(XAXM) - \mathrm{tr}(XC), \\ 
   &\quad \text{ s.t. } X\in \mathcal{S}_+(p,n).
  \end{aligned}
\end{equation}
Note that the Euclidean gradient of~\eqref{Pro_stat-OptProb_Bart} is $\nabla h(X)=AXM + MXA - C.$ It follows that any stationary point of $h(X)$ is a solution of~\eqref{Intro-LyapMatEqu}. Since $\mathbb{R}_*^{n \times p} / \mathcal{O}_p$ is diffeomorphic to $\mathcal{S}_+(p, n)$, Problem~\eqref{Pro_stat-OptProb_Bart} can be equivalently reformulated as\footnote{The term ``equivalent'' means if $X\in \mathcal{S}_+(p,n)$ is a stationary point of $h$ then $\pi(Y)\in \mathbb{R}^{n\times p}/\mathcal{O}_p$ satisfying $X=YY^T$ is also a stationary point of $f$; conversely, if $\pi(Y)$ is a stationary point of $f$, then $X=YY^T$ is a stationary point of $h$.}: 
\begin{equation}
  \begin{aligned} \label{Pro_stat-FinalProb} %\label{Pro_stat-FuncOnQuotMani}
    \min_{\pi(Y)} f: \mathbb{R}_*^{n\times p}/\mathcal{O}_p \rightarrow \mathbb{R} :  \pi(Y) \mapsto h(YY^T)=\mathrm{tr}(Y^TAYY^TMY) - \mathrm{tr}(Y^TCY),
  \end{aligned}
\end{equation}
which is defined on the quotient manifold $\mathbb{R}_*^{n\times p}/\mathcal{O}_p$. Therefore, Riemannian optimization algorithms can be used. 
It is shown later in Section~\ref{NumExp} that the reformulated Problem~\eqref{Pro_stat-FinalProb} in the quotient manifold $\mathbb{R}_*^{n\times p}/\mathcal{O}_p$ has advantages over the original Problem~\eqref{Pro_stat-OptProb_Bart} from the viewpoint of Riemannian optimization.
Note that Problem~\eqref{Pro_stat-FinalProb} is different from the Burer and Monteiro approach~\cite{Burer2003} in the sense that \cite{Burer2003} optimizes over the factor $Y$ but Problem~\eqref{Pro_stat-FinalProb} is over the quotient manifold~$\mathbb{R}_*^{n\times p}/\mathcal{O}_p$.

Problem~\eqref{Pro_stat-FinalProb} requires a good estimation of the rank $p$, which is usually unknown in practice. It has been shown in~\cite{Bart10} that if the rank of the solution of~\eqref{Intro-LyapMatEqu} is larger than $p$, then any local minimizer of the fixed-rank formulation~\eqref{Pro_stat-FinalProb} must have full rank $p$. Therefore, one can estimate the rank $p$ by using a rank-increasing algorithm. The proposed Riemannian optimization over fixed rank manifold is discussed in Section~\ref{RieNew} and the rank-increasing algorithm is described in Section~\ref{FinalAlgPrecond}.

\section{A Riemannian Truncated Newton's Method} \label{RieNew}
% 这个章节将算法描述清楚，给出伪代码。同时给出算法的收敛性分析，复杂度分析以及实现方面的细节。这部分根据情况会需要分解成一些子章节。比如如果是算法设计的论文，伪代码和代码解释一个子章节，收敛性分析一个子章节（如果收敛性分析很长，可以单独成为一个章节），复杂度分析和实现一个子章节。

Riemannian optimization has attracted more and more researchers' attention for recent years and many related algorithms have been investigated, such as the Riemannian trust-region methods \cite{Absil2007TrustRegionMO}, the Riemannian steepest descent method \cite{AbsMahSep2008}, the Riemannian Barzilai Borwein method \cite{Iannazzo2018TheRB}, the Riemannian nonlinear conjugate gradient methods \cite{ring_optimization_2012, Sato2014ADR, Sato2013ANG, Zhu2017ARC}, the Riemannian versions of quasi-Newton's methods \cite{Huang2015ABC,Huang2015ARS,huang_riemannian_2018,Huang2022ALR}. %, and so forth. %\zwhcomm{}{} 
This paper considers the Riemannian optimization problems in the form of
\begin{equation} \label{RieNew-Prob}
  \begin{aligned} 
    \min_{x} f(x) \text{ s.t. }x\in \mathcal{M},
  \end{aligned}
\end{equation}
where $\mathcal{M}$ is a finite dimensional Riemannian manifold, and $f:\mathcal{M} \rightarrow \mathbb{R}$ is a real-valued function. It is further assumed throughout this paper that the below assumptions hold.

\begin{assumption} \label{TwiceContinuDiff}
  $f$ is twice continuously differentiable.
\end{assumption} 

\begin{assumption} \label{BoundedLevel}
  For all starting $x_0\in \mathcal{M}$ the level set 
  \[
    L(x_0) := \{x\in\mathcal{M} : f(x) \le f(x_0)\} 
    \]
    are bounded.
\end{assumption}

\begin{assumption} \label{RadLipCont}
	$f$ is radially Lipschitz continuous 	(Definition~\ref{PrelNotat-RLipConti}). 
\end{assumption}

%\whcomm{[ZWTODO: We need the function $f$ satisfies Definition~\ref{PrelNotat-RLipConti}. In the later proofs or descriptions, we need to point out where we need this assumption. (i) the existence of step size of the Wolfe conditions and Armijo-Goldstein conditions (ii) the Wolfe conditions and Armijo-Goldstein conditions implies BN condition.]}{}

\subsection{Algorithm Statement} %The Truncated-Newton's Method} 

% In the Euclidean setting, Newton's method is often used because of its superlinear or even quadratic rate of convergence. In this paper, therefore, we use also Riemannian Newton's method to solve the problem (\ref{RieNew-Prob}). 

%On Riemannian manifolds, all algorithms we consider is retraction-based. In other words, they iterate $x_{k+1} \gets R_{x_{k}}(\eta_k)$ for some step $\eta_k$. For Riemannian Newton's method, the step is called the Newton step given by solving the Newton equation: 
%\begin{equation}
%  \begin{aligned} \label{RieNew-NewEuq}
%    \mathrm{Hess}f(x)[\eta] = -\mathrm{grad}f(x),
%  \end{aligned}
%\end{equation}
%where the solution $\eta$ is our desired Newton step. 

% Since Newton's method will, in practice, encounter some drawbacks, such as  
%   (i) it is not globally convergent;
%   (ii) it is not defined at points where $\mathrm{Hess}f(x_k)$ is singular;
%   (iii) for non-convex problem, it even does not generete a sequence of descent directions 
% and so forth, Dembo and Steihaug in \cite{dembo_truncated-newton_1983} embedded the standard Newton's method in a modified Newton framework to break through these shortcomings. Here we will generalize it to Riemannian setting, as seen in Algorithm \ref{RieNew-TruncatedNewton}. Different with the method in \cite{dembo_truncated-newton_1983}, where the Wolfe conditions are required for computing stepsize, the stepsize in the proposed Algorithm \ref{RieNew-TruncatedNewton} only needs to satisfy weaker condition either (\ref{RieNew-BrydCond1}) or (\ref{RieNew-BrydCond2}).

The proposed Riemannian Truncated-Newton's method is stated in Algorithm~\ref{RieNew-TruncatedNewton}.

\begin{algorithm}[htbp]
  \caption{Riemannian Truncated-Newton's method} 
  \begin{algorithmic}[1] \label{RieNew-TruncatedNewton}
  \REQUIRE initial iterate $x_0$; line search parameters $\chi_1$ and $\chi_2\in(0,1)$; a preconditioner $\mathscr{P}(x):\T_x \mathcal{M} \rightarrow \T_x \mathcal{M}$ is a linear operator; a truncated conjugate gradient parameter $\epsilon>0$ and a forcing sequence $\{\phi_k\}$ satisfying $\phi_k>0$ and $\lim_{k\rightarrow \infty}\phi_k=0$; %\zwhcomm{}{}
  \ENSURE sequence $\{x_k\}$; %\sout{solution $x^*$} \zwhcomm{}{};
  \STATE Set $k\gets0$;
  \WHILE{do not accurate enough}
  % \STATE Compute $f(x_k)$, $\mathrm{grad}f(x_k)$ and $\mathrm{Hess}f(x_k)$;
  \STATE \label{RieNew:st01} Approximately solve the Newton equation $\Hess f(x_k) [\eta_k] = - \grad f(x_k)$ by the truncated conjugate gradient algorithm in Algorithm~\ref{RieNew-TrunConjGrad} with inputs $(\grad f(x_k), \Hess f(x_k), \mathscr{P}(x_k), \epsilon, \phi_k)$;
  \STATE \label{RieNew:st02} Find a step size $\alpha_k$ satisfying %\zwhcomm{ }{}
  \begin{equation}
  \begin{aligned}
    \label{RieNew-BrydCond1}
      h_k(\alpha_k)-h_k(0)\le -\chi_1 \frac{h'_k(0)^2}{\|\eta_k\|^2}, 
  \end{aligned}
  \end{equation}
  or 
  \begin{equation}
    \begin{aligned}
      \label{RieNew-BrydCond2}
      h_k(\alpha_k)- h_k(0) \le \chi_2 h'_k(0), 
    \end{aligned}
  \end{equation}
  where %\sout{$0<\chi_1<1$ and $0<\chi_2<1$ independent of $k$ are constants and} 
  $h_k(t)=f(R_{x_k}(t\eta_k))$;
  \STATE \label{RieNew:st03} Set $x_{k+1} \gets R_{x_k}(\alpha_k\eta_k)$;
  \STATE \label{RieNew:st04} Set $k\gets k+1$;
  \ENDWHILE
  
 \end{algorithmic}
\end{algorithm}

\begin{algorithm}[htbp]
  \caption{Truncated Preconditioned Conjugate Gradient Method (tPCG)} 
  \begin{algorithmic}[1] \label{RieNew-TrunConjGrad}
  \REQUIRE $(\mathrm{grad}f(x_k), \mathrm{Hess}f(x_k), \mathscr{P}(x_k), \epsilon, \phi_k)$, where $\epsilon>0$ and $\phi_k>0$;
  \ENSURE $\eta_k$;
  \STATE \label{RieNewtCG:st01} \textbf{Initializations:} \\ 
  Set $\eta^{(0)} \gets 0$, $r^{(0)}\gets -\mathrm{grad}f(x_k)$, $y^{(0)}\gets \mathscr{P}(x_k)^{-1}[r^{(0)}]$, $d^{(0)}\gets y^{(0)}$, $\delta^{(0)}\gets g_{x_k}(y^{(0)},y^{(0)})$, and $i\gets0$;
  \STATE \label{RieNewtCG:st02} \textbf{Check for negative curvature:} \label{RieNew-TrunConjGrad-CheckForNegCurv} \\ 
  Set $q^{(i)} \gets \mathrm{Hess}f(x_k)[d^{(i)}]$; \\
  If $g_{x_k}(d^{(i)}, q^{(i)}) \le \epsilon \delta_i$, return $\eta_k \gets \begin{cases} d^{(0)}, & \text{ if }i=0,\\ \eta^{(i)}, & \text{ otherwise}; \end{cases}$
  \STATE \label{RieNewtCG:st03} \textbf{Generate next inner iterate:}  \\ 
  Set $\alpha^{(i)} \gets g_{x_k}(r^{(i)}, y^{(i)})/g_{x_k}(d^{(i)},q^{(i)})$;\\ 
  Set $\eta^{(i+1)} \gets \eta^{(i)} + \alpha^{(i)}d^{(i)}$;
  \STATE \label{RieNewtCG:st04} \textbf{Update residual and search direction:} \\
  Set $r^{(i+1)} \gets r^{(i)} - \alpha^{(i)}q^{(i)}$; \\
  Set $y^{(i+1)} \gets \mathscr{P}(x_k)^{-1}[r^{(i+1)}]$; \\ 
  Set $\beta^{(i+1)}\gets g_{x_k}(r^{(i+1)},y^{(i+1)})/g_{x_k}(r^{(i)},y^{(i)})$;\\
  Set $d^{(i+1)} \gets y^{(i+1)} + \beta^{(i+1)}d^{(i)}$;\\ 
  Set $\delta^{(i+1)} \gets g_{x_k}(y^{(i+1)}, y^{(i+1)}) + \left(\beta^{(i+1)}\right)^2\delta^{(i)}$; 
  \STATE \label{RieNewtCG:st05} \textbf{Check residual:} \\ 
  If $\|r^{(i+1)}\|/\|\mathrm{grad}f(x_k)\|\le \phi_k$, return $\eta_k\gets \eta^{(i+1)}$; \\
  Set $i\gets i+1$; \\
  Return to Step \ref{RieNew-TrunConjGrad-CheckForNegCurv};
 \end{algorithmic}
\end{algorithm}

Step~\ref{RieNew:st01} of Algorithm~\ref{RieNew-TruncatedNewton} invokes Algorithm~\ref{RieNew-TrunConjGrad}, which approximately solves the Newton equation $\Hess f(x_k) [\eta_k] = - \grad f(x_k)$ by the truncated preconditioned conjugate gradient (tPCG) method. 
If $\Hess f(x_k)$ is not sufficiently positive definite along the direction $d^{(i)}$ in the sense that $g_{x_k}(d^{(i)}, \Hess f(x_k) [d^{(i)}])$ is sufficiently greater than 0, then the preconditioned conjugate gradient is terminated, see Step~\ref{RieNewtCG:st02} of Algorithm~\ref{RieNew-TrunConjGrad}. This early termination guarantees that the output $\eta_k$ is a sufficient descent direction, see Lemma~\ref{RieNew-DescentDir}. 
Moreover, the preconditioned conjugate gradient method is terminated if the relative residual is sufficiently small in the sense that $\|r^{(i+1)}\| / \| \grad f(x_k)\|$ is smaller than the forcing term $\phi_k$, see Step~\ref{RieNewtCG:st05} of Algorithm~\ref{RieNew-TrunConjGrad}. This condition not only reduces the computational cost by not requiring solving the Newton equation exactly but also ensures the local superlinear convergence rate of Algorithm~\ref{RieNew-TruncatedNewton}, see Theorem~\ref{RieNew-OrderOfConv}. 
Since the Newton equation is defined on the Euclidean space $\T_{x_k} \mathcal{M}$, Algorithm~\ref{RieNew-TrunConjGrad} is equivalent to the one in~\cite[Minor~Iteration]{dembo_truncated-newton_1983} when the preconditioner $\mathscr{P}(x):\mathrm{T}_x \mathcal{M} \rightarrow \mathrm{T}_x \mathcal{M}$ is an identity operator.
%Note that the sequence $\{\phi_k\}$ is called a forcing sequence, and the output is referred to as a truncated-Newton direction.

Step~\ref{RieNew:st02} is used to find a step size satisfying inequalities~\eqref{RieNew-BrydCond1} and~\eqref{RieNew-BrydCond2}, which is a Riemannian generalization of the line search conditions in~\cite{Byrd1989ATF}. It has been pointed out in~\cite{huang_riemannian_2018} that if the function $f$ is radially $L$-$C^1$ function (see Definition \ref{PrelNotat-RLipConti}), then many commonly-used line search conditions including the Wolfe conditions and Armijo-Glodstein conditions imply one or both of~\eqref{RieNew-BrydCond1} and~\eqref{RieNew-BrydCond2}. Therefore, Algorithm~\ref{RieNew-TruncatedNewton} uses a weaker line search condition than the one in~\cite{dembo_truncated-newton_1983}.

% In Step 4 of Algorithm \ref{RieNew-TruncatedNewton}, the solution $\eta_k$ satisfying (\ref{RieNew-TrunNewEqu}) is given by \sout{the} Algorithm \ref{RieNew-TrunConjGrad}, called truncated conjugate gradient method; see \cite{dembo_truncated-newton_1983}. Conditions (\ref{RieNew-BrydCond1}) and (\ref{RieNew-BrydCond2}) in \sout{step} Step 4 of Algorithm \ref{RieNew-TruncatedNewton} will make the objective function "sufficient descent". Many conditions, e.g., Armijo-Goldstein conditions, Wolfe conditions and so on, imply one or both of (\ref{RieNew-BrydCond1}) and (\ref{RieNew-BrydCond2}), when $\mathcal{M}$ is a Euclidean space and the gradient of the objective function is Lipschitz continuous \cite{richard_h_byrd_tool_1989}, \zwhcomm{which means that the line search condition (\ref{RieNew-BrydCond1}) and (\ref{RieNew-BrydCond2}) are  weaker than conditions mentioned above.}{} Considering that function $f\circ R_x : \mathrm{T}_{x}\mathcal{M} \rightarrow \mathbb{R}$ is defined on a linear space $\mathrm{T}_x\mathcal{M}$, when $\mathcal{M}$ is a manifold, the previous results also hold if \sout{the gradient of the function is Riemannian Lipschitz continuous} \zwhcomm{the function is radially $L$-$C^1$ function}{}(see Definition \ref{PrelNotat-RLipConti})\cite{huang_riemannian_2018}.

% For Equation (\ref{RieNew-TrunNewEqu}), we use truncated conjugate gradient (tCG) method to find $\eta_k$, as seen in Algorithm \ref{RieNew-TrunConjGrad}, where the subscript represents the outer iteration and the superscript means the inner iteration.

\subsection{Convergence Analysis}

In this section, we establish the convergence results by merging the theoretical techniques in~\cite{dembo_truncated-newton_1983} and~\cite{huang_riemannian_2018}. 

% {\color{red}
We firstly suppose the preconditioner $\mathscr{P}$ is self-adjoint positive definite and its maximum and minimum eigenvalues are respectively upper and below bounded by two constants independent of $x$. The formal statement is referred to Assumption~\ref{ass:precond}.
\begin{assumption} \label{ass:precond}
	%For any starting $x_0\in \mathcal{M}$, let $x\in L(x_0)$. 
 The preconditioner $\mathscr{P}(x): \mathrm{T}_x \mathcal{M} \rightarrow \mathrm{T}_x \mathcal{M}$ is self-adjoint positive definite and there exists two constant $p_1$ and $p_2$ such that for all $x \in L(x_0)$,
	\[
		p_1 \le \lambda_{\min}(\mathscr{P}(x)) \le \lambda_{\max}(\mathscr{P}(x)) \le p_2
	\]
	holds, where $\lambda_{\min}(\mathscr{P}(x))$ and $\lambda_{\max}(\mathscr{P}(x))$ denote the minimum and maximum eigenvalues of $\mathscr{P}(x)$. 
\end{assumption}
% }

Theorem~\ref{App-Theorem1} generalizes the properties of the conjugate gradient method in~\cite{Hestenes1952MethodsOC} to a generic Euclidean space while considering preconditioning. 
Essentially, preconditioning is a variable substitution, which is then solved by the ordinary conjugate gradient method. Then the variables are substituted back, and then the preconditioned scheme is obtained. Their theoretical properties are consistent. We refer interested readers to~\cite[P.118]{NW06} for more details.
Such generalizations are trivial and have been used in~\cite{AbsMahSep2008}. These results are given here for completeness and will be used in Lemma~\ref{RieNew-DescentDir}. 

\begin{theorem} \label{App-Theorem1}
%\whcomm{[ZWTODO: Not to use this simplified notation: notation $f_x$, $g_x$, $\mathcal{H}_x$, and $\left<\cdot,\cdot\right>_x$ to denote $f(x)$, $\mathrm{grad}f(x)$, $\mathrm{Hess}f(x)$ and $g_x(\cdot,\cdot)$ respectively for simplicity. ]}{}

%  If $\left<\mathcal{H}_x[d_i],d_i\right>_x>\epsilon \delta_i$, $i=0,1,2,\cdots,k$, then 

  If $g_x(\mathrm{Hess}f(x)[d^{(i)}], d^{(i)})>\epsilon \delta^{(i)}$, $i=0,1,\dots,k$, then
\begin{align}
      & g_x( \mathrm{Hess}f(x)[d^{(i)}],d^{(j)} ) = 0,\quad\quad\quad\quad\;\; i\not=j,\;i,j=0,1,\cdots,k, \label{App-Theorem1-1} \\ 
      & g_x( d^{(i)}, r^{(j)} )=0,\quad\quad\quad\quad\quad\quad\quad\quad\quad i<j,\;i,j=0,1,\cdots,k+1, \label{App-Theorem1-2} \\ 
      & g_x( r^{(i)}, d^{(j)} ) = g_x(r^{(j)}, y^{(j)}) = g_x(r^{(0)}, d^{(j)}),\; i\le j,\; i,j=0,1,\cdots,k, \label{App-Theorem1-3} \\ 
      & \mathrm{span}\{ d^{(0)},d^{(1)},\cdots,d^{(k)} \} = \mathrm{span}\big\{y^{(0)},\mathrm{Hess}f(x)[y^{(0)}],\cdots, (\mathrm{Hess}f(x))^{k-1}[y^{(0)}]\big\}, \label{App-Theorem1-4}\\ 
      & \delta^{(i)} = g_x(d^{(i)},d^{(i)}),\;i=0,1,\cdots,k. \label{App-Theorem1-6}
\end{align}
\end{theorem}

Lemma~\ref{RieNew-DescentDir} proves that the search direction $\eta_k$ from Algorithm~\ref{RieNew-TrunConjGrad} is a descent direction and its norm is not small compared to the norm of $\grad f(x_k)$. It is a Riemannian generalization of~\cite[Lemma~A.2]{dembo_truncated-newton_1983}.
Since $\eta_k$ is a descent direction, there exists a step size satisfying Inequality~\eqref{RieNew-BrydCond1} or~\eqref{RieNew-BrydCond2} as shown in~\cite{NW06, huang_riemannian_2018}. Thus, Algorithm~\ref{RieNew-TruncatedNewton} is well-defined.

%In what follows, we firstly point out that the truncated-Newton direction $\eta_k$ generated by Algorithm \ref{RieNew-TrunConjGrad} is descent as stated in Lemma \ref{RieNew-DescentDir}.
\begin{lemma} \label{RieNew-DescentDir}
There exist two positive constants $\gamma_1$ and $\gamma_2$ that only rest with $\|\mathrm{Hess}f(x_k)\|$, $p_1$, $p_2$, and $\epsilon$ respectively so that 
\begin{equation}
  g_{x_k}(\mathrm{grad}f(x_k),\eta_k) \le -\gamma_1 \|\mathrm{grad}f(x_k)\|^2 \label{RieNew-Lem-1}
\end{equation}
and
\begin{equation}
  \begin{aligned}
     \|\eta_k\|\le \gamma_2 \|\mathrm{grad}f(x_k)\|. \label{RieNew-Lem-2}
  \end{aligned}
\end{equation}  

\begin{proof} %This is similar to the proof of \cite[Theorem A.3]{dembo_truncated-newton_1983}.
  Out of convenience, use $g_k$, $\mathcal{H}_k$ and $\left<\cdot,\cdot\right>_k$ to denote $\mathrm{grad}f(x_k)$, $\mathrm{Hess}f(x_k)$ and $g_{x_k}(\cdot,\cdot)$. Suppose that $\eta_k=\eta^{(i)}$, $i\ge0$. From  Algorithm~\ref{RieNew-TrunConjGrad}, we have 
\begin{align}
	\eta_k = \eta^{(i)} = \sum_{j=0}^{i-1}\alpha^{(j)}d^{(j)}. \label{RieNew-Lem-3}
\end{align}
From Step~\ref{RieNewtCG:st03} in Algorithm~\ref{RieNew-TrunConjGrad} and (\ref{App-Theorem1-3}) we have that 
\begin{align}
  \alpha^{(j)} = \frac{ \left< d^{(j)}, r^{(0)} \right>_k }{ \left< d^{(j)},q^{(j)} \right>_k }. \label{RieNew-Lem-4}
\end{align} 
Hence using \eqref{RieNew-Lem-3} and \eqref{RieNew-Lem-4} results in 

\begin{equation}
  \begin{aligned} \nonumber
  \left<\eta_k,g_k\right>_k &= \left<\eta^{(i)},-r^{(0)}\right>_k=-\left< \sum_{j=0}^{i-1} \frac{ \left<d^{(j)},r^{(0)}\right>_k d^{(j)}}{\left<d^{(j)},q^{(j)}\right>_k},r^{(0)} \right>_k  = -\sum_{j=0}^{i-1}\frac{\left<d^{(j)},r^{(0)}\right>_k^2}{\left<d^{(j)},q^{(j)}\right>_k}  \\ 
  & \le - \frac{\left<d^{(0)},r^{(0)}\right>_k^2}{\left<d^{(0)},q^{(0)}\right>_k} = -\frac{\left<d^{(0)},r^{(0)}\right>_k^2}{\left<d^{(0)},\mathcal{H}_k[d^{(0)}]\right>_k}.
%  =- \frac{\left<d^{(0)},d^{(0)}\right>_k}{\left<d^{(0)},\mathcal{H}_k[d^{(0)}]\right>_k}\|g_k\|^2
  \end{aligned}
\end{equation}
Noting that $r(0)=\mathscr{P}_kd^{(0)}$, and $r^{(0)}=-g_k$, by the Cauchy-Swartz inequality, we have 
\begin{align} \nonumber
\left<d^{(0)},r^{(0)}\right>^2_k \ge \frac{p_1}{p_2}\left<d^{(0)},d^{(0)}\right>_k\|g_k\|^2.
\end{align}
Consequently, it follows that
\[
  \frac{ \left< d^{(0)},d^{(0)} \right>^2_k }{\left<d^{(0)},\mathcal{H}_x[d^{(0)}] \right>_k}\ge \frac{\|g_k\|^2}{\bar{\kappa}{\|\mathcal{H}_k}\|} \ge\frac{1}{\bar{\kappa} \bar{\gamma}_1}\|g_k\|^2, 
\] 
where $\bar{\kappa}:=p_2/p_1$, $\bar{\gamma}_1>0$ is a constant and the existence of $\bar{\gamma_1}$ is guaranteed by the uniformly boundedness of $\mathcal{H}_k$ from Assumptions~\ref{TwiceContinuDiff} and~\ref{BoundedLevel}. Therefore, inequality~\eqref{RieNew-Lem-1} with $\gamma_1=1/(\bar{\kappa}\bar{\gamma}_1)$ holds.
%So picking \sout{$\gamma_1 = \min \left\{ 1, {\|\mathcal{H}_k\|}^{-1} \right\}$} \zwhcomm{$\gamma_1 = \min\left\{1,\sigma^{-1}\right\}$}{}, we obtain the result (\ref{RieNew-Lem-1}). 
Note from (\ref{RieNew-Lem-3}) and (\ref{RieNew-Lem-4}) that 
\[
  \eta^{(i)} = \sum_{j=0}^{i-1}\frac{\left<d^{(j)},r^{(0)}\right>_k}{\left<d^{(j)},q^{(j)}\right>_k}d^{(j)}.
  \]
Hence we have 
\begin{equation}
  \begin{aligned} \nonumber
    \|\eta_k\| = \|\eta^{(i)}\| &= \left\|  \sum_{j=0}^{i-1}\frac{\left<d^{(j)},r^{(0)}\right>_k}{\left<d^{(j)},q^{(j)}\right>_k}d^{(j)} \right\| \le \sum_{j=0}^{i-1} \frac{|\left< d^{(j)},r^{(0)} \right>_k|}{\left<d^{(j)},q^{(j)}\right>_k} \|d^{(j)}\| \\ 
    & \le \sum_{j=0}^{i-1} \frac{ \|d^{(j)}\|\|r^{(0)}\|\|d^{(j)}\| }{\left<d^{(j)},q^{(j)}\right>_k} = \sum_{j=0}^{i-1}\frac{\|d^{(j)}\|^2}{\left<d^{(j)},q^{(j)}\right>_k}\|r^{(0)}\| \\ 
    &= \sum_{j=0}^{i-1} \frac{\left<d^{(j)},d^{(j)}\right>_k}{\left<d^{(j)},q^{(j)}\right>_k}\|r^{(0)}\| \le i \frac{1}{\epsilon} \|r^{(0)}\| \\ 
    &= i \frac{1}{\epsilon} \|g_k\|.
  \end{aligned}
\end{equation}
So for $\gamma_2=\max\{n \epsilon^{-1},1\}$, we have the desired result (\ref{RieNew-Lem-2}), where $n=\dim(\mathrm{T}_{x_k}\mathcal{M}$). 
\end{proof}
\end{lemma}

%In Algorithm \ref{RieNew-TruncatedNewton}, we compute descent direction $\eta_k$ by tCG to solve the truncated-Newton equation (\ref{RieNew-TrunNewEqu}). Theorem \ref{RieNew-GlobalConver} tell us that the Algorithm \ref{RieNew-TruncatedNewton} has global convergence. 

Now, we are ready to give the global convergence result of Algorithm~\ref{RieNew-TruncatedNewton} in Theorem~\ref{RieNew-GlobalConver}. The proofs of Theorem~\ref{RieNew-GlobalConver} rely on the techniques in~\cite{huang_riemannian_2018}, not the ones in~\cite{dembo_truncated-newton_1983}.
\begin{theorem} \label{RieNew-GlobalConver}
  Let $\{x_n\}$ denote the sequence generated by Algorithm \ref{RieNew-TruncatedNewton}. Then it holds that
  \[
    \lim_{n \rightarrow \infty}\|\mathrm{grad}f(x_n)\|=0.
    \]
 If $x^*$ is accumulation point of the sequence $\{x_k\}$ and $\mathrm{Hess}f(x^*)$ is positive definite, then $x_k \rightarrow x^*$.
\end{theorem}
  \begin{proof}
    Let $\{x_n\}$, $\{\eta_n\}$ be generated by Algorithm \ref{RieNew-TruncatedNewton} and Algorithm \ref{RieNew-TrunConjGrad} respectively and step sizes $\{\alpha_n\}$ satisfy either (\ref{RieNew-BrydCond1}) or (\ref{RieNew-BrydCond2}). Noting that (\ref{RieNew-Lem-1}) and (\ref{RieNew-Lem-2}), we have 
    \begin{align*}
      \infty & > f(x_0) - f(x_{n})=\sum\limits_{k=0}^{n-1}(f(x_{k})-f(x_{k+1}))=\sum\limits_{k=0}^{n}(h_k(0)-h_{k}(\alpha_{k})) \\ & \ge \sum\limits_{k=0}^{n-1}\min \left(\chi_{1}\frac{h_{k}'(0)^{2}}{\|\eta_{k}\|^{2}},-\chi_{2}h_{k}'(0) \right) \\ &=\sum\limits_{k=0}^{n-1}\min \left( \chi_{1}\frac{g(\mathrm{grad}f(x_{k}),\eta_{k})^{2}}{\|\eta_{k}^{2}\|}, -\chi_{2}g(\mathrm{grad}f(x_{k}),\eta_{k}) \right) \\ & \ge \sum\limits_{k=0}^{n-1}\min \left( \chi_{1}\gamma_{1}\frac{\|\mathrm{grad}f(x_{k})\|^{2}}{\|\eta_{k}\|^{2}}(-g(\mathrm{grad}f(x_{k}),\eta_{k})), -\chi_{2}g(\mathrm{grad}f(x_{k}),\eta_{k}) \right) \\ & \ge \sum\limits_{k=0}^{n-1}\min \left(\chi_{1}{\gamma_{1}}{\gamma_{2}^{-2}},\chi_{2} \right)(-g(\mathrm{grad}f(x_{k}),\eta_{k})),
    \end{align*}
    which implies 
    \begin{align}
      \lim_{n \rightarrow \infty}(-g(\mathrm{grad}f(x_n),\eta_n))=0. \label{RieNew-Theorem-Conv}
    \end{align}
    Combining (\ref{RieNew-Lem-1}) and (\ref{RieNew-Theorem-Conv}) we have 
    \[
      0\le \gamma_1 \|\mathrm{grad}f(x_n)\|^2\le  (-g(\mathrm{grad}f(x_n),\eta_n)) \rightarrow 0 \;(n \rightarrow \infty).
      \]
          
  For the second part, noting that $\|\mathrm{grad}f(x_k)\|\rightarrow 0$ and continuity of $\mathrm{grad}f$, any accumulation point $\tilde{x}$ of the sequence $\{x_k\}$ is a stationary point of $f$. Under Assumption \ref{BoundedLevel}, each subsequence of $\{x_k\}$ converges to a stationary point. Therefore, $\{x_k\}$ converges to $x^*$ since $\mathrm{Hess}f(x^*)$ is positive definite.
  \end{proof}

The local convergence rate of Algorithm~\ref{RieNew-TruncatedNewton} is established in Theorem~\ref{RieNew-OrderOfConv}. The convergence rates are generalizations of the results in~\cite{dembo_truncated-newton_1983}. However, the proofs are not simply generalizations of~\cite{dembo_truncated-newton_1983}. That the step size one eventually satisfies the line search condition is guaranteed by the Riemannian Dennis-Mor{\'e} condition in~\cite{ring_optimization_2012}. The analysis of the order of convergence requires the relationship between multiple definitions of Riemannian gradients and Riemannian Hessians, i.e., $\grad f(R_x(\eta_x))$ versus $\grad (f \circ R_x) (\eta_x)$ and $\Hess f(x_k)$ versus $\Hess (f \circ R_x)(0_x)$.
\begin{theorem}
	\label{RieNew-OrderOfConv} 
  Let $\{x_k\}$ be the sequence generated by Algorithm \ref{RieNew-TruncatedNewton} with forcing sequence 
  $ 
  \phi_k = \min\left\{\beta,\|\mathrm{grad}f(x_k)\|^t\right\}, 0 < \beta,t \le 1 .
  $
  Suppose that $\{x_k\}$ converges to $x^*$ at which $\mathrm{Hess}f(x^*)$ is positive definite. Then 
  \begin{enumerate}
    \item[1.]  the stepsize $\alpha_k=1$ is acceptable for sufficiently large $k$; and 
    \item[2.] the convergence rate is superlinear\footnote{A sequence $\{x_k\}\subset \mathcal{M}$ converging to $x^*\in\mathcal{M}$ is superlinear if $ \frac{\mathrm{dist}(x^*,x_{k+1})}{\mathrm{dist}(x^*,x_k)} \rightarrow 0(k \rightarrow \infty). $ }.
  \end{enumerate}
Moreover, suppose that $\mathrm{Hess}\hat{f}(0_{x_k})$ is sufficiently close to $\mathrm{Hess}f(x_k)$, i.e., it holds that $\|\mathrm{Hess}f(x_k)- \mathrm{Hess}\hat{f}_{x_k}(0_{x_k})\|\le \beta_1 \|\mathrm{grad}f(x_k)\|$, $\hat{f}$ as defined in Definition \ref{PrelNotat-Pullback}, with a positive constant $\beta_1$, and that $\mathrm{Hess}\hat{f}_x$ is Lipschitz-continous at $0_x$ uniformly in $x$ in a neighborhood of $x^*$, i.e., there exist $\beta_2>0,\mu_1 > 0$ and $\mu_2>0$ such that for all $x\in B_{\mu_{1}}(x^*)$ and all $\eta_x\in B_{\mu_2}(0_x)$, it holds that $\|\mathrm{Hess}\hat{f}_x(\eta_x)-\mathrm{Hess}\hat{f}_x(0_x)\|\le \beta_2\|\eta_x\|$. Then, 
\begin{enumerate}
	\item[3.] the convergence rate is $1+\min(1,t)$\footnote{A sequence $\{x_k\}\subset \mathcal{M}$ converging to $x^*\in\mathcal{M}$ has Q-order of $1+\min(1,t)$ if $\limsup_{k \rightarrow \infty} \frac{\mathrm{dist}(x^*,x_{k+1})}{\mathrm{dist}(x^*,x_k)^{1+\min(1,t)}} <\infty. $}.
\end{enumerate}
\end{theorem}

It is worth mentioning that in Theorem \ref{RieNew-OrderOfConv} the assumption ``$\|\mathrm{Hess}f(x_k) - \mathrm{Hess}\hat{f}_{x_k}(0_{x_k})\|\le \beta_1\|\mathrm{grad}f(x_k)\|$''~\cite[Theorem~7.4.10]{AbsMahSep2008} is reasonable since that for sufficiently large $k$, $x_k$ is sufficiently close to $x^*$ and the right-hand side can take zero at the stationary point $x^*$. On the other hand, if a second-order retraction is used, this assumption naturally holds since the right-hand side of the inequality can take up to zero~\cite[Proposition~5.5.5]{AbsMahSep2008}. The another assumption ``$\mathrm{Hess}\hat{f}_x$ is Lipschitz-continous at $0_x$ uniformly in $x$ in a neighborhood of $x^*$''~\cite[Theorem~7.4.10]{AbsMahSep2008} is the counterpart in the classical Newton-tCG method~\cite{DES82}.

\begin{proof}
	Since $\mathrm{Hess}f(x^*)$ is positive definite, there exists a neighborhood $\mathcal{U}_{x^*}$ of $x^*$ in $\mathcal{M}$ such that $\mathrm{Hess}f(x)$ is also positive definite for all $x\in \mathcal{U}_{x^*}$. From Theorem~\ref{RieNew-GlobalConver} it follows that there exists a positive constant integer $k_0$ such that $\mathrm{Hess}f(x_k)$ is positive definite for $k\ge k_0$. 
	By Algorithm \ref{RieNew-TrunConjGrad}, we have 
    \begin{align}
      \|r_k\|=\|\mathrm{Hess}f(x_k)[\eta_k]+\mathrm{grad}f(x_k)\|\le \|\mathrm{grad}f(x_k)\|\phi_k\le\|\mathrm{grad}f(x_k)\|^{1+t}. \label{RieNew-OrderOfConv-1}
    \end{align}
    By the trigonometric inequality of norm and (\ref{RieNew-OrderOfConv-1}), we have  
    \begin{align*}
      \|\mathrm{grad}f(x_k)\|-\|\mathrm{Hess}f(x_k)[\eta_k]\|\le\|\mathrm{grad}f(x_k)\|^{1+t}.
    \end{align*}
    It continues 
    \[
      (1-\|\mathrm{grad}f(x_k)\|^t)\|\mathrm{grad}f(x_k)\|\le\|\mathrm{Hess}f(x_k)[\eta_k]\|.
      \]
    By Assumption~\ref{TwiceContinuDiff}, there exist $L > 0,\tilde{L}>0$ such that $\|\mathrm{Hess}f(x_k)\|\le L$ and $\|\mathrm{Hess}f(x_k)^{-1}\|\le \tilde{L}$ for sufficiently large $k$ since $\mathrm{Hess}f(x^*)$ is positive definite. it follows that
    \[
      (1-\|\mathrm{grad}f(x_k)\|^t)\|\mathrm{grad}f(x_k)\|\le L\|\eta_k\|.
      \]
    By Theorem \ref{RieNew-GlobalConver}, for sufficiently large $k$, we have 
    $
      \frac{1}{2}\|\mathrm{grad}f(x_k)\|\le L\|\eta_k\|,
	$
    that is, 
    \begin{align}
      \|\mathrm{grad}f(x_k)\|\le 2L\|\eta_k\|. \label{RieNew-OrderOfConv-2}
    \end{align}
    Besides, noting that $\|r_k\|\le \|\mathrm{grad}f(x_k)\|$, we have 
    \begin{equation}    	
    \begin{aligned}
      \|\eta_k\| &= \|\mathrm{Hess}f(x_k)^{-1}[r_k-\mathrm{grad}f(x_k)]\|\le\tilde{L}(\|r_k\|+\|\mathrm{grad}f(x_k)\|)\\
      & \le 2\tilde{L}\|\mathrm{grad}f(x_k)\|.  \label{RieNew-OrderOfConv-3}
    \end{aligned}
    \end{equation}
    Thus, we have 
  	\begin{align*}  		
  		\frac{ \|\mathrm{grad}f(x_k)+ \mathrm{Hess}f(x_k)\eta_k \| }{\|\eta_k\|} &\le \frac{\|\mathrm{grad}f(x_k)\|^{1+t}}{\|\eta_k\|}\le {(2L)^{1+t}\|\eta_k\|^t} \\
  		& \le (2L)^{1+t}\gamma_2^t\|\mathrm{grad}f(x_k)\|^t,
  	\end{align*}
  	which gives that 
  	\begin{align}
  		\label{RieNew-OrderOfConv-4}
  		\lim_{k\rightarrow \infty}\frac{ \|\mathrm{grad}f(x_k)+ \mathrm{Hess}f(x_k)\eta_k \| }{\|\eta_k\|}=0.
  	\end{align}
  	By \cite[Proposition 5]{ring_optimization_2012}, (\ref{RieNew-OrderOfConv-4}) implies that for sufficiently large $k$, the step size $\alpha_k=1$ is acceptable for Wolfe conditions and thus is too for (\ref{RieNew-BrydCond1}) or (\ref{RieNew-BrydCond2}) under Assumptions~\ref{TwiceContinuDiff} and~\ref{RadLipCont}.   
  	In addition, the superlinear convergence rate is obtained by \cite[Proposition 8]{ring_optimization_2012}.
  	
%  	We next prove the second part. 
  	By \cite[Lemma 7.4.8]{AbsMahSep2008} and \cite[Lemma 7.4.9]{AbsMahSep2008}, there exist constants $c_0>0,c_1>0$ and $c_2>0$ such that 
  	\begin{align}
  		c_0\mathrm{dist}(x,x^*) \le \|\mathrm{grad}f(x)\| \le c_1 \mathrm{dist}(x,x^*),\;\forall\;x\in B_{\mu_1}(x^*), \label{RieNew-OrderOfConv-5}
  	\end{align}  
  	and 
  	\begin{align} \label{RieNew-OrderOfConv-6}
  		\|\mathrm{grad}f(R_x(\eta_x))\| \le c_2 \|\mathrm{grad}\hat{f}_x(\eta_x)\|,\;\forall\;x\in B_{\mu_1}(x^*),\;\forall\;\eta_{x}\in B_{\mu_2}(0_{x}).
  	\end{align}
  	From the Taylor's formula for $\hat{f}_{x_k}$ at $0_{x_k}$, for sufficiently large $k$, we have 
  	 \begin{align*} 
  	 \mathrm{grad}\hat{f}_{x_{k}}(\eta_{k})&= \mathrm{grad}\hat{f}_{x_{k}}(0_{x_{k}}) + \mathrm{Hess}\hat{f}_{x_{k}}(0_{x_{k}})[\eta_{k}] + O(\|\eta_{k}\|^{2}) \\ 
  	 	&= \mathrm{grad}f(x_{k}) + \mathrm{Hess}f(x_{k})[\eta_{x}] + (\mathrm{Hess}\hat{f}_{x_{k}}(0_{x_{k}})-\mathrm{Hess}f(x_{k}))[\eta_{k}] \\ 
  	 	& \quad + O(\|\eta_{k}\|^{2})  \\ 
  	 	&=  r_{k} + (\mathrm{Hess}\hat{f}_{x_{k}}(0_{x_{k}})-\mathrm{Hess}f(x_{k}))[\eta_{k}] + O(\|\eta_{k}\|^{2}).
  	 \end{align*} 
  	 It, together with (\ref{RieNew-OrderOfConv-1}), (\ref{RieNew-OrderOfConv-3}), (\ref{RieNew-OrderOfConv-5}) and (\ref{RieNew-OrderOfConv-6}), implies that 
  	  \begin{align*} 
  	  \|\mathrm{grad}f(x_{k+1})\| &\le c_{2}\|\mathrm{grad}\hat{f}_{x_{k}}(\eta_{k})\| \\ & \le c_{2}\|r_{k}\| + c_{2}\|\mathrm{Hess}\hat{f}_{x_{k}}(0_{x_{k}})-\mathrm{Hess}f(x_{k})\|\|\eta_{k}\| + c_{2}c_{3}\|\eta_{k}\|^{2} \\ & \le c_{2}\|r_{k}\| + c_{2}(\beta_{1}+c_{3})\|\eta_{k}\|^{2} \\ & \le c_{2}(\|\mathrm{grad}f(x_{k})\|^{t} + 4\tilde{L}^{2}(\beta_{1}+c_{3})\|\mathrm{grad}f(x_{k})\|)\|\mathrm{grad}f(x_{k})\| ,
  	  \end{align*} 
  	  where $c_3>0$ is a constant. Hence, we obtain that 
  	  \begin{equation} \nonumber 	  	
  	   \begin{aligned}   	   
  	   \frac{\mathrm{dist}(x_{k+1},x^{*})}{\mathrm{dist}(x_{k},x^{*})^{1+\min(1,t)}} & \le \frac{c_{1}^{1+\min(1,t)}}{c_{0}}\frac{\|\mathrm{grad}f(x_{k+1})\|}{\|\mathrm{grad}f(x_{k})\|^{1+\min(1,t)}} \\ & \le \frac{c_{2}c_{1}^{1+\min(1,t)}}{c_{0}}\bigg( \|\mathrm{grad}f(x_{k})\|^{t-\min(1,t)} \\ 
  	     & \quad + 4\tilde{L}^{2}(\beta_{1}+c_{3})\|\mathrm{grad}f(x_{k})\|^{1-\min(1,t)} \bigg)  \\ & \le C, 
  	   \end{aligned} 
  	  \end{equation}
  	  for a constant $C>0$, which completes the proof. %For sufficiently large $k$, the inequality above is equivalent to 
%  	  $ 
%  	  \|R_{x^{*}}^{-1}(x_{k+1})\|\le \tilde{C} \|R_{x^{*}}^{-1}(x_{k})\| 
%  	  $
%  	  with a positive constant $\tilde{C}$. We complete the proofs. 

\end{proof}

\section{Ingredients on Riemannian Quotient Manifold \texorpdfstring{$\mathbb{R}_*^{n\times p}/\mathcal{O}_p$}{} } \label{QuotMani}

In this section, the optimization tools of $\mathbb{R}_*^{n \times p}/\mathcal{O}_p$ are reviewed, including tangent space, horizontal space, vertical space, Riemannian metric, retraction, Riemannian gradient, and Riemannian Hessian. The detailed derivations can be found in~\cite{Zheng2022RiemannianOU}.

%The matrices set $\mathbb{R}_*^{n\times p}$ is a open submanifold of $\mathbb{R}^{n\times p}$, where $p\ll n$,  and the quotient space $\mathbb{R}_*^{n\times p}/\mathcal{O}_p$ is a quotient manifold of $\mathbb{R}_*^{n\times p}$ \cite{AbsMahSep2008,MFPR10}, therefore the relationship between them is very close. In fact, for any $Y \in \mathbb{R}_*^{n\times p}$, there exists a bijective $\mathrm{D}\pi(Y)$ from a subspace, so-called horizontal space, of $\mathrm{T}_Y\mathbb{R}_*^{n\times p}$ to the tangent space $\mathrm{T}_{\pi(Y)}\mathbb{R}_*^{n\times p}/\mathcal{O}_p$. Then for each tangent vector $\eta_{\pi(Y)}$, we can find the unique vector corresponding to it in the horizontal space. This makes implementation of algorithms on quotient manifold possible. 

%{\color{red}

%\subsection{Riemannian manifold \texorpdfstring{$\mathbb{R}_*^{n\times p}$}{}}

%The ingredients, such as Riemannian metric and retraction, on the total space $\mathbb{R}_*^{n\times p}$ can become the ones in the quotient manifold $\mathbb{R}_*^{n\times p}/\mathcal{O}_p$ under certain conditions, so we first give the related notions on $\mathbb{R}_*^{n\times p}$.

\paragraph{Vertical spaces.}
For any $Y \in \mathbb{R}_*^{n \times p}$, it has been shown in~~\cite[Proposition 3.4.4]{AbsMahSep2008} that the equivalence class $\pi^{-1}(\pi(Y))$ is an embedded submanifold $\mathbb{R}_*^{n \times p}$. The tangent space of $\pi^{-1}(\pi(Y))$ at $Y$ is called the vertical space of $\mathbb{R}_*^{n \times p}/\mathcal{O}_p$ at $Y$, denoted by $\mathrm{V}_Y$, and is given by 
\[
  \mathrm{V}_Y=\mathrm{T}_Y \pi^{-1}(\pi(Y))=\{Y\Omega: \Omega  \in \mathrm{Skew}(p)\} \subseteq \mathrm{T}_Y\mathbb{R}_*^{n\times p},
\]
  where %$\mathcal{F}=\{X\in \mathbb{R}_*^{n\times p}: X \sim Y\}$ is the fiber of $Y$ and 
  $\mathrm{Skew}(p) = \{A\in\mathbb{R}^{p\times p}: A^T=-A\}$. 
 
\paragraph{Riemannian metric on $\mathbb{R}_*^{n\times p}$.} The Riemannian metric on $\mathbb{R}_*^{n\times p}$ given in~\cite{Zheng2022RiemannianOU} is considered: 
\begin{equation} \label{Rmetric1}
g_Y(\eta_Y,\xi_Y) = 
	2\mathrm{tr}(Y^T\eta_YY^T\xi_Y+Y^TY\eta_Y^T\xi_Y) 
        +\mathrm{tr}\left(Y^TY(\eta_Y^{\mathrm{V}})^T(\xi_Y^{\mathrm{V}})\right), 
\end{equation}
where $\eta_Y,\xi_Y\in \mathrm{T}_Y\mathbb{R}_*^{n\times p} = \mathbb{R}^{n \times p}$, $\eta_Y^{\mathrm{V}}= Y \left( (Y^TY)^{-1}Y^T \eta_Y-\eta_Y^T Y(Y^TY)^{-1} \right) / 2$, $\xi_Y^{\mathrm{V}} = Y ( (Y^TY)^{-1}Y^T \xi_Y-$ $\xi_Y^T Y(Y^TY)^{-1} ) / 2$.
This Riemannian metric is essentially equivalent to the Euclidean metric on $\mathcal{S}_+(p,n)$, see details in~\cite{Zheng2022RiemannianOU}. We also take two other Riemannian metrics on $\mathbb{R}^{n\times p}_*$ into consideration, and the detialed discussion is refered to Appendix~\ref{Appendix}.

\paragraph{Horizontal spaces.} 
Given a Riemannian metric $g$ on the total space $\mathbb{R}_*^{n \times p}$, the orthogonal complement space in $\mathrm{T}_Y\mathbb{R}_*^{n\times p}$ of $\mathrm{V}_Y$ with respect to $g_Y$ is called the horizontal space at $Y$, denoted by $\mathrm{H}_Y$. The horizontal spaces with respect to the Riemannian metric in~\eqref{Rmetric1} are respectively given by
\begin{equation} \nonumber
	\mathrm{H}_Y = 
		\{YS+Y_\perp K:S^T=S,S\in\mathrm{R}^{p\times p},K\in\mathbb{R}^{(n-p)\times p}\} 
\end{equation}
By~\cite[Section~3.5.8]{AbsMahSep2008}, the mapping $\mathrm{D}\pi(Y)$ is a bijection from $\mathrm{H}_Y$ to $\mathrm{T}_{\pi(Y)}\mathbb{R}_*^{n\times p}/\mathcal{O}_p$, so for each $\eta_{\pi(Y)}\in \mathrm{T}_{\pi(Y)}\mathbb{R}_*^{n\times p}/\mathcal{O}_p$, there exists unique vector $\eta_{Y}\in \mathrm{H}_Y$ such that $\mathrm{D}\pi(Y)[\eta_{Y}]=\eta_{\pi(Y)}$. This $\eta_{Y}$ is called the horizontal lift of $\eta_{\pi(Y)}$ at $Y$, denoted by $\eta_{\uparrow_Y}$.

\paragraph{Projections onto Vertical Spaces and Horizontal Spaces.}
For any $Y\in \mathbb{R}_*^{n\times p}$ and $\eta_Y\in \mathrm{T}_Y\mathbb{R}_*^{n\times p}$, the orthogonal projections of $\eta_Y$ to $\mathrm{V}_Y$ and $\mathrm{H}_Y$ with respect to the Riemannian metric are given by
	$$
	\begin{aligned} 
	\mathcal{P}_Y^{\mathrm{V}}(\eta_Y) = Y\Omega, \text{ and } \mathcal{P}_Y^{\mathrm{H}}(\eta_Y)&=\eta_Y-\mathcal{P}_Y^{\mathrm{V}}(\eta_Y)=\eta_Y-Y\Omega \\ &= Y\left( \frac{(Y^TY)^{-1}Y^T\eta_Y+\eta_Y^TY(Y^TY)^{-1}}{2} \right) + Y_\perp Y_\perp^T\eta_Y, 
    \end{aligned}
	$$
	where $\Omega = \frac{(Y^TY)^{-1}Y^T\eta_Y-\eta_Y^TY(Y^TY)^{-1}}{2}$.

\paragraph{Riemannian metric on $\mathbb{R}_*^{n\times p}/\mathcal{O}_p$.} If for all $\xi_{\pi(Y)},\eta_{\pi(Y)} \in\mathrm{T}_{\pi(Y)}\mathbb{R}_*^{n\times p}/\mathcal{O}_p$, it holds that 
\begin{align} \label{IngredQuotMani-Metric-induced}
	Y\sim Z \Rightarrow g_Y(\xi_{\uparrow_Y},\eta_{\uparrow_Y})  = g_Z(\xi_{\uparrow_Z},\eta_{\uparrow_Z}), 
\end{align}
the Riemannian metric on $\mathbb{R}_*^{n\times p}/\mathcal{O}_p$ induced by $g$ is defined as 
$
	g_{\pi(Y)}(\xi_{\pi(Y)},\eta_{\pi(Y)}) = g_Y(\xi_{\uparrow_Y},\eta_{\uparrow_Y}).
$
Since the Riemannian metric $g_Y$ on $\mathbb{R}_*^{n\times p}$ satisfy~\eqref{IngredQuotMani-Metric-induced} by~\cite{Zheng2022RiemannianOU}, the induced Riemannian metric on $\mathbb{R}_*^{n\times p}/\mathcal{O}_p$ is given by:  
\begin{align} \label{IngredQuotMani-Metric1}
    g_{[Y]}(\xi_{[Y]},\eta_{[Y]})  = 2\mathrm{tr}(Y^T\xi_{\uparrow_Y}Y^T\eta_{\uparrow_Y}+Y^TY\xi_{\uparrow_Y}^T\eta_{\uparrow_Y}), 
\end{align}
for all $\xi_{\pi(Y)},\eta_{\pi(Y)}\in\mathrm{T}_{\pi(Y)}\mathbb{R}_*^{n\times p}/\mathcal{O}_p$.
In the following, with a slight abuse of notation, we use $g$ to denote the Riemannian metric on both $\mathbb{R}_*^{n\times p}$ and $\mathbb{R}_*^{n\times p}/\mathcal{O}_p$. 

\paragraph{Retraction.} We choose the retraction on $\mathbb{R}_*^{n\times p}/\mathcal{O}_p$ defined by
\begin{equation} \label{retraction}
   		R_{\pi(Y)}(\xi_{\pi(Y)}) = \pi(Y+ \xi_{\uparrow_Y}),
\end{equation}
where $Y \in \mathbb{R}_*^{n \times p}$ and $\xi_{\pi(Y)} \in \T_{\pi(Y)} \mathbb{R}_*^{n\times p}/\mathcal{O}_p$.  Retraction~\eqref{retraction} has been used in~\cite{MA20,Zheng2022RiemannianOU}.
%For any $\eta_Y \in \mathrm{T}_Y\mathbb{R}_*^{n\times p}$ and appropriate step size $\alpha>0$, $\bar{R}_Y(\alpha\eta_Y)=Y + \alpha\eta_Y$ is a retraction on $\mathbb{R}_*^{n\times p}$. For any vector field $\xi $ on $\mathbb{R}_*^{n\times p}/\mathcal{O}_p$, $Y \in \mathbb{R}_*^{n\times p}$ and $O\in\mathcal{O}_p$, it holds that $\xi_{\uparrow_{YO}}=\xi_{\uparrow_Y}O$ \cite[Prop.A.8]{Absil20}, which implies that 
%\[
%   		R_{\pi(Y)}(\alpha\xi_{\pi(Y)}) = \pi(Y+ \alpha \xi_{\uparrow_Y})
%\]
%defines a retraction on $\mathbb{R}_*^{n\times p}/\mathcal{O}_p$ with appropriate step size $\alpha>0$.

\paragraph{The horizontal lifts of Riemannian Gradients and the actions of Riemannian Hessians.} 
It follows from~\cite{Zheng2022RiemannianOU} that the Riemannian gradients and Riemannian Hessian of $f$ in~\eqref{Pro_stat-FinalProb} can be characterized by the Euclidean gradient and the Euclidean Hessian of $h$ in~\eqref{Pro_stat-OptProb_Bart}, see Proposition~\ref{IngredQuotMani-Prop-Gradient}.
\begin{proposition} \label{IngredQuotMani-Prop-Gradient}
  The Riemannian gradients of the smooth real-valued function $f$ in \eqref{Pro_stat-FinalProb} on $\mathbb{R}_*^{n\times p}/\mathcal{O}_p$ with respect to the Riemannian metric~\eqref{Rmetric1} is given by 
  \[
  (\mathrm{grad}f(\pi(Y)))_{\uparrow_Y} = \mathrm{grad}\bar{f}(Y) =
    \left( I - \frac{1}{2}P_Y \right)\nabla h(YY^T)Y(Y^TY)^{-1},
  \]
and the action of the Riemannian Hessian is given by
\begin{align} \label{IngredQuotMani-Hessian1}
(\mathrm{Hess}f(\pi(Y))[\eta_{\pi(Y)}])_{\uparrow_Y}= \left(I - \frac{1}{2}P_Y\right) \nabla^2h(YY^T)[Y\eta_{\uparrow_Y}^T+\eta_{\uparrow_Y}Y^T]Y(Y^TY)^{-1} + T_1,
\end{align}
where $\nabla^2h(YY^T)[V] = AVM + MVA$, $P_Y=Y(Y^TY)^{-1}Y^T$, $T_1 = (I-P_Y)\nabla h(YY^T)(I-P_Y)\eta_{\uparrow_Y}(Y^TY)^{-1}$.
\end{proposition}

\section{Preconditioning and Increasing Rank Algorithm} \label{FinalAlgPrecond}

\subsection{Preconditioning} \label{sec:IncRankAlgPrecond:Precond}

%For Newton's method, e.g., Algorithm \ref{RieNew-TruncatedNewton}, most of the cost is concentrated on solving Newton equations. In large-scale setting, it can be solve with iterative methods, say Algorithm \ref{RieNew-TrunConjGrad}. Using iterative methods, we will naturally find ways to reduce the number of iterations in inner iterations. The reason why we choose CG method to solve Newton equation is that CG method is especially suitable for preconditioning. A good preconditioner will significantly reduce the number of iterations for solving Newton equations. We, in this section, will work to derive such preconditioners.

When the condition number of the Riemannian Hessian of $f$ at the minimizer $x^*$ is large and the sequence $\{x_k\}$ generated by Algorithm~\ref{RieNew-TruncatedNewton} converges to $x^*$, the number of iterations in the conjugate gradient method (Algorithm~\ref{RieNew-TrunConjGrad}) can be large. To improve the efficiency, the preconditioned conjugate gradient method given in~\cite{NW06} is used. In this section, we derive preconditioners that approximate the inverse of the Riemannian Hessian with respect to Riemannian metric~\eqref{IngredQuotMani-Metric1}. 

%\subsubsection{Preconditioning under Riemannian Metric~\eqref{IngredQuotMani-Metric1}} 

The Newton direction is given by solving the Newton equation, i.e.,
\begin{gather*}
\hbox{ find } \xi_{\pi(Y)} \in \T_{\pi(Y)} \mathbb{R}_*^{n \times p} / \mathcal{O}_p \\
\hbox{ such that } \Hess f( \pi(Y) ) [\xi_{\pi(Y)}] = - \grad f(\pi(Y)).
\end{gather*}
Therefore, when the Riemannian metric~\eqref{IngredQuotMani-Metric1} is used, Algorithm~\ref{RieNew-TruncatedNewton} needs to approximately solve 
\begin{equation} \label{e01}
  \left( I-\frac{1}{2}P_Y\right)\nabla^2h(YY^T)[Y\xi_{\uparrow_Y}^T+\xi_{\uparrow_Y}Y^T]Y(Y^TY)^{-1} + (I-P_Y)\nabla h(YY^T)(I-P_Y)\xi_{\uparrow_Y}(Y^TY)^{-1} =\eta_{\uparrow_Y},
\end{equation}
for $\xi_{\uparrow_Y}\in\mathrm{H}_Y$, where $\eta_{\uparrow_Y}\in \mathrm{H}_Y$ denotes the horizontal lift of $-\grad f(\pi(Y))$.

%The preconditioner that we proposed aims to solve
%\begin{align} \label{Precond-Metric1-ApproxNewEqu}
%  \left( I - \frac{1}{2}P_Y \right)\nabla^2h(YY^T)[Y\xi_{\uparrow_Y}^T+\xi_{\uparrow_Y}Y^T]Y(Y^TY)^{-1}=\eta_{\uparrow_Y},
%\end{align}
%by omitting the second term in~\eqref{e01}, 
The preconditioner that we proposed is given by
\begin{align*}
	\mathscr{P}(\pi(Y))  : \;& \mathrm{T}_{\pi(Y)}(\mathbb{R}_*^{n\times p}/\mathcal{O}_p) \rightarrow \mathrm{T}_{\pi(Y)}(\mathbb{R}_*^{n\times p}/\mathcal{O}_p)\\ 
	& \eta_{\pi(Y)} \mapsto \mathscr{P}(\pi(Y))[\eta_{\pi}(Y)] = \xi_{\pi(Y)},
\end{align*}
where $\xi_{\pi(Y)}$ and $\eta_{\pi(Y)}$ satisfies the following
\begin{align} \label{Precond-Metric1-ApproxNewEqu}
  \left( I - \frac{1}{2}P_Y \right)\nabla^2h(YY^T)[Y\xi_{\uparrow_Y}^T+\xi_{\uparrow_Y}Y^T]Y(Y^TY)^{-1}=\eta_{\uparrow_Y}.
\end{align}
%\[
%(\mathscr{P}(\pi(Y))[\xi_{\pi(Y)}])_{\uparrow_Y} = \left( I - \frac{1}{2}P_Y \right)\nabla^2h(YY^T)[Y\xi_{\uparrow_Y}^T+\xi_{\uparrow_Y}Y^T]Y(Y^TY)^{-1}.
%\]
Therefore, applying the preconditioner aims to solve~\eqref{Precond-Metric1-ApproxNewEqu}
by omitting the second term in~\eqref{e01}.
Such an approximation is reasonable since (i) the second term is approximately zero if $Y_k Y_k^T \approx X^*$ and (ii) our numerical experiments show that this preconditioner effectively reduces the number of inner iterations.

\paragraph{Positive definiteness of the preconditioner.} 
We claim that the proposed preconditioner satisfies Assumption~\ref{ass:precond}, which is formally stated in Theorem~\ref{thm:PosDefPre}. 
\begin{theorem} \label{thm:PosDefPre}
	Assume $A$ and $M$ are positive definite and let $\mathcal{L}=A\otimes M+A\otimes M$. Then for any $\pi(Y)$, the proposed preconditioner $\mathscr{P}(\pi(Y))$ is self-adjoint positive definite and it satisfies Assumption~\ref{ass:precond} with parameters $p_1=1/\lambda_{\max}(\mathcal{L})$ and $p_2=1/\lambda_{\min}(\mathcal{L})$. 
\end{theorem}
\begin{proof}
In what follows, 
we focus on the self-adjoint positive definiteness of $\mathscr{P}(\pi(Y))^{-1}$, since the self-adjoint positive definiteness of $\mathscr{P}(\pi(Y))$ is the same as that of $\mathscr{P}(\pi(Y))^{-1}$. 
First consider the self-adjointness. From~\eqref{Precond-Metric1-ApproxNewEqu}, we have 
\[
	(\mathscr{P}(\pi(Y))^{-1}[\xi_{\pi(Y)}])_{\uparrow_Y} =  \left( I - \frac{1}{2}P_Y \right)\nabla^2h(YY^T)[Y\xi_{\uparrow_Y}^T+\xi_{\uparrow_Y}Y^T]Y(Y^TY)^{-1},
\]
where we assume the invertibility of $\mathscr{P}$ and its rigorous proof is shown later.
For any $\xi_{\pi(Y)}$ and $\eta_{\pi(Y)}\in \mathrm{T}_{\pi(Y)} \mathbb{R}_*^{n\times p}/\mathcal{O}_p$ whose corresponding horizontal lifts are denoted by $\xi_{\uparrow_Y}$ and $\eta_{\uparrow_Y} \in \mathrm{H}_Y$, we have 
\begin{align*}
&g_{\pi(Y)} (\mathscr{P}(\pi(Y))^{-1}[{\xi}_{\pi(Y)}], {\eta}_{\pi(Y)}) = g_Y((\mathscr{P}(\pi(Y))^{-1}[{\xi}_{\pi(Y)}])_{\uparrow_Y},{\eta}_{\uparrow_Y}) \\
	&= 2\mathrm{tr}(Y^T(I-\frac{1}{2}P_Y)\nabla^2h(YY^T)[Y\xi_{\uparrow_Y}^T+\xi_{\uparrow_Y}Y^T]Y(Y^TY)^{-1}Y^T\eta_{\uparrow_Y}) \\ 
	& \quad + 2\mathrm{tr}(Y^TY((I-\frac{1}{2}P_Y)\nabla^2h(YY^T)[Y\xi_{\uparrow_Y}^T+\xi_{\uparrow_Y}Y^T]Y(Y^TY)^{-1})^T\eta_{\uparrow_Y}) \\ 
	&= 2\mathrm{tr}(\frac{1}{2}Y^T\nabla^2h(YY^T)[Y\xi_{\uparrow_Y}^T+\xi_{\uparrow_Y}Y^T]P_Y\eta_{\uparrow_Y})  \\ 
	& \quad + 2\mathrm{tr}(Y^T\nabla^2h(YY^T)[Y\xi_{\uparrow_Y}^T+\xi_{\uparrow_Y}Y^T](I-\frac{1}{2}P_Y)\eta_{\uparrow_Y}) \\ 
	&= 2\mathrm{tr}(Y^T\nabla^2h(YY^T)[Y\xi_{\uparrow_Y}^T+\xi_{\uparrow_Y}Y^T]\eta_{\uparrow_Y}) \\ 
	&= 2\mathrm{tr}(Y^T\nabla^2h(YY^T)[Y\eta_{\uparrow_Y}^T+\eta_{\uparrow_Y}Y^T]\xi_{\uparrow_Y}) \\ 
	&= g_Y(\xi_{\uparrow_Y},((\mathscr{P}(\pi(Y))^{-1}[{\eta}_{\pi(Y)}])_{\uparrow_Y})) \\ 
	&= g_{\pi(Y)}(\xi_{\pi(Y)}, \mathscr{P}(\pi(Y))^{-1}[{\eta}_{\pi(Y)}]),
\end{align*}
where the third to the last equation is due to that
\begin{align*}
	&\mathrm{tr}(Y^T\nabla^2h(YY^T)[Y\xi_{\uparrow_Y}^T+\xi_{\uparrow_Y}Y^T]\eta_{\uparrow_Y}) \\
	&= \mathrm{tr}(\eta_{\uparrow_Y}Y^{T}AY\xi_{\uparrow_Y}^{T} M)+\mathrm{tr}(\eta_{\uparrow_Y}Y^{T}MY\xi_{\uparrow_Y}^{T}A) + \mathrm{tr}(\eta_{\uparrow_Y}Y^{T}A\xi_{\uparrow_Y} Y^{T}M ) + \mathrm{tr}(\eta_{\uparrow_Y}Y^{T}M\xi_{\uparrow_Y} Y^{T}A) \\ 
	&= \mathrm{tr}(Y\xi_{\uparrow_Y}^{T} M\eta_{\uparrow_Y}Y^{T}A)+\mathrm{tr}(Y\xi_{\uparrow_Y}^{T}A\eta_{\uparrow_Y}Y^{T}M) + \mathrm{tr}(\xi_{\uparrow_Y}Y^TM\eta_{\uparrow_Y}Y^TA) + \mathrm{tr}(\xi_{\uparrow_Y} Y^{T}A\eta_{\uparrow_Y}Y^{T}M)\\
	&= \mathrm{tr}(Y\xi_{\uparrow_Y}\nabla^2h(YY^T)[\eta_{\uparrow_Y}Y^T]) + \mathrm{tr}(\xi_{\uparrow_Y}Y^T\nabla^2h(YY^T)[YY^T][\eta_{\uparrow_Y}Y^T]) \\ 
	&= \mathrm{tr}(Y \nabla^2h(YY^T)[Y\eta_{\uparrow_Y}^T+\eta_{\uparrow_Y}Y^T]\xi_{\uparrow_Y}).
\end{align*} 
Consequently the proposed preconditioner $\mathscr{P}(\pi(Y))^{-1}$ is self-adjoint.

Let $\lambda$ and $\xi_{\pi(Y)}$ (whose horizontal lift is ${\xi}_{\uparrow_Y}$) be an arbitrary eigenpair of $\mathscr{P}(\pi(Y))$. Then we have 
\begin{align*}
g_{\pi(Y)}(\mathscr{P}(\pi(Y))^{-1}[{\xi}_{\pi(Y)}], {\xi}_{\pi(Y)}) &= \lambda g_{\pi(Y)}(\xi_{\pi(Y)},\xi_{\pi(Y)}),
\end{align*}
which is equivalent to 
\begin{align*}
g_{Y}((\mathscr{P}(\pi(Y))^{-1}[{\xi}_{\pi(Y)}])_{\uparrow_Y}, {\xi}_{\uparrow_Y}) &= \lambda g_{Y}(\xi_{\uparrow_Y},\xi_{\uparrow_Y}).
\end{align*}
The left-hand side equals to 
\begin{equation}	
\begin{aligned} \label{pre:01}
	\mathrm{left} &= 2\mathrm{tr}(Y^T\nabla^2h(YY^T)[Y\xi_{\uparrow_Y}^T+\xi_{\uparrow_Y}Y^T]\xi_{\uparrow_Y}) \\ 
	& = 2\mathrm{tr}(Y^{T}AY\xi_{\uparrow_Y}^{T} M\xi_{\uparrow_Y})+ 2\mathrm{tr}(Y^{T}MY\xi_{\uparrow_Y}^{T}A\xi_{\uparrow_Y}) + 2\mathrm{tr}(Y^{T}A\xi_{\uparrow_Y} Y^{T}M\xi_{\uparrow_Y} ) + 2\mathrm{tr}(Y^{T}M\xi_{\uparrow_Y} Y^{T}A\xi_{\uparrow_Y}).
\end{aligned}
\end{equation}
Let $F=Y\xi_{\uparrow_Y}^T+\xi_{\uparrow_Y}Y^T$. We have 
\[
	\mathrm{tr}(F \nabla^{2}h(YY^{T})[F]) = \mathrm{vec}(F)^{T} \mathcal{L} \mathrm{vec}(F) \ge 0.
\]
The equation holds if and only if $F=0$, because $\mathcal{L}=A\otimes M+M\otimes A$ is positive definite~\cite{Bart10}. On the other hand, we note that 
\begin{equation}	
\begin{aligned} \label{pre:02}
&\mathrm{tr}(F\nabla^{2}h(YY^{T})[F]) \\
&= \mathrm{tr}(\xi_{\uparrow_Y} Y^{T} \nabla^{2}h(YY^{T})[Y\xi_{\uparrow_Y}^{T}+\xi_{\uparrow_Y} Y^{T}]) + \mathrm{tr}(Y\xi_{\uparrow_Y}^{T}\nabla^{2}h(YY^{T})[Y\xi_{\uparrow_Y}^{T}+\xi_{\uparrow_Y} Y^{T}]) \\ &= \mathrm{tr}(Y^{T}AY\xi_{\uparrow_Y}^{T} M\xi_{\uparrow_Y})+\mathrm{tr}(Y^{T}MY\xi_{\uparrow_Y}^{T}A\xi_{\uparrow_Y}) + \mathrm{tr}(Y^{T}A\xi_{\uparrow_Y} Y^{T}M\xi_{\uparrow_Y} ) + \mathrm{tr}(Y^{T}M\xi_{\uparrow_Y} Y^{T}A\xi_{\uparrow_Y})  \\ & \quad +  \mathrm{tr}(\xi_{\uparrow_Y}^{T} A\xi_{\uparrow_Y} Y^{T}MY)+\mathrm{tr}(\xi_{\uparrow_Y}^{T}M\xi_{\uparrow_Y} Y^{T}AY)+\mathrm{tr}(\xi_{\uparrow_Y}^{T}AY\xi_{\uparrow_Y}^{T}MY)+\mathrm{tr}(\xi_{\uparrow_Y}^{T}MY\xi_{\uparrow_Y}^{T}AY) \\ &= g_{Y}((\mathscr{P}(\pi(Y))^{-1}[\xi_{\pi(Y)}])_{\uparrow_Y},[\xi_{\uparrow_Y}]), 
\end{aligned} 
\end{equation}
and 
\begin{equation} \label{pre:03}
\begin{aligned} 
\mathrm{vec}(F)^{T}\mathrm{vec}(F) &= \mathrm{tr}(F^{T}F) \\ &= \mathrm{tr}((Y\xi_{\uparrow_Y}^{T}+\xi_{\uparrow_Y} Y^{T})(Y\xi_{\uparrow_Y}^{T}+\xi_{\uparrow_Y} Y^{T})) \\ &= \mathrm{tr}(Y\xi_{\uparrow_Y}^{T}Y\xi_{\uparrow_Y}^{T} + \xi_{\uparrow_Y} Y^{T}Y\xi_{\uparrow_Y}^{T} + Y\xi_{\uparrow_Y}^{T}\xi_{\uparrow_Y} Y^{T}+\xi_{\uparrow_Y} Y^{T}\xi_{\uparrow_Y} Y^{T}) \\ &= 2\mathrm{tr}(Y^{T}\xi_{\uparrow_Y} Y^{T}\xi_{\uparrow_Y} + Y^{T}Y\xi_{\uparrow_Y}^{T}\xi_{\uparrow_Y}) \\ &= g_{Y}(\xi_{\uparrow_Y},\xi_{\uparrow_Y}).  
\end{aligned} 
\end{equation}
Hence by~\eqref{pre:01},~\eqref{pre:02}, and~\eqref{pre:03}, we have 
\begin{align*} \lambda &= \frac{g_{Y}((\mathscr{P}(\pi(Y))^{-1}[\xi_{\pi(Y)}])_{\uparrow_Y},\xi_{\uparrow_Y})}{g_{Y}(\xi_{\uparrow_Y},\xi_{\uparrow_Y})} = \frac{\mathrm{tr}(F\nabla^{2}h(YY^{T})[F])}{\mathrm{tr}(F^{T}F)} \\ &= \frac{\mathrm{vec}(F)^{T}\mathcal{L} \mathrm{vec}(F)}{\mathrm{vec}(F)^{T} \mathrm{vec}(F)},	
\end{align*}
which implies that 
\begin{align} \label{pre:04}	
0<\lambda_{\min}(\mathcal{L}) \le \lambda(\mathscr{P}(\pi(Y))^{-1}) \le \lambda_{\max}(\mathcal{L}) .
\end{align}
Therefore, Assumption~\ref{ass:precond} holds with $p_1=1/\lambda_{\max}(\mathcal{L})$ and $p_2=1/\lambda_{\min}(\mathcal{L})$. It follows that $\mathscr{P}(\pi(Y))$ is invertible.
\end{proof}

\paragraph{Positive definiteness of the Hessian.} It is natural to characterize the positive definiteness of the Hessian of Problem~\eqref{Pro_stat-FinalProb}, which is stated in Theorem~\ref{thm:PosDefHess}, after analysing the positive definiteness of the proposed preconditioner. 

\begin{theorem} \label{thm:PosDefHess}
	For all point $\pi(Y) \in \mathbb{R}_*^{n\times p}/\mathcal{O}_p$, if $2\lambda_{\min}(Y^TY)\lambda_{\min}(A)\lambda_{\min}(M)>p\mathcal{R}_Y$, where $\mathcal{R}_Y:=\|AYY^TM+MYY^T-C\|_F$ is the residual of Equation~\eqref{Intro-LyapMatEqu} at $Y Y^T$, then the Hessian of $f$ in Problem~\eqref{Pro_stat-FinalProb} is self-adjoint positive definite at $\pi(Y)$. Moreover, if $\pi(Y^*)$ is a critical point of $f$, then $\mathrm{Hess}f(\pi(Y^*))$ is self-adjoint positive definite if $2\lambda_{min}((Y^*)^TY^*)\lambda_{\min}(A)\lambda_{\min}(M)>\mathcal{R}_{Y^*}$. 
\end{theorem}

\begin{proof}
For any $\pi(Y)\in \mathbb{R}_*^{n\times p}/\mathcal{O}_p$, $\xi_{\pi(Y)}\in\mathrm{T}_{\pi(Y)}\mathbb{R}_*^{n\times p}/\mathcal{O}_p$ (whose horizontal lift is $\xi_{\uparrow_{Y}}$), we have 
\begin{align}	
 & g_Y((\mathrm{Hess}f(\pi(Y))[\xi_{\pi(Y)}])_{\uparrow_{Y}},\xi_{\uparrow_{Y}}) \nonumber \\
 &= g_Y((\mathscr{P}(\pi(Y))^{-1}[\xi_{\pi(Y)}])_{\uparrow_{Y}},\xi_{\uparrow_{Y}})  \nonumber \\
&\quad + g_{Y}((I-P_{Y})\nabla h(YY^T)(I-P_{Y})\xi_{\uparrow_{Y}}(Y^TY)^{-1},\xi_{\uparrow_{Y}})  \nonumber \\ 
&= g_{Y}((\mathscr{P}(\pi(Y))^{-1}[\xi_{\pi(Y)}])_{\uparrow_{Y}},\xi_{\uparrow_{Y}}) + 2\mathrm{tr}(\xi_{\uparrow_{Y}}^T(I-P_{Y})\nabla h(YY^T)(I-P_{Y})\xi_{\uparrow_{Y}}) \label{thm:PosDefHess:1} \\ 
&\ge \lambda_{\min}(\mathcal{L})\|\xi_{\uparrow_Y}\|_{Y}^2 - 2p \mathcal{R}_Y \|\xi_{\uparrow_{Y}}\|_{F}^2 \nonumber  \\
& \ge 2\lambda_{\min}(A)\lambda_{\min}(M)\|\xi_{\uparrow_Y}\|_Y^2 - 2p \mathcal{R}_Y\|\xi_{\uparrow_Y}\|_F^2, \nonumber
\end{align}
where $\mathcal{R}_Y:=\|\nabla h(YY^T)\|_F=\|AYY^TM+MYY^T-C\|_F$, the second to the last inequality is due to~\eqref{pre:04}, 
\begin{align*}
	\mathrm{tr}(\xi_{\uparrow_{Y}}^T(I-P_{Y})\nabla h(YY^T)(I-P_Y)\xi_{\uparrow_{Y}}) &= - \mathrm{tr}(\xi_{\uparrow_{Y}}^T(P_{Y}-I)\nabla h(YY^T)(I-P_Y)\xi_{\uparrow_{Y}}) \\ 
	& \ge - \|I-P_Y\|_F^2\mathcal{R}_Y\|\xi_{\uparrow_Y}\|_F^2, \nonumber 
\end{align*}
and $\|I-P_Y\|_F^2=p$; 
and the last due to $\lambda_{\min}(A\otimes M)=\lambda_{\min}(A)\lambda_{\min}(M)$. Additionally, we note that 
\begin{align*}
	\|\xi_{\uparrow_Y}\|_Y^2 &= g_Y(\xi_{\uparrow_Y},\xi_{\uparrow_Y}) \ge 2 \mathrm{tr}(Y^TY\xi_{\uparrow_Y}^T\xi_{\uparrow_Y}) = 2 \mathrm{tr}((Y\xi_{\uparrow_Y}^T)^TY\xi_{\uparrow_Y}^T)  \\ 
	&= 2 \mathrm{vec}(\xi_{\uparrow_Y}^T)(I\otimes Y^TY)\mathrm{vec}(\xi_{\uparrow_Y}^T) \\ 
	& \ge 2\lambda_{\min}(Y^TY)) \|\xi_{\uparrow_Y}\|_F^2. 
\end{align*}
Thus we have
\[
	g_Y((\mathrm{Hess}f(\pi(Y))[\xi_{\pi(Y)}])_{\uparrow_{Y}},\xi_{\uparrow_{Y}}) \ge (4\lambda_{\min}(Y^TY) \lambda_{\min}(A) \lambda_{\min}(M) - 2p \mathcal{R}_Y )\|\xi_{\uparrow_Y}\|_F^2.
\]

%where $\mathcal{R}_Y:=\|\nabla h(YY^T)\|_F=\|AYY^TM+MYY^T-C\|_F$, the second to the last inequality is due to~\eqref{pre:04}, 
%\begin{align*}
%	\mathrm{tr}(\xi_{\uparrow_{Y}}^T(I-P_{Y})\nabla h(YY^T)(I-P_Y)\xi_{\uparrow_{Y}}) &= - \mathrm{tr}(\xi_{\uparrow_{Y}}^T(P_{Y}-I)\nabla h(YY^T)(I-P_Y)\xi_{\uparrow_{Y}}) \\ 
%	& \ge - \|I-P_Y\|_F^2\mathcal{R}_Y\|\xi_{\uparrow_Y}\|_F^2, 
%\end{align*}
%and $\|I-P_Y\|_F^2=p$; 
%and the last due to $\lambda_{\min}(A\otimes M)=\lambda_{\min}(A)\lambda_{\min}(M)$.

Let $\pi(Y^*)$ be a critical point of Problem~\eqref{Pro_stat-FinalProb}. Then we have 
\begin{align*}
	(\mathrm{grad}f(\pi(Y^*)))_{\uparrow_{Y^*}} =\left( I - \frac{1}{2}P_{Y^*} \right)\nabla h(Y^*(Y^*)^T)Y^*((Y^*)^TY^*)^{-1} = 0	,
\end{align*}
which gives that 
\[
	\nabla h(Y^*(Y^*)^T)P_{Y^*} = 0,
\]
where we multiply by $(I+P_{Y^*})$ and $(Y^*)^T$ from left and right-hand sides respectively. This shows that $\nabla h(Y^*(Y^*)^T)\in \mathrm{range}(Y^*)^{\perp}$, and thus $(I-P_{Y^*})\nabla h(Y^*(Y^*)^T)=\nabla h(Y^*(Y^*)^T)$.
Using this fact, for  any $\xi_{\uparrow_{Y^*}}\in \mathrm{H}_{Y^*}$, we have that 
\begin{align*}
	\eqref{thm:PosDefHess:1} &=  g_{Y^*}((\mathscr{P}(\pi(Y^*))^{-1}[\xi_{\pi(Y^*)}])_{\uparrow_{Y^*}},\xi_{\uparrow_{Y^*}}) + 2\mathrm{tr}(\xi_{\uparrow_{Y^*}}^T \nabla h(Y^*(Y^*)^T) \xi_{\uparrow_{Y^*}}) \\ 
	& \ge (4\lambda_{\min}((Y^*)^TY^*) \lambda_{\min}(A)\lambda_{\min}(M)-2\mathcal{R}_{Y^*})\|\xi_{\uparrow_{Y^*}}\|^2_{F},
\end{align*}
which completes the proof.
\end{proof}

\paragraph{Computing the preconditioner.}
Note that for any $Z \in \mathrm{T}_Y\mathbb{R}_*^{n \times p}=\mathbb{R}^{n\times p}$, there exist $W \in \mathbb{R}^{p \times p}, K \in \mathbb{R}^{(n - p) \times p}$ such that $Z = YW + Y_{\perp_M} K$.
% where $Y_{\perp_M} \in \mathbb{R}^{n \times (n - p)}$ satisfying $Y^T M Y_{\perp_M} = 0$ and $Y_{\perp_M}^T Y_{\perp_M} = I_{n - p}$. 
Furthermore, we have $Z = Y \mathrm{sym}(W) + Y \mathrm{skew}(W) + Y_{\perp_M} K$. Note that $Y \mathrm{skew}(W) \in \mathrm{V}_Y$, we have $Y (Y \mathrm{skew}(W))^T + Y \mathrm{skew}(W) Y^T = 0$, which implies the skew symmetric part of $W$ does not affect the solution of~\eqref{Precond-Metric1-ApproxNewEqu}. Next, we show an approach to find the solution of
\begin{align} \label{e03}
  \left( I - \frac{1}{2}P_Y \right)\nabla^2h(YY^T)[Y\xi_{Y}^T+\xi_{Y}Y^T]Y(Y^TY)^{-1}=\eta_{\uparrow_Y},
\end{align}
over the space $\{Y S + Y_{\perp_M} K \mid S \in\mathcal{S}_p^{\mathrm{sym}}, K \in \mathbb{R}^{(n - p) \times p}  \}$.

%\whcomm{[ZHTODO: $\xi_{\uparrow_Y} \Rightarrow \xi_Y$. ]}{}
Multiplying~\eqref{e03} by $I+P_Y$ and $(Y^TY)$ from left and right respectively yields 
\begin{equation} \label{e02}
  \nabla^2h(YY^T)[Y\xi_{Y}^T+\xi_{Y}Y^T]Y=(I+P_Y)\eta_{\uparrow_Y}(Y^TY).
\end{equation}
It follows from~\eqref{e02} and the expression of $\nabla^2 h(Y Y^T)$ in Proposition~\ref{IngredQuotMani-Prop-Gradient} that
\begin{align} \label{Precond-Metric1-ApproxNewEqu-1}
  [A(Y\xi_{Y}^T+\xi_{Y}Y^T)M+M(Y\xi_{Y}^T+\xi_{Y}Y^T)A]Y=(I+P_Y)\eta_{\uparrow_Y}(Y^TY).
\end{align} 
Multiplying~\eqref{Precond-Metric1-ApproxNewEqu-1} from left by $Y$ yields
\begin{equation}
\begin{aligned} \label{Precond-Metric1-1}
  &Y^TAY\xi_{Y}^TMY +Y^TA\xi_{Y}Y^TMY+Y^TMY\xi_{Y}^TAY+Y^TM\xi_{Y}Y^TAY\\ 
  &= 2Y^T\eta_{\uparrow_Y}(Y^TY) 
\end{aligned}
\end{equation}
and multiplying~\eqref{Precond-Metric1-ApproxNewEqu-1} from left by $Y_{\perp_M}$ yields
\begin{equation}
  \begin{aligned} \label{Precond-Metric1-2}
    & Y_{\perp_M}^TAY\xi_{Y}^TMY+Y_{\perp_M}^TA\xi_{Y}Y^TMY+Y_{\perp_M}^T  M\xi_{Y}Y^TAY\\
    & =Y_{\perp_M}^T(I+Y(Y^TY)^{-1}Y^T)\eta_{\uparrow_Y}(Y^TY) .
  \end{aligned}
\end{equation}
Plugging the decomposition $\xi_{Y}=YS_\xi+Y_{\perp_M}K_\xi$ into~\eqref{Precond-Metric1-1} and~\eqref{Precond-Metric1-2} respectively gives
\begin{equation}
  \begin{aligned} \label{Precond-Metric1-12}
    & 2Y^TAYS_\xi Y^TMY+Y^TAY_{\perp_M} K_\xi Y^TMY+Y^TMYK_\xi^TY_{\perp_M}^TAY + 2Y^TMYS_\xi Y^TAY \\ 
    & = 2Y^T\eta_{\uparrow_Y}(Y^TY), 
  \end{aligned}
\end{equation}
and
\begin{equation}
  \begin{aligned} \label{Precond-Metric1-22}
    & 2Y_{\perp_M}^TAYS_\xi Y^TMY+Y_{\perp_M}^TAY_{\perp_M} K_\xi Y^TMY +Y_{\perp_M}^T MY_{\perp_M} K_\xi Y^TAY \\ 
    & = Y_{\perp_M}^T(I+Y(Y^TY)^{-1}Y^T)\eta_{\uparrow_Y}(Y^TY).
  \end{aligned}
\end{equation}
Let $Z_\xi=Y_{\perp_M}K_\xi$. We obtain from \eqref{Precond-Metric1-22} that
\begin{align*} 
  Y_{\perp_M}^TAZ_\xi Y^TMY+Y_{\perp_M}^TMZ_\xi Y^TAY=Y_{\perp_M}^T(I+Y(Y^TY)^{-1}Y^T)\eta_{\uparrow_Y}(Y^TY)-2Y_{\perp_M}^TAYS_\xi Y^TMY. 
\end{align*}
Let $Y^TMY=LL^T$ be the Cholesky factorization. We have 
\begin{align*}
  & Y_{\perp_M}^TAZ_\xi L+Y_{\perp_M}^TMZ_\xi LL^{-1}Y^TAYL^{-T}\\ 
  & =(Y_{\perp_M}^T(I+Y(Y^TY)^{-1}Y^T)\eta_{\uparrow_Y}(Y^TY)-2Y_{\perp_M}^TAYS_\xi Y^TMY)L^{-T}. 
\end{align*}
Let $L^{-1}Y^TAYL^{-T}=Q\Lambda Q^T$ be the eigenvalue decomposition. We have 
 \begin{align*} 
  &Y_{\perp_M}^T AZ_\xi LQ+Y_{\perp_M}^TMZ_\xi LQ\Lambda  \\
  & = (Y_{\perp_M}^T(I+Y(Y^TY)^{-1}Y^T)\eta_{\uparrow_Y}(Y^TY)-2Y_{\perp_M}^TAYS_\xi Y^TMY)L^{-T}Q. 
  \end{align*}
 Let $\tilde{Z}_\xi=Z_\xi LQ$. Then $Y^TM\tilde{Z}_\xi=Y^TMY_{\perp_M} K_\xi LQ=0$ and we have
  \begin{align*} 
    & Y_{\perp_M}^TA\tilde{Z}_\xi + Y_{\perp_M}^T M\tilde{Z}_\xi\Lambda = (Y_{\perp_M}^T(I+Y(Y^TY)^{-1}Y^T)\eta_{\uparrow_Y}(Y^TY)-2Y_{\perp_M}^TAYS_\xi Y^TMY)L^{-T}Q,
  \end{align*} 
that is, 
\begin{align} \label{e04}
  (I-\hat{V}\hat{V}^T)(A\tilde{Z}_\xi+M\tilde{Z}_\xi\Lambda)=(I-\hat{V}\hat{V}^T)((I+Y(Y^TY)^{-1}Y^T)\eta_{\uparrow_Y}(Y^TY)-2AYS_\xi Y^TMY)L^{-T}Q, 
\end{align}
where $\hat{V}=\mathrm{orthonormal}(MY)=MYG$.
Using $\tilde{Z}_\xi(:,i)$ to denote the $i$-th column of $\tilde{Z}_\xi$, we have the saddle-point problems from~\eqref{e04} as follows 
\begin{equation}
  \begin{aligned} \label{Precond-Metric1-SadProd1}    
    \begin{bmatrix} A+\lambda_iM &\hat{V}\\ \hat{V}^T & 0 \end{bmatrix} \begin{bmatrix} \tilde{Z}_\xi(:,i)\\y\end{bmatrix}=\begin{bmatrix} F_{\eta,\xi}(:,i) \\ 0 \end{bmatrix}
  \end{aligned}
\end{equation}
where $\Lambda = \mathrm{diag}\{\lambda_1,\cdots,\lambda_p\}$, $y\in\mathbb{R}^p$, $F_{\eta,\xi}=(I-\hat{V}\hat{V}^T)((I+Y(Y^TY)^{-1}Y^T)\eta_{\uparrow_Y}(Y^TY)-2AYS_\xi Y^TMY)L^{-T}Q$, and the second equation holds due to $\hat{V}^T\tilde{Z}_\xi=G^TY^TMY_{\perp_M} K_\xi LQ=0$. Therefore, if we denote the solution of the $i$-th saddle-point problem, corresponding to (\ref{Precond-Metric1-SadProd1}), with right-hand side $F$ by $\mathcal{T}_i^{-1}(F)$, then (\ref{Precond-Metric1-SadProd1}) can be write as 
\begin{equation}
  \begin{aligned} \label{Precond-Metric1-SadProdSol}
    \tilde{Z}_\xi(:,i) &=\mathcal{T}_i^{-1}\left((I-\hat{V}\hat{V}^T)(I+Y(Y^TY)^{-1}Y^T)\eta_{\uparrow_Y}(Y^TY)L^{-1}Q(:,i)\right) \\ 
    & \quad - \mathcal{T}_i^{-1}\left((I-\hat{V}\hat{V}^T)2AY\right)L^{-T}Q\tilde{S}_\xi(:,i),
  \end{aligned}
\end{equation}
where $\tilde{S}_\xi=Q^TL^TS_\xi LQ$. 

Plugging $\tilde{S}_\xi=Q^TL^TS_\xi LQ$ into (\ref{Precond-Metric1-1}), we obtain 
\begin{align} \label{e05}
  2\Lambda\tilde{S}_\xi+2\tilde{S}_\xi\Lambda+Q^TL^{-1}Y^TA\tilde{Z}_\xi+\tilde{Z}_\xi^TAYL^{-T}Q=2Q^TL^{-1}Y^T\eta_Y(Y^TY)L^{-T}Q.
\end{align}
Let 
\begin{align*}
  v_i&=Q^TL^{-1}Y^TA\mathcal{T}_i^{-1}\left( (I-\hat{V}\hat{V}^T)(I+Y(Y^TY)^{-1}Y^T)\eta_Y(Y^TY)L^{-T}Q(:,i)\right), \\ w_i&=Q^TL^{-1}Y^TA\mathcal{T}_i^{-1}\left( (I-\hat{V}\hat{V}^T)2AYL^{-T}Q\right)\tilde{S}_\xi(:,i).
\end{align*}
It follows from~\eqref{e05} that
\begin{align*}
  2\Lambda \tilde{S}_\xi+2\tilde{S}_\xi\Lambda + \begin{bmatrix} v_1-w_1&\cdots&v_k-w_k \end{bmatrix} + \begin{bmatrix} (v_1-w_1)^T \\ \cdots\\(v_k-w_k)^T \end{bmatrix}=2Q^TL^{-1}Y^T\eta_Y(Y^TY)L^{-T}Q,
\end{align*}
that is, 
\begin{align} \label{Precond-Metric1-S}
  (\mathcal{K}+\Pi\mathcal{K}\Pi)\mathrm{vec}(\tilde{S}_\xi)=\mathrm{vec}(R),
\end{align}
where $\mathcal{K}=\mathrm{diag}(K_i)\in\mathbb{R}^{p^2\times p^2}$, $K_i=2\lambda_iI-Q^TL^{-1}Y^TA\mathcal{T}_i^{-1}\left( (I-\hat{V}\hat{V}^T)2AYL^{-T}Q\right)$,
$\Pi$ is the perfect shuttle matrix \cite{Loan2000TheUK} such that $\mathrm{vec}(X)=\Pi\mathrm{vec}(X)^T$ and $R=2Q^TL^{-1}Y^T\eta_{\uparrow_Y}(Y^TY)L^{-T}Q-[v_1,\cdots,v_k]-[v_1,\cdots,v_k]^T$.
To sum up, after solving (\ref{Precond-Metric1-S}), we obtain $\tilde{Z}_\xi$ from (\ref{Precond-Metric1-SadProdSol}). Therefore, the solution $\xi_Y$ of~\eqref{e03} is obtained.

Since $\xi_Y = \mathcal{P}_Y^{\mathrm{V}}(\xi_Y) + \mathcal{P}_Y^{\mathrm{H}} (\xi_Y)$ and $\mathcal{P}_Y^{\mathrm{V}} (\xi_Y)$ does not influence the left hand side of~\eqref{e03} because of $Y(\mathcal{P}_Y^{\mathrm{V}}(\xi_Y))^T+\mathcal{P}_Y^{\mathrm{V}}(\xi_Y)Y^T=0$, the solution to~\eqref{Precond-Metric1-ApproxNewEqu} is $\xi_{\uparrow_Y} = \mathcal{P}_Y^{\mathrm{H}} (\xi_Y)$. The final algorithm for solving~\eqref{Precond-Metric1-ApproxNewEqu} is stated in Algorithm \ref{Precond-Metric1-Alg}.

\begin{algorithm}[htp]
  \caption{Preconditioner under Riemannian metric~\eqref{IngredQuotMani-Metric1}} 
  \begin{algorithmic}[1] \label{Precond-Metric1-Alg}
  \REQUIRE Matrices $A$ and $M$ and horizontal vector $\eta_{\uparrow_Y}\in\mathrm{H}_Y^1$;
  \ENSURE $\xi_{\uparrow_Y}$ satisfying (\ref{Precond-Metric1-ApproxNewEqu-1});
  \STATE Set $LL^T\gets Y^TMY$ (Cholesky factorization);
  \STATE Set $Q\Lambda Q^T\gets L^{-1}Y^TAYL^{-T}$ (Eigenvalues decomposition); 
  \STATE Set $\hat{V}\gets \mathrm{orthonormal}(MY)$;
  \STATE Set $v_i\gets Q^TL^{-1}Y^TA\mathcal{T}_i^{-1}\left( (I-\hat{V}\hat{V}^T)(I+Y(Y^TY)^{-1}Y^T)\eta_{\uparrow_Y}(Y^TY)L^{-T}Q(:,i)\right)$;
  \STATE Set $K_i\gets 2\lambda_iI-Q^TL^{-1}Y^TA\mathcal{T}_i^{-1}\left( (I-\hat{V}\hat{V}^T)2AYL^{-T}Q\right)$;
  \STATE Set $R\gets 2Q^TL^{-1}Y^T\eta_{\uparrow_Y}(Y^TY)L^{-T}Q-[v_1,\cdots,v_k]-[v_1,\cdots,v_k]^T$;
  \STATE Solve for $\tilde{S}_\xi$ by $(\mathcal{K}+\Pi\mathcal{K}\Pi)\mathrm{vec}(\tilde{S}_\xi)=\mathrm{vec}(R)$;
  \STATE Solve for $\tilde{Z}_\xi$ by  
%  {\small
  \begin{align*}
  \tilde{Z}_\xi (:, i)&\gets \mathcal{T}_i^{-1}\left((I-\hat{V}\hat{V}^T)(I+Y(Y^TY)^{-1}Y^T)\eta_{\uparrow_Y}(Y^TY)L^{-1}Q(:,i)\right) \\ & \;\;\;\;\; - \mathcal{T}_i^{-1}\left((I-\hat{V}\hat{V}^T)2AY\right)L^{-T}Q\tilde{S}_\xi(:,i);
  \end{align*}
%  }
  \STATE Set $Z_{\xi,1}\gets \frac{1}{2}Y\left( (Y^TY)^{-1}Y^T\tilde{Z}_\xi Q^TL^{-1} + L^{-T}Q\tilde{Z}_{\xi}^TY(Y^TY)^{-1}  \right)$;
  \STATE Set $Z_{\xi,2}\gets (I - Y(Y^TY)^{-1}Y^T\tilde{Z}_{\xi}Q^TL^{-1}$;
  \STATE Set $\xi_{\uparrow_Y}\gets YL^{-T}Q\tilde{S}_\xi Q^TL^{-1}+ Z_{\xi,1} + Z_{\xi,2}$.
  
 \end{algorithmic}
\end{algorithm}

For applying this preconditioner, the dominating costs lie in solving the saddle-point problems and solving the linear system (\ref{Precond-Metric1-S}). As for the saddle-point problems (\ref{Precond-Metric1-SadProd1}), we firstly can solve $\mathcal{T}_i(X_i)=B_i$ by eliminating the (negative) Schur complement $S_i=\hat{V}^T(A+\lambda_iM)^{-1}\hat{V}$ (see \cite{benzi_numerical_2005} for details), where $$ B_i=[(I-\hat{V}\hat{V}^T)(I+Y(Y^TY)^{-1}Y^T)\eta_{\uparrow_Y}(Y^TY)^{-1}L^{-1}Q(:,i),\;(I-\hat{V}\hat{V}^T)2AYL^{-1}Q].
$$ 
Therefore, we have
\begin{align*} 
v_i &= Q^{T}L^{-1}Y^TAX_i(:,1), \\ 
%&\whcomm{}{w_i= ...} \\
K_i &= 2\lambda_iI- Q^TL^{-1}Y^TA J_i(:,2:\text{end}),
\end{align*}
where $N_i = S_i^{-1}(\hat{V}^T(A+\lambda_iM)^{-1}B_i)$ and $J_i=(A+\lambda_iM)^{-1}B_i - (A+\lambda_iM)^{-1}\hat{V}N_i$. The linear system~\eqref{Precond-Metric1-S} therefore can be solved by the conjugate gradient method.

\subsection{Increasing Rank Algorithm}

Algorithm \ref{RieNew-TruncatedNewton} gives a low-rank approximation when the rank of the exact solution to Equation~\eqref{Intro-LyapMatEqu} is known in advance. 
In practice, however, it
%the rank of the exact solution $X^*$ of Equation (\ref{Intro-LyapMatEqu}) 
is unknown beforehand. 
Therefore, we propose a rank-increasing algorithm, which solves Equation~\eqref{Intro-LyapMatEqu} with a low estimation of the rank. If the residual of the accumulation point is not sufficiently small, then the rank is increased and Problem~\eqref{Pro_stat-FinalProb} is solved by Algorithm~\ref{RieNew-TruncatedNewton} with an initial iterate motivated from the accumulation point of the lower rank. 
%Hence, we find a series of low-rank approximations of Problem (\ref{Pro_stat-FinalProb}) by Algorithm \ref{RieNew-TruncatedNewton} with rank $p$ increasing. 
The details are stated in Algorithm~\ref{FinalAlgPrecond-RLyap-RNewton} (IRRLyap).

\begin{algorithm}[htbp]
  \caption{An Increasing Rank Riemannian Method for Lyapunov Equations (IRRLyap)} 
  \begin{algorithmic}[1] \label{FinalAlgPrecond-RLyap-RNewton}
  \REQUIRE minimum rank $p_{\min}$; maximum rank $p_{\max}$; rank increment $p_{\mathrm{inc}}$; initial iterate $Y_{p_{\min}}^{\mathrm{initial}} \in \mathbb{R}_*^{n\times p_{\min}}$; tolerance sequence of inner iteration $\{\tau_{p}:p \in \{ p_{\min}, p_{\min} + p_{\mathrm{inc}}, p_{\min} + 2 p_{\mathrm{inc}},\ldots, p_{\max} \} \}$; residual tolerance $\tau$;
  \ENSURE low-rank approximation $\widetilde{Y}$;
  \FOR{$p=p_{\min}, p_{\min} + p_{\mathrm{inc}}, p_{\min} + 2 p_{\mathrm{inc}}, \ldots, p_{\max}$}
  \STATE \label{FinalAlgPrecond-RLyap-RNewton-st01} Invoke an optimization algorithm, such as Algorithm \ref{RieNew-TruncatedNewton}, to approximately solve Problem \eqref{Pro_stat-FinalProb} with the initial iterate $\pi(Y_p^{\mathrm{initial}})$ until the last iterate $\pi(Y_p)$ satisfies $\|\grad f(\pi(Y_p))\|\le \tau_{p} \|\mathrm{grad}f(\pi(Y_p^{\mathrm{initial}}))\|$;
%   \zwhcomm{where $Y_p\gets Y_p^{(k)}$ if $\|\grad f(\pi(Y_p^{(k)}))\|\le\tau_{p}$}{};
  \STATE \label{FinalAlgPrecond-RLyap-RNewton-st02} Compute relative residual of $Y_p$: $r_p \gets \|AY_pY_p^TM + MY_pY_p^TA-C\|_F/\|C\|_F$; 
  \IF {$r_p\le \tau$}
  \STATE \label{FinalAlgPrecond-RLyap-RNewton-st03} Return $\widetilde{Y}\gets Y_p$;
  \ELSE 
  \STATE \label{FinalAlgPrecond-RLyap-RNewton-st04} Calculate the next initial iterate $Y_{p+p_{\mathrm{inc}}}^{\mathrm{initial}}$ by performing one step of steepest descent on $\begin{bmatrix}
    Y_p & \mathbf{0}_{n\times p_{\text{inc}}}
  \end{bmatrix}$;
  \ENDIF
  \ENDFOR
  \STATE Return $\widetilde{Y}\gets Y_{p_{\max}}$;
 \end{algorithmic}
\end{algorithm}

Step~\ref{FinalAlgPrecond-RLyap-RNewton-st01} in Algorithm~\ref{FinalAlgPrecond-RLyap-RNewton} can invoke Algorithm~\ref{RieNew-TruncatedNewton}, which approximately solves Proplem~\eqref{Pro_stat-FinalProb} with fixed rank $p$ in the sense that the norm of gradient is reduced sufficiently, i.e., $\|\grad f(\pi(Y_p))\|\le \tau_{p} \|\mathrm{grad}f(\pi(Y_p^{\mathrm{initial}}))\|$. The tolerance sequence $\{\tau_p\}$ can be prescribed parameters or adaptively dependent on the current iterate.

	Step~\ref{FinalAlgPrecond-RLyap-RNewton-st02} computes the relative residual $r_p$ by following the steps in~\cite{Bart10}. Note that the steps in~\cite{Bart10} avoid the computations of the $n$-by-$n$ matrices in the residual and only require $O(n p^2)$ flops.
	
	Step~\ref{FinalAlgPrecond-RLyap-RNewton-st04} performs one step steepest descent with initial iterate $[Y_p\; 0_{n\times p_{\mathrm{inc}}}]$ for minimizing the cost function $\tilde{f}: \mathbb{R}^{n \times (p + p_{\mathrm{inc}})} \rightarrow \mathbb{R}: Y \mapsto \mathrm{tr}(Y^TAYY^TMY) - \mathrm{tr}(Y^TCY)$. Since $\tilde{f}$ is defined on the Euclidean space $\mathbb{R}^{n \times (p + p_{\mathrm{inc}})}$, the steepest descent algorithm can be defined at $[Y_p\; 0_{n\times p_{\mathrm{inc}}}]$. In our implementation, the step size is found by the backtracking line search algorithm. Therefore, Algorithm~\ref{FinalAlgPrecond-RLyap-RNewton} is a descent algorithm in the sense that the value of $h(Y_p Y_p^T) > h(Y_{p+p_\mathrm{inc}} Y_{p+p_\mathrm{inc}}^T)$.

\section{Numerical Experiments} \label{NumExp} 

In this section, we illustrate the performance of the proposed algorithms. % with three examples. 
%In the first example shown 
In Section~\ref{sec:ComAlg1}, we demonstrate the performance of the proposed algorithm~\ref{RieNew-TruncatedNewton} compared with some existing Riemannian methods and
the gain of the proposed preconditioner for Algorithm~\ref{RieNew-TruncatedNewton}, and investigate the superiority defining the problem on quotient manifold $\mathbb{R}_*^{n\times p}/\mathcal{O}_p$ over defining the problem on the embedded submanifold $\mathcal{S}_+(p,n)$. 
%The second and third examples are used i
In Section~\ref{NumExp-CompWithExistingSols}.
%In the second example, 
we investigate the quality of the solutions and the efficiency of Algorithm~\ref{FinalAlgPrecond-RLyap-RNewton} by comparing with some state-of-the-art low-rank methods. 
% In the third example, we report on the performance of Algorithm~\ref{FinalAlgPrecond-RLyap-RNewton} compared to three state-of-the-art low-rank methods for real-world Lyapunov equations. 

\subsection{Testing Environment, Data, and Parameter Settings}

All experiments are done in MATLAB R2021b on a 64-bit GNU/Linux platform with 2.10 GHz CPU (Intel Xeon Gold 5318Y). The Riemannian optimization methods used in this paper are implemented based on ROPTLIB~\cite{HAGH18}. 
%The code
%, a C++ package for Riemannian optimization with interfaces to MATLAB and Julia\footnote{The code is available at www}.

In the first example, we generate the data by discretizing the finite difference of $2D$ poisson problem on the square, as listed Listing~\ref{code2}. In remaining two examples, the generalized Lyapunov equation is drawn from a RAIL benchmark problem\footnote{The data are available at \href{https://www.cise.ufl.edu/research/sparse/matrices/list\_by\_id.html}{https://www.cise.ufl.edu/research/sparse/matrices/list\_by\_id.html} with different dimensions.} stemmed from a semidiscretized heat transfer problem for optimal cooling of steel profiles \cite{Benner2005ASH, Saak2004EfficientNS}. 

The elements of the initial iterate $Y_0$ are drawn from the standard normal distribution. In Algorithm~\ref{RieNew-TruncatedNewton}, the step sizes satisfy the Armijo-Goldstein conditions and are found by interpolation-based backtracking algorithm~\cite{DS1983}. The default parameters in ROPTLIB are used.
In Algorithm~\ref{FinalAlgPrecond-RLyap-RNewton}, the tolerance for the inner iteration is chosen adaptively by $\tau_p= \min(10^{-6}, r_p / 10)$. 

{\scriptsize\begin{lstlisting}[language=MATLAB,caption=the finite difference discretized $2D$ poisson problem on the square,label=code2]
	h = 1 / (n + 1);
	A = 1 / (h * h) * spdiags(ones(n, 1) * [-1 2 1], [-1 0 1], n, n)
	M = spdiags([rand(n - 1, 1); 0] + 0.1 , 0, n, n);
	c = randn(n, 1);
	C = c * c';
\end{lstlisting}}

\subsection{Comprehensive Comparisons of Algorithm~\ref{RieNew-TruncatedNewton}} \label{sec:ComAlg1}

% {\color{red}

In this section, we demonstrate the performance of the proposed Algorithm~\ref{RieNew-TruncatedNewton} compared to RTRNewton~\cite{Absil2007TrustRegionMO}, RCG~\cite{Boumal2015LowrankMC}, and LRBFGS~\cite{huang_riemannian_2018}, the gain of the proposed preconditioner in Section~\ref{sec:IncRankAlgPrecond:Precond} compared to the one in~\cite{Bart10}, as well as the superiority reformulating the problem on the quotient manifold $\mathbb{R}_*^{n\times p}/\mathcal{O}_p$ against the embedded submanifold $\mathcal{S}_+(p,n)$.  The testing data are generated according to Listing~\ref{code2}, and the results are reported in Figure~\ref{fig:NumExp:1},~Table~\ref{NumExp-table1} and~Table~\ref{NumExp-table2}.

\paragraph{Comparisons with Existing Riemannian Optimization Methods.}
Here we compared RNewton (i.e., Algorithm~\ref{RieNew-TruncatedNewton} combined with Algorithm~\ref{RieNew-TrunConjGrad}), RTRNewton, RCG, and LRBFGS methods, which are commonly used Riemannian smooth methods. The Lyapunov equations are generated according to Listing~\ref{code2}, and the comparisons are reported in Figure~\ref{fig:NumExp:1}, where the first row corresponds to the parameters $n=6000$ and $p=5$, and the second row corresponds to the parameters $n=8000$ and $p=6$. In Figure~\ref{fig:NumExp:1}, all methods apply the proposed preconditioner. 

\begin{figure}[htp]
\centering
% first figure
\begin{minipage}[t]{0.45\textwidth}
\centering
\includegraphics[width=6cm]{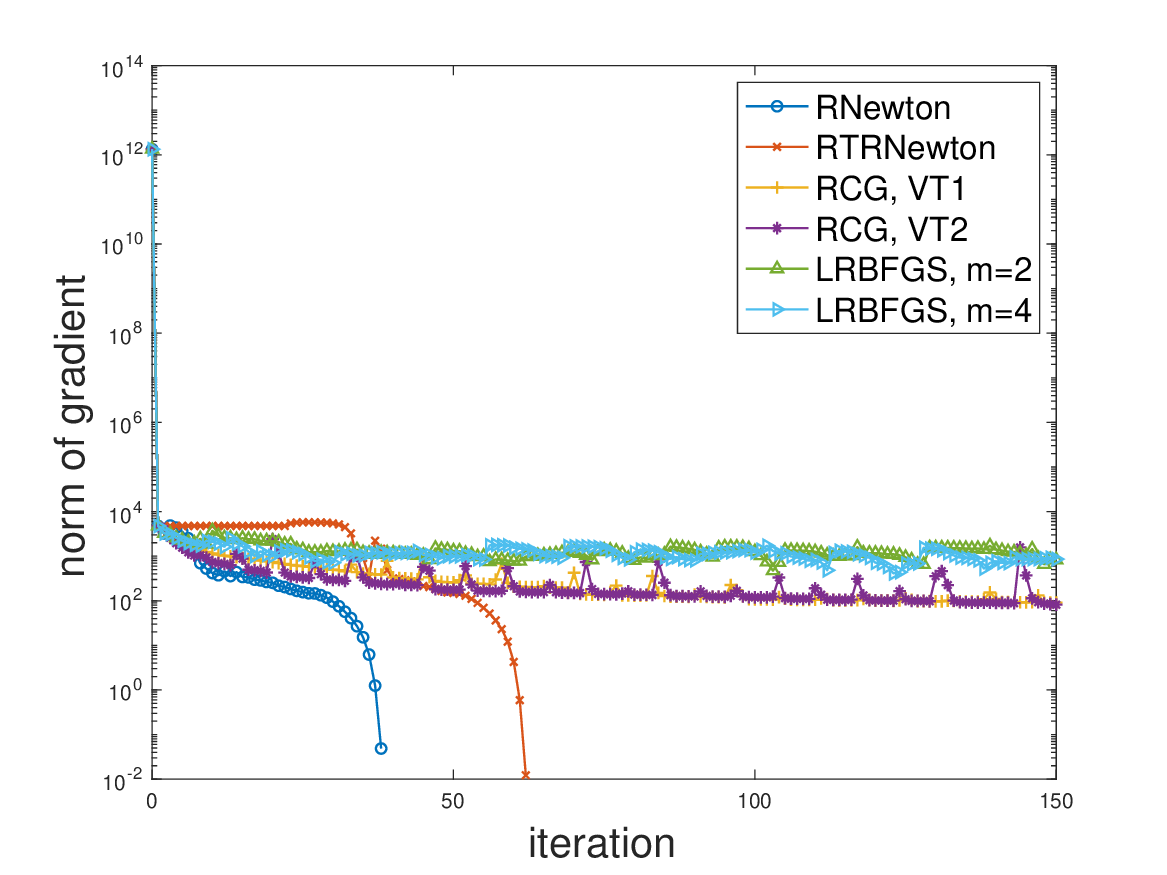}
\end{minipage}
% second figure
\begin{minipage}[t]{0.45\textwidth}
\centering
\includegraphics[width=6cm]{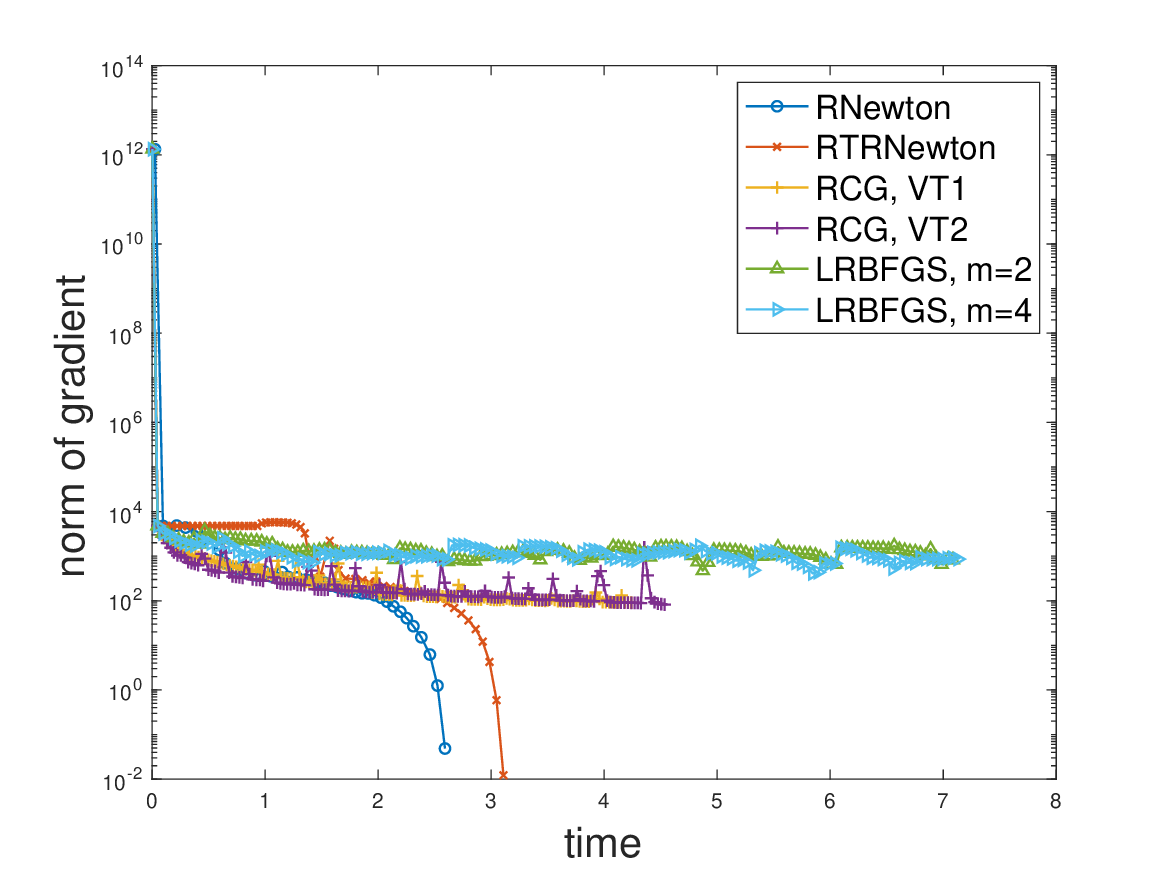}
\end{minipage}
% third figure
\begin{minipage}[t]{0.45\textwidth}
\centering
\includegraphics[width=6cm]{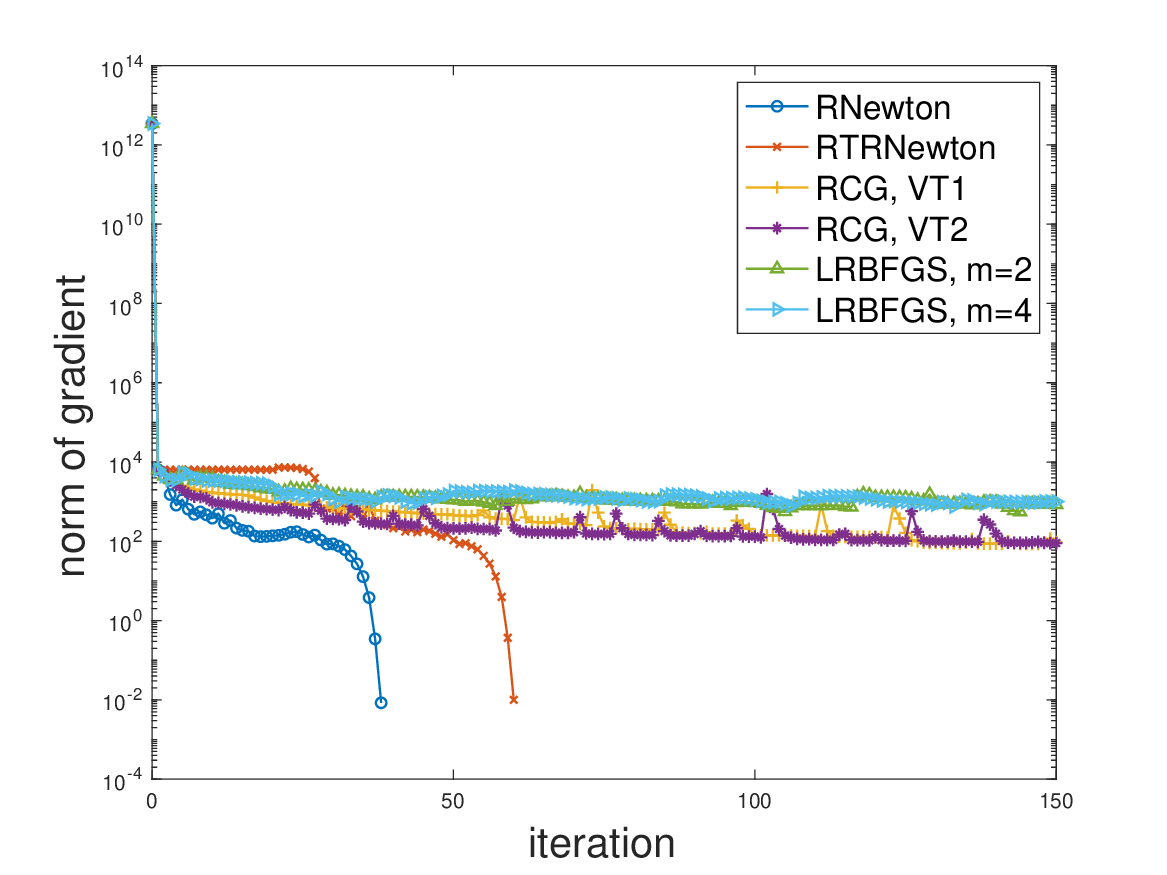}
\end{minipage}
% fourth figure
\begin{minipage}[t]{0.45\textwidth}
\centering
\includegraphics[width=6cm]{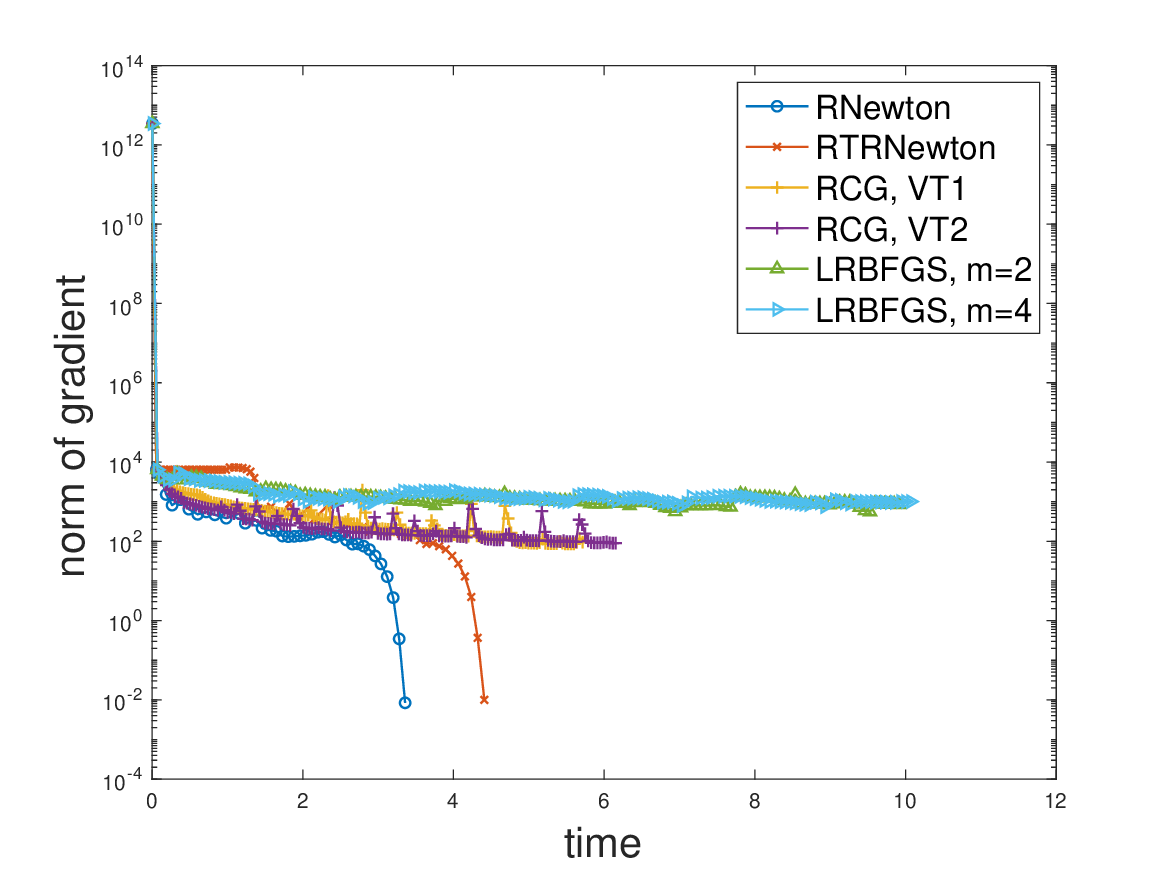}
\end{minipage}
\caption{Compare RNewton, RTRNewton, RCG, and LRBFGS methods. For RCG, VT1 and VT2 respectively means that the vector transport is given vector transport by parallelization and by projection. For LRBFGS, $m$ means memory size.}
\label{fig:NumExp:1}
\end{figure}

%We observe that RCG converges slowest compared with the other three methods in the two problems. Though LRBFGS converges rapidly in the first several iterations, it then fall into stagnation.
We observe that all methods under consideration declined sharply at the beginning, which is due to the use of the proposed preconditioner (it will be discussed in detail later). Then both RCG and LRBFGS slowly decrease.
RNewton and RTRNewton outperform the other two methods both in terms of the number of iterations and CPU time and they both exhibit superlinear convergence locally. While we note that RNewton is superior to RTRNewton in the number of iterations and CPU time. This shows the superiority of RNewton over other methods under consideration for solving Problem~\eqref{Pro_stat-FinalProb}.

\paragraph{The Proposed Preconditioner v.s. the Preconditioner in~\cite{Bart10}.}
We observe that from Table~\ref{NumExp-table1}, after applying preconditioners, the number of inner iterations is significantly reduced, which greatly improves the efficiency of the algorithm. 
Since the proposed preconditioner is a generalization of the one in~\cite{Bart10}, we should expect that when the mass matrix $M=I$, the gain of applying the two preconditioners is coincident due to that the preconditioner~\cite{Bart10} is derived under setting $M=I$. When $M\not=I$, due to that the preconditioner in~\cite{Bart10} does not consider the structure of $M$, it can be expected that the efficiency of using the preconditioner in~\cite{Bart10} is significantly lower than the efficiency of using the proposed preconditioner. These are verified by Table~\ref{NumExp-table1} and Table~\ref{NumExp-table2}. 

\begin{table}[htp] 
\caption{Comparison of RNewton and RTRNewton on $\mathbb{R}_*^{n\times p}/\mathcal{O}_p$ for solving Problem~\eqref{Pro_stat-FinalProb} with data from Listing~\ref{code2}. There exist two choices for the mass matrix $M$: i) $M=I$ or ii) $M=\mathrm{diag}([\mathrm{rand}(n-1,1);0]+0.1)$. Algorithm stops if $\|\mathrm{grad}f(x_i)\|/\|\mathrm{grad}f(x_0)\|<10^{-12}$. Indices ``iter'' and ``nH'' denote the number of iterations and the number of action of Hessian evaluations, respectively. Besides, the number $a.bc_{-k}$ denotes $a.bc\times 10^{-k}$.}
\begin{minipage}[ht]{0.25\textwidth}
\resizebox{16cm}{!}{
\begin{tabular}{lllllllll}
%\hline
\toprule
\multicolumn{3}{l}{\multirow{2}{*}{}}                                                                           & \multicolumn{3}{c}{RNewton}                                                     & \multicolumn{3}{c}{RTRNewton}                                                   \\
\cmidrule(rl){4-6} \cmidrule(rl){7-9} 
\multicolumn{3}{l}{}                                                                                            & no precond.      & proposed precond.     & precond~\cite{Bart10} & no precond      & proposed precond.     & precond.~\cite{Bart10} \\ \hline
\multirow{8}{*}{$M=I$}     & \multirow{2}{*}{\begin{tabular}[c]{@{}l@{}}$n=4000$\\ $p=3$\end{tabular}}   &     iter &$    3.80_{1}$ &$    3.00_{1}$ &$    3.10_{1}$ &$    6.70_{1}$ &$    5.50_{1}$ &$    5.50_{1}$  \\  
                           &                                                                             &      nH &$    8.01_{3}$ &$    5.50_{1}$ &$    5.60_{1}$ &$    7.68_{3}$ &$    1.30_{2}$ &$    1.30_{2}$  \\   \cline{2-9} 
                           & \multirow{2}{*}{\begin{tabular}[c]{@{}l@{}}$n=8000$\\ $p=5$\end{tabular}}   &       iter &$    5.70_{1}$ &$    3.30_{1}$ &$    3.30_{1}$ &$    9.20_{1}$ &$    6.00_{1}$ &$    6.00_{1}$  \\  
                           &                                                                             &         nH &$    2.44_{4}$ &$    5.20_{1}$ &$    5.00_{1}$ &$    2.18_{4}$ &$    1.34_{2}$ &$    1.34_{2}$  \\    \cline{2-9} 
                           & \multirow{2}{*}{\begin{tabular}[c]{@{}l@{}}$n=20000$\\ $p=8$\end{tabular}}  &         iter &$    7.00_{1}$ &$    3.60_{1}$ &$    3.80_{1}$ &$    9.60_{1}$ &$    5.80_{1}$ &$    6.00_{1}$  \\   
                           &                                                                             &         nH &$    2.96_{4}$ &$    5.60_{1}$ &$    5.70_{1}$ &$    2.53_{4}$ &$    1.26_{2}$ &$    1.29_{2}$  \\    \cline{2-9} 
                           & \multirow{2}{*}{\begin{tabular}[c]{@{}l@{}}$n=40000$\\ $p=10$\end{tabular}} &        iter &$    7.90_{1}$ &$        9.00$ &$        9.00$ &$    1.24_{2}$ &$    4.20_{1}$ &$    4.50_{1}$  \\  
                           &                                                                             &          nH &$    2.70_{4}$ &$    1.50_{1}$ &$    1.30_{1}$ &$    2.38_{4}$ &$    8.30_{1}$ &$    9.10_{1}$  \\   \hline
\multirow{8}{*}{$M\not=I$} & \multirow{2}{*}{\begin{tabular}[c]{@{}l@{}}$n=4000$\\ $p=3$\end{tabular}}   &        iter &$    7.20_{1}$ &$    2.90_{1}$ &$    3.90_{1}$ &$    7.70_{1}$ &$    4.90_{1}$ &$    5.40_{1}$  \\  
                           &                                                                             &         nH &$    2.10_{4}$ &$    6.10_{1}$ &$    1.31_{2}$ &$    1.53_{4}$ &$    1.19_{2}$ &$    2.60_{2}$  \\   \cline{2-9} 
                           & \multirow{2}{*}{\begin{tabular}[c]{@{}l@{}}$n=8000$\\ $p=5$\end{tabular}}   &        iter &$    1.07_{2}$ &$    4.00_{1}$ &$    4.40_{1}$ &$    1.22_{2}$ &$    5.90_{1}$ &$    6.40_{1}$  \\   
                           &                                                                             &          nH &$    4.01_{4}$ &$    8.90_{1}$ &$    1.56_{2}$ &$    3.65_{4}$ &$    1.47_{2}$ &$    3.04_{2}$  \\    \cline{2-9} 
                           & \multirow{2}{*}{\begin{tabular}[c]{@{}l@{}}$n=20000$\\ $p=8$\end{tabular}}  &        iter &$    7.10_{1}$ &$    5.00_{1}$ &$    6.80_{1}$ &$    1.17_{2}$ &$    7.70_{1}$ &$    8.40_{1}$  \\   
                           &                                                                             &          nH &$    2.84_{4}$ &$    1.22_{2}$ &$    2.50_{2}$ &$    3.11_{4}$ &$    2.12_{2}$ &$    3.57_{2}$  \\    \cline{2-9} 
                           & \multirow{2}{*}{\begin{tabular}[c]{@{}l@{}}$n=40000$\\ $p=10$\end{tabular}} &        iter &$    7.40_{1}$ &$    2.00_{1}$ &$    6.30_{1}$ &$    1.38_{2}$ &$    5.00_{1}$ &$    6.20_{1}$  \\  
                           &                                                                             &          nH &$    2.86_{4}$ &$    4.40_{1}$ &$    2.27_{2}$ &$    1.66_{4}$ &$    1.25_{2}$ &$    2.41_{2}$  \\  
%\hline
\bottomrule
\end{tabular}
}
\end{minipage}
\label{NumExp-table1}
\end{table}

\paragraph{Quotient Manifold \texorpdfstring{$\mathbb{R}_*^{n\times p}/\mathcal{O}_p$}{} v.s. Embedded Manifold \texorpdfstring{$\mathcal{S}_+(p,n)$}{}.}
Applying the preconditioner proposed in Section~\ref{sec:IncRankAlgPrecond:Precond}, RNewton and RTRNewton have a remarkable advantage for solving Problem~\eqref{Pro_stat-FinalProb} over them for solving Problem~\eqref{Pro_stat-OptProb_Bart}, which can be observed from Table~\ref{NumExp-table1} and Table~\ref{NumExp-table2}, in terms of the number of iterations and the number of the actions of Hessian.
Theoretically, the quotient manifold $\mathbb{R}_*^{n\times p}/\mathcal{O}_p$ is diffeomorphic to the embedded manifold $\mathcal{S}_+(p,n)$, but there exists significant difference lying in optimization. The main difference is the choice of retraction. For $\mathcal{S}_+(p,n)$, a projection-type retraction is used (see~\cite{Bart10}), while an addition-type retraction~\eqref{retraction} is used for $\mathbb{R}_*^{n\times p}/\mathcal{O}_p$. In most cases, the number of inner iterations (i.e., Algorithm~\ref{RieNew-TrunConjGrad} for quotient manifold) is very small, generally 0 or 1. That is, Algorithm~\ref{RieNew-TrunConjGrad} returns the approximate solution to Equation~\eqref{Precond-Metric1-ApproxNewEqu} with $-(\mathrm{grad}f(\pi(Y_k)))_{\uparrow_{Y_k}}$ as the right-hand side. Specifically, the solution $\xi_{\uparrow_{Y_k}}$ satisfies 
\[
	\left( I - \frac{1}{2}P_{Y_k} \right)\nabla^2h(Y_kY_k^T)[Y_k\xi_{\uparrow_{Y_k}}^T+\xi_{\uparrow_{Y_k}}Y_k^T]Y_k(Y_k^TY_k)^{-1} \approx -\left( I - \frac{1}{2}P_{Y_k} \right)\nabla h(Y_kY_k^T)Y_k(Y_k^TY_k)^{-1}.
\]
This is equivalent to 
\[
	(\nabla^2h(Y_kY_k^T)[Y_k\xi_{\uparrow_{Y_k}}^T+\xi_{\uparrow_{Y_k}}Y_k^T]+\nabla h(Y_kY_k^T))Y_k\approx0,
\]
this is, 
\[
	(A(Y_k\xi_{\uparrow_{Y_k}}^T+\xi_{\uparrow_{Y_k}}Y_k^T + Y_kY_k^T)M+M(Y_k\xi_{\uparrow_{Y_k}}^T+\xi_{\uparrow_{Y_k}}Y_k^T+YY^T)A-C)Y_k\approx0,
\]
which implies
\[
	(A(Y_k+\xi_{\uparrow_{Y_k}})(Y_k+\xi_{\uparrow_{Y_k}})^TM+M(Y_k+\xi_{\uparrow_{Y_k}})(Y_k+\xi_{\uparrow_{Y_k}})^TA-C)Y_k\approx(A\xi_{\uparrow_{Y_k}}\xi_{\uparrow_{Y_k}}^TM+M\xi_{\uparrow_{Y_k}}\xi_{\uparrow_{Y_k}}^TA)Y_k.
\]
For a randomly given initial guess $Y_0\in \mathbb{R}_*^{n\times p}$, we observe that $\|A\xi_{\uparrow_{Y_k}}\xi_{\uparrow_{Y_k}}^TM+M\xi_{\uparrow_{Y_k}}\xi_{\uparrow_{Y_k}}^TA\|_F$ is far less than $\|AY_0Y_0^TM+MY_0Y_0^TA-C\|_F$ from Table~\ref{NumExp-table3}. That is to say, $Y+\eta_{\uparrow_Y}$ is a suitable candidate for the solution of~\eqref{Pro_stat-FinalProb} in the column space of $Y_0$. This also partly explains why it is better to define the problem on the quotient manifold $\mathbb{R}_*^{n\times p}/\mathcal{O}_p$ than on the embedded submanifold $\mathcal{S}_+(p,n)$.  

\begin{table}[H]
\centering
\caption{Comparisons of RNewton and RTRNewton on $\mathcal{S}_+(p,n)$ for solving Problem~\eqref{Pro_stat-OptProb_Bart} and on $\mathbb{R}_*^{n\times p}/\mathcal{O}_p$ for solving Problem~\eqref{Pro_stat-FinalProb} with data from Listing~\ref{code2}. The mass matrix is given by $M=\mathrm{diag}([\mathrm{rand}(n-1,1);0]+0.1)$. Algorithms stop if $\|\mathrm{grad}f(x_i)\|/\|\mathrm{grad}f(x_0)\|<10^{-11}$.}
\resizebox{16cm}{!}{
\begin{tabular}{lllllll}
%\hline
\toprule
\multicolumn{3}{l}{\multirow{2}{*}{}}                                                                           & \multicolumn{2}{c}{RNewton}                                                     & \multicolumn{2}{c}{RTRNewton}                                                   \\
\cmidrule(rl){4-5} \cmidrule(rl){6-7} 
\multicolumn{3}{l}{}                                                                                            & proposed precond.     & precond.~\cite{Bart10} &  proposed precond.     & precond.~\cite{Bart10} \\ \hline
\multirow{6}{*}{$\mathcal{S}_+(p,n)$} & \multirow{2}{*}{\begin{tabular}[c]{@{}l@{}}$n=20000$\\ $p=6$\end{tabular}}  &             iter &$    6.90_{1}$ &$    9.30_{1}$ &$    1.21_{2}$ &$    1.41_{2}$   \\      
                           &                                                                             &      nH &$    1.38_{2}$ &$    2.89_{2}$ &$    2.57_{2}$ &$    4.27_{2}$   \\     \cline{2-7} 
& \multirow{2}{*}{\begin{tabular}[c]{@{}l@{}}$n=50000$\\ $p=12$\end{tabular}}   &         iter &$    5.50_{1}$ &$    6.80_{1}$ &$    1.15_{2}$ &$    1.26_{2}$   \\          
                           &                                                                             &          nH &$    1.01_{2}$ &$    1.58_{2}$ &$    2.45_{2}$ &$    3.67_{2}$   \\          \cline{2-7} 
                           & \multirow{2}{*}{\begin{tabular}[c]{@{}l@{}}$n=80000$\\ $p=18$\end{tabular}}  &     iter &$    6.50_{1}$ &$    1.04_{2}$ &$    1.04_{2}$ &$    1.23_{2}$   \\    
                           &                                                                             &       nH &$    8.50_{1}$ &$    2.08_{2}$ &$    2.11_{2}$ &$    3.57_{2}$   \\  \hline
\multirow{6}{*}{$\mathbb{R}_*^{n\times p}/\mathcal{O}_p$} & \multirow{2}{*}{\begin{tabular}[c]{@{}l@{}}$n=20000$\\ $p=6$\end{tabular}}  &          iter &$    2.10_{1}$ &$    3.60_{1}$ &$    5.40_{1}$ &$    6.00_{1}$   \\      
                           &                                                                             &          nH &$    5.70_{1}$ &$    1.21_{2}$ &$    1.23_{2}$ &$    2.76_{2}$   \\      \cline{2-7} 
                           & \multirow{2}{*}{\begin{tabular}[c]{@{}l@{}}$n=50000$\\ $p=12$\end{tabular}}   &       iter &$        4.00$ &$    1.70_{1}$ &$    4.10_{1}$ &$    4.10_{1}$   \\         
                           &                                                                             &          nH &$        6.00$ &$    4.90_{1}$ &$    8.20_{1}$ &$    1.39_{2}$   \\        \cline{2-7} 
                           & \multirow{2}{*}{\begin{tabular}[c]{@{}l@{}}$n=80000$\\ $p=18$\end{tabular}}  &       iter &$        2.00$ &$    1.20_{1}$ &$    3.70_{1}$ &$    3.70_{1}$   \\    
                           &                                                                             &        nH &$        2.00$ &$    4.20_{1}$ &$    7.30_{1}$ &$    1.88_{2}$   \\ 
\bottomrule
\end{tabular}
}
\label{NumExp-table2}
\end{table}

\begin{table}[htp]
\centering
\caption{The data and the parameters are the same as Table~\ref{NumExp-table1}. Here approx\_res = $\|A\xi_{\uparrow_{Y_k}}\xi_{\uparrow_{Y_k}}^TM+M\xi_{\uparrow_{Y_k}}\xi_{\uparrow_{Y_k}}^TA\|_F / \|AY_0Y_0^T+MY_0Y_0^TA-C\|_F$.} 
\resizebox{16cm}{!}{
\begin{tabular}{lllllll}
%\hline
\toprule
\multicolumn{3}{l}{\multirow{2}{*}{}}                                                                           & \multicolumn{2}{c}{RNewton}                                                     & \multicolumn{2}{c}{RTRNewton}                                                   \\
\cmidrule(rl){4-5} \cmidrule(rl){6-7} 
\multicolumn{3}{l}{}                                                                                            & proposed precond.     & precond.~\cite{Bart10} &  proposed precond.     & precond.~\cite{Bart10} \\ \hline
\multirow{3}{*}{$\mathcal{S}_+(p,n)$} & \multirow{1}{*}{\begin{tabular}[c]{@{}l@{}}$(n,p)=(20000,6)$\end{tabular}}  &          approx\_res &$  2.32_{-12}$ &$  4.32_{-14}$ &$  8.26_{-12}$ &$  4.12_{-14}$   \\      \cline{2-7} 
& \multirow{1}{*}{\begin{tabular}[c]{@{}l@{}}$(n,p)=(50000,12)$\end{tabular}}   &  approx\_res &$   3.34_{-7}$ &$  1.13_{-10}$ &$  3.09_{-11}$ &$  8.11_{-11}$   \\          \cline{2-7} 
& \multirow{1}{*}{\begin{tabular}[c]{@{}l@{}}$(n,p)=(80000,18)$\end{tabular}}  &   approx\_res &$   8.07_{-9}$ &$  1.95_{-11}$ &$  8.39_{-11}$ &$  1.30_{-11}$   \\  
\bottomrule
\end{tabular}
}
\label{NumExp-table3}
\end{table}

% }

\subsection{Comparison with Existing Low-Rank Methods} \label{NumExp-CompWithExistingSols}

In this section, Algorithm~\ref{FinalAlgPrecond-RLyap-RNewton}, called IRRLayp, is compared with three existing state-of-the-art low-rank methods for Lyapunov equations, including K-PIK from~\cite{simoncini_new_2007}, RKSM from~\cite{DS11,KS20}, and mess\_lradi from a Matlab toolbox named M-M.E.S.S~\cite{SaaKB21-mmess-2.2}. 
These methods are respectively based on Krylov subspace techniques (K-PIK, RKSM), and alternating direction implicit iterative (mess\_lradi). For mess\_lradi, there are three methods for selecting shifts, ``heuristic'', ``wachspress'', and ``projection'', which are all tested in our experiments. Since, in Algorithm~\ref{FinalAlgPrecond-RLyap-RNewton}, other Riemannian methods can be used to solve the subproblem, we further use the Riemannian trust-region Newton's method in~\cite{Absil2007TrustRegionMO} and the resulting algorithm is denoted by IRRLyap-RTRNewton. Algorithm~\ref{FinalAlgPrecond-RLyap-RNewton} combined with Algorithm~\ref{RieNew-TruncatedNewton} is denoted by IRRLyap-RNewton. As a reference, the notation ``best low rank'' with rank $p$ denotes a best low-rank approximation by the truncated singular value decomposition to the exact solution of~\eqref{Intro-LyapMatEqu}.

\paragraph{Quality of low-rank solutions.}
The generalized Lyapunov equation was drawn from a RAIL benchmark problem with the coefficient matrix of size $n=1357$. For the sake of simplicity, the right-hand side is taken as $C = B(:,1)B(:,1)^T$. Figure \ref{NumExp-QualityofLowRank} shows the experiment results.

\begin{figure}[htbp]
	\centering
	\includegraphics[scale=0.3]{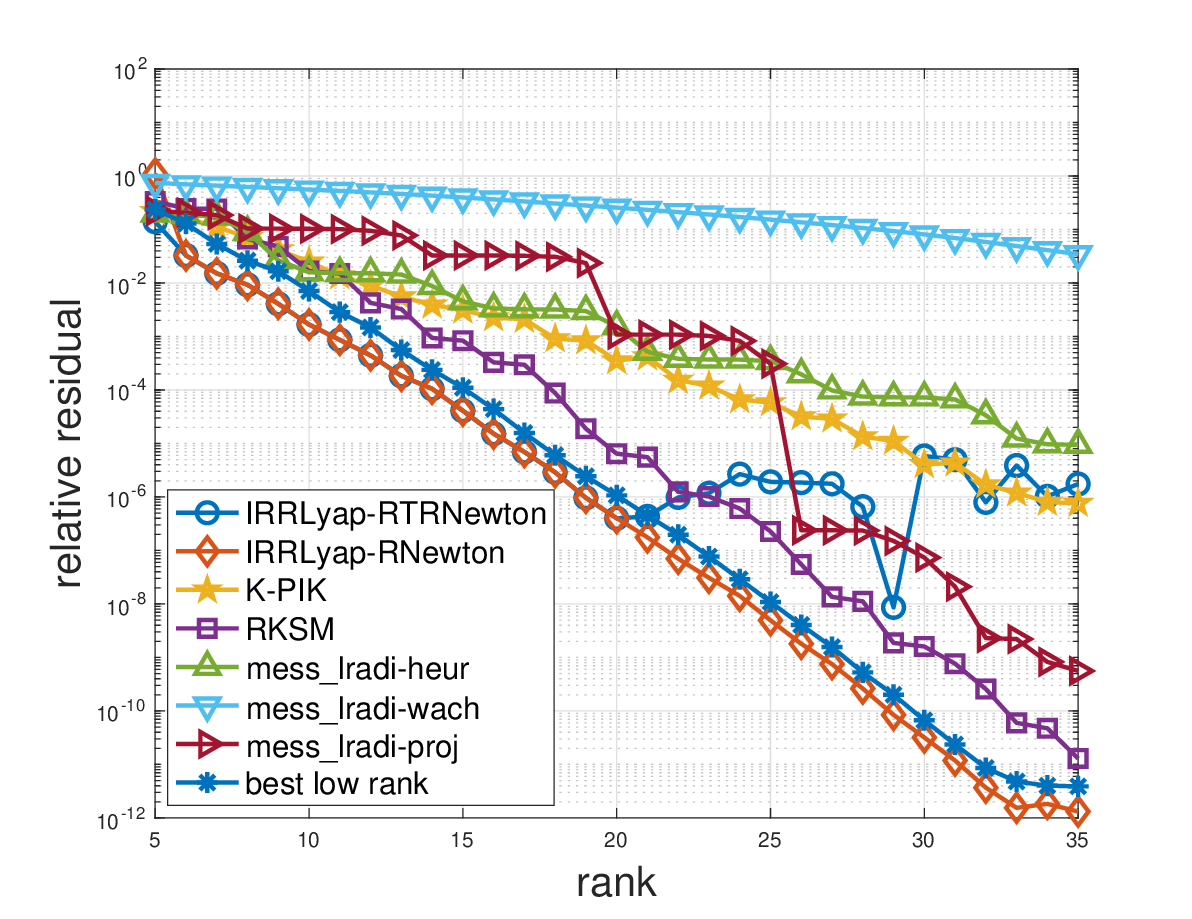}
	\caption{The relative residual for one-rank right-hand-side RAIL benchmark with $n=1357$.} 
	\label{NumExp-QualityofLowRank}
\end{figure}

%\whcomm{[ZHTODO: update the following paragraph. Be specific and precise.]}{}

%Let us compare the performance of the tested methods with the reference method--``best low rank''. 
It can be seen from Figure~\ref{NumExp-QualityofLowRank} that 
under the same rank, the relative residuals of the low-rank approximations from K-PIK and mess\_lradi are notably greater than those from ``best low rank''. Therefore, mess\_lradi and K-PIK are not preferred if lower-rank solutions are desired.
In addition, selecting shifts method by projection performs best for mess\_lradi. 
When the relative residual is greater than $10^{-6}$, that is, highly accurate solutions are not needed, the solutions from IRRLyap-RTRNewton and IRRLyap-RNewton are close in the sense of the relative residual. This is because IRRLyap-RTRNewton and IRRLyap-RNewton solve the same optimization problem. However, when the relative residual is smaller than $10^{-6}$, the relative residuals from IRRLyap-RTRNewton begin to fluctuate without falling. We find that for this problem, RTRNewton is more sensitive to numerical error than RNewton in the sense that RTRNewon sometimes terminates before reaching the stopping criterion. Such behavior of RTRNewton prevents it from finding highly accurate solutions. Therefore, we conclude that IRRLyap-RNewton is preferable compared to IRRLyap-RTRNewton.
The solutions found by IRRLyap-RNewton and RKSM are close to the ones by ``best low rank''.
IRRLyap-RNewton even have smaller relative residuals than those from ``best low rank''. This is not surprising since ``best low rank'' uses truncated SVD and completely ignore the original problem whereas IRRLyap-RNewton aims to find a stationary point, i.e., minimizes the Riemannian gradient---the residual in the horizontal space, see~Proposition~\ref{IngredQuotMani-Prop-Gradient}. Overall, IRRLyap-RNewton is able to find the best solution compared to the tested methods in the sense of the relative residual.

%the performance of IRRLyap(RNewton) is similar to that of ``best low rank'', even slightly better than that of ``best low rank''.

%One of the reasons is that in each iteration IRRLyap(RNewton) solves a fixed rank subproblem which make the relative residual smaller. It is worth noting that in IRRLyap(RNewtom) there is no high requirement for accuracy in the inner subproblem. For instance, $\tau_p=10^{-2}$ is used in this case. 

%We observe, in Figure \ref{NumExp-QualityofLowRank}, that in each step the low-rank approximation from K-PIK are far way from optimal. However, the numerical results of mess\_lradi are between IRRLyap(RNewton) and K-PIK and it seems to be numerically unstable. For lower accuracy, e.g., $\tau \approx 10^{-6}$, the approximations by IRRLyap(RTRNewton) and IRRLyap(RNewton) are very close in a sense of the relative residual. However, as the accuracy is increasing, because of numerical error IRRLyap(RTRNewton) begins to fluctuate, which is one of the reasons why we use Newton's method in Algorithm~\ref{RieNew-TruncatedNewton} for the fixed-rank subproblem(i.e., solving Problem \ref{Pro_stat-FinalProb}). IRRLyap(RNewton) achieves the same accuracy ($\tau \approx 10^{-12}$) as "best low rank". Even RLyap-RNewton can reach lower relative residual than ``best low rank'', in that in each iteration IRRLyap(RNewton) solves a fixed rank subproblem which make the relative residual smaller. It is worth noting that the in IRRLyap(RNewton) there is no high requirements for accuracy in the inner subproblem. For instance, $\tau_p=10^{-2}$ is used in this case. 

\paragraph{Efficiency and performance.}
K-PIK, RKSM, mess\_lradi, IRRLyap-RNewton, and lradi-RNewton (i.e., mess\_lradi provides a coarse solution that is used to be an initial guess of RNewton) are compared by using the RAIL benchmark problems with size $n=5177,20209,79841$. The stopping criterion for all methods are unified as $\|R\|_F=\|AXM+MXA - C\|_F  \le \tau \cdot \|C\|_F$ with a tolerance $\tau = 10^{-6}$. 
The results are reported in Table~\ref{NumExp-Table4}.

\begin{table}[htbp]
\caption{Comparison for the simplified RAIL benchmark with existing methods. ``rank'', ``time'', ``rel\_res'', ``numSys'', and ``timeSys'' respectively denote the rank of the approximation, running time, the relative residual of the approximation, the number of solving shift systems $(A+\lambda M)X=b$ for $X$ with given $A,\lambda, M$ and $b$ (column vector or matrix, the numbers in parentheses indicate the number of matrices), and the CPU time for solving all those shift systems. Besides, the number $a.bc_{-k}$ denotes $a.bc\times 10^{-k}$.}
\centering
\setlength{\tabcolsep}{3.5pt}
\resizebox{16.5cm}{!}{
\begin{tabular}{llllllllllllllll}
\toprule
\multirow{2}{*}{} & rannk & time(s.)     & rel\_res     & numSys & timeSys  & rannk & time(s.)     & rel\_res     & numSys & timeSys  & rannk & time(s.)     & rel\_res     & numSys & timeSys  \\
\cmidrule(rl){2-6} \cmidrule(rl){7-11} \cmidrule(rl){12-16}
                  & \multicolumn{5}{c}{5177}                                        & \multicolumn{5}{c}{20209}                                       & \multicolumn{5}{c}{79841}                                       \\
\midrule                  
K-PIK             & 126   & $      4.61$ & $ 1.46_{-6}$ & 64         & $      1.44$ & 182   & $  4.78_{1}$ & $ 2.65_{-6}$ & 92         & $      7.47$ & 248   & $  5.34_{2}$ & $ 4.39_{-6}$ & 125        & $  6.36_{1}$ \\
RKSM              & 27    & $      1.94$ & $ 7.07_{-7}$ & 28         & $ 1.67_{-1}$ & 38    & $  2.66_{1}$ & $ 2.65_{-6}$ & 39         & $      1.31$ & 41    & $  2.70_{2}$ & $ 4.39_{-6}$ & 42         & $      7.79$ \\
mess\_lradi       & 32    & $ 3.58_{-1}$ & $ 1.47_{-7}$ & 64         & $          $ & 37    & $      1.12$ & $ 5.90_{-7}$ & 74         & $          $ & 38    & $      4.14$ & $ 6.12_{-8}$ & 76         & $          $ \\
lradi-RNewton     & 26    & $      2.37$ & $ 6.31_{-7}$ & 8268(156)       & $      1.29$ & 32    & $      7.56$ & $ 2.28_{-7}$ & 6240(96)       & $      4.61$ & 32    & $  7.26_{1}$ & $ 4.13_{-6}$ & 14560(224)      & $  4.67_{1}$ \\
IRRlyap-RNewton     & 22    & $      5.78$ & $ 7.12_{-7}$ & 16095(521)      & $      3.51$ & 27    & $  4.44_{1}$ & $ 3.29_{-7}$ & 28410(822)      & $  3.04_{1}$ & 27    & $  2.71_{2}$ & $ 5.21_{-7}$ & 36950(1038)      & $  1.79_{2}$ \\
\bottomrule
\end{tabular}
}
\label{NumExp-Table4}
\end{table}

% {\color{red}
For the three Lyapunov equations, K-PIK fails to satisfy the stopping criterion, implying that K-PIK has difficulty finding highly accurate solutions. Moreover, K-PIK needs to use a higher rank to reach a similar residual compared to other methods. Therefore, K-PIK is not preferred. 
The method mess\_lradi is the most efficient algorithm in the sense of computational time. However, it also requires a higher rank for similar accuracy when compared to lradi\_RNewton and IRRLyap-RNewton. Thus, mess\_lradi is preferred if one has strict requirements on efficiency but not on rank. In view of this, noting the fact that mess\_lradi can rapid provides a coarse solution and the fact that RNewton can refine the coarse solution,  we integrate mess\_lradi with RNewton. The resulting method, lradi\_RNewton, requires a relative lower rank for similar accuracy while reduces significantly computational time. Additionally, the method RKSM also requires a higher rank for similar accuracy compared to IRRLya-RNewton, and for two medium to large-scale problems, it seems to be difficult to find highly accurate solutions. However, the proposed method IRRLyap-RNewton gives the solutions with lowest rank compared to the rest methods. We can observe from Table~\ref{NumExp-Table4} that most of the time in lradi\_RNewton and IRRLyap-RNewton are taken to solve the shift systems. Therefore, for some problems, if the resulting shift systems can be solved more efficiently, then the efficiency of lradi\_RNewton and IRRLyap-RNewton can be further improved. Overall, we suggest using IRRLyap-RNewton if one does not have strict restrictions on computational time and desires as low-rank solutions as possible.

\section{Conclusions} \label{Concl}

In this paper, we have generalized the truncated Newton's method from Euclidean spaces to Riemannian manifolds, called Riemannian truncated-Newton's method, and shown the convergence results, e.g., global convergence and local superlinear convergence. Moreover, the cost function from~\cite{Bart10} is reformulated as an optimization problem defined on the Riemannian quotient manifold $\mathbb{R}^{n\times p}/\mathcal{O}_p$. An algorithm, called IRRLyap-RNewton, is proposed and is used to find a low-rank approximation of the generalized Lyapunov equations. We develop a new preconditioner that take $M \neq I$ into consideration. The numerical results show that the proposed RNewton has superior advangate over RTRNewton, RCG, and LRBGFS for solving Problem~\eqref{Pro_stat-FinalProb}, and that
 the new preconditioner significantly reduce the number of actions of Riemannian Hessian evaluations even when $M \neq I$. In addition, IRRLyap-RNewton is able to find a similar accurate solution with the lowest rank compared to some state-of-the-art methods, including K-PIK, and mess\_lradi. 
%On the other hand, even though RLyap-RNewton performs very well with preconditioning, there is still the possibility of further speedup. For example, efficiently solving the shifted systems will significantly reduce the running time.

\normalem
\bibliographystyle{alpha}
\bibliography{references}

\appendix
\section{Appendix} \label{Appendix}

\subsection{The Other Two Metrics and Corresponding Preconditioners.}

\paragraph{Riemannian metrics on $\mathbb{R}_*^{n\times p}$.} Two other Riemannian metrics on $\mathbb{R}_*^{n\times p}$ given in~\cite{Zheng2022RiemannianOU} also are considered here: 
\begin{equation} \label{Rmetrics}
g^{i}_Y(\eta_Y,\xi_Y) = 
\begin{cases}
	\mathrm{tr}(Y^TY\eta_Y^T\xi_Y) & i = 2,  \\
	\mathrm{tr}(\eta_Y^T\xi_Y)  & i = 3,  
\end{cases}
\end{equation}
the Riemannian metric $g_Y^2$ has been used in \cite{HUANG2017}, where the space under consideration is on the complex field; and the Riemannian metric $g_Y^3$ is the standard Euclidean inner product on $\mathbb{R}^{n\times p}$ and has been considered in~\cite{MA20}.

\paragraph{Horizontal spaces.} 
Given a Riemannian metric $g$ on the total space $\mathbb{R}_*^{n \times p}$, the orthogonal complement space in $\mathrm{T}_Y\mathbb{R}_*^{n\times p}$ of $\mathrm{V}_Y$ with respect to $g_Y$ is called the horizontal space at $Y$, denoted by $\mathrm{H}_Y$. The horizontal spaces with respect to the three Riemannian metrics in~\eqref{Rmetrics} are respectively given by
\begin{equation} \nonumber
	\mathrm{H}_Y^i = \begin{cases}		 
		\{YS+Y_\perp K:S^T=S,S\in\mathrm{R}^{p\times p},K\in\mathbb{R}^{(n-p)\times p}\} & i=2, \\
		\{Y(Y^TY)^{-1}S+Y_\perp K:S^T=S,S\in\mathbb{R}^{p\times p},K\in\mathbb{R}^{(n-p)\times p}\} & i=3.
	\end{cases}
\end{equation}

\paragraph{Projections onto Vertical Spaces and Horizontal Spaces.}
For any $Y\in \mathbb{R}_*^{n\times p}$ and $\eta_Y\in \mathrm{T}_Y\mathbb{R}_*^{n\times p}$, the orthogonal projections of $\eta_Y$ to $\mathrm{V}_Y$ and $\mathrm{H}_Y^i$ with respect to the second metric are respectively given by
	$$
	\begin{aligned} 
	\mathcal{P}_Y^{\mathrm{V}^2}(\eta_Y) = Y\Omega, \text{ and } \mathcal{P}_Y^{\mathrm{H}^2}(\eta_Y)&=\eta_Y-\mathcal{P}_Y^{\mathrm{V}^{2}}(\eta_Y)=\eta_Y-Y\Omega \\ &= Y\left( \frac{(Y^TY)^{-1}Y^T\eta_Y+\eta_Y^TY(Y^TY)^{-1}}{2} \right) + Y_\perp Y_\perp^T\eta_Y, 
    \end{aligned}
	$$
	where $\Omega = \frac{(Y^TY)^{-1}Y^T\eta_Y-\eta_Y^TY(Y^TY)^{-1}}{2}$. The orthogonal projections to vertical and horizontal spaces with respect to the third metric are 
    $$ 
    \mathcal{P}_Y^{\mathrm{V}^3}(\eta_Y)=Y\Omega, \text{ and } \mathcal{P}_Y^{\mathrm{H}^3}(\eta_Y)=\eta_Y-\mathcal{P}_Y^{\mathrm{V}^3}(\eta_Y)=\eta_Y-Y\Omega,
    $$
    where $\Omega$ is the skew-symmetric matrix satisfying $ \Omega Y^TY +Y^TY\Omega=Y^T\eta_Y-\eta^T_YY.$
  
\paragraph{Riemannian metrics on $\mathbb{R}_*^{n\times p}/\mathcal{O}_p$.}
Since the two Riemannian metrics $g_Y^2,g_Y^3$ on $\mathbb{R}_*^{n\times p}$ satisfy~\eqref{IngredQuotMani-Metric-induced} by~\cite{Zheng2022RiemannianOU}, the corresponding three metrics on $\mathbb{R}_*^{n\times p}/\mathcal{O}_p$ are given by:  
\begin{subnumcases} { g^i_{\pi(Y)}(\xi_{\pi(Y)},\eta_{\pi(Y)})=}
\mathrm{tr}(Y^TY\xi_{\uparrow_Y}^T\eta_{\uparrow_Y}) & $i = 2,$ \label{IngredQuotMani-Metric2} \\
\mathrm{tr}(\xi_{\uparrow_Y}^T\eta_{\uparrow_Y})  & $i =3,$ \label{IngredQuotMani-Metric3}
\end{subnumcases}
for all $\xi_{\pi(Y)},\eta_{\pi(Y)}\in\mathrm{T}_{\pi(Y)}\mathbb{R}_*^{n\times p}/\mathcal{O}_p$.
In the following, with a slight abuse of notation, we use $g^i$, $i=2,3$, to denote the Riemannian metrics on both $\mathbb{R}_*^{n\times p}$ and $\mathbb{R}_*^{n\times p}/\mathcal{O}_p$.

\paragraph{The horizontal lifts of Riemannian Gradients and the actions of Riemannian Hessians.} 
It follows from~\cite{Zheng2022RiemannianOU} that the Riemannian gradients and Riemannian Hessian of $f$ in~\eqref{Pro_stat-FinalProb} can be characterized by the Euclidean gradient and the Euclidean Hessian of $h$ in~\eqref{Pro_stat-OptProb_Bart}, see Proposition~\ref{IngredQuotMani-Prop-Gradient}.
\begin{proposition} \label{IngredQuotMani-Prop-Grad-Hess23}
  The Riemannian gradients of the smooth real-valued function $f$ in \eqref{Pro_stat-FinalProb} on $\mathbb{R}_*^{n\times p}/\mathcal{O}_p$ with respect to the two other Riemannian metrics are respectively given by 
  \[
  (\mathrm{grad}f(\pi(Y)))_{\uparrow_Y} = \mathrm{grad}\bar{f}(Y) =
  \begin{cases}
    2\nabla h(YY^T)Y(Y^TY)^{-1} & \text{ under metric } g^{2}, \\ 
    2\nabla h(YY^T)Y  & \text{ under metric } g^{3},
  \end{cases}
  \]
and the actions of the Riemannian Hessians are respectively given by
\begin{subnumcases} {\label{IngredQuotMani-Hessian} (\mathrm{Hess}f(\pi(Y))[\eta_{\pi(Y)}])_{\uparrow_Y}=}
2\nabla^2h(YY^T)[Y\eta_{\uparrow_Y}^T+\eta_{\uparrow_Y}Y^T]Y(Y^TY)^{-1}  + T_2 & $i = 2,$ \label{IngredQuotMani-Hessian2} \\
2\nabla^2h(YY^T)[Y\eta_{\uparrow_Y}^T+\eta_{\uparrow_Y}Y^T]Y  + T_3  & $i =3,$ \label{IngredQuotMani-Hessian3}
\end{subnumcases}
where $\nabla^2h(YY^T)[V] = AVM + MVA$, $P_Y=Y(Y^TY)^{-1}Y^T$, $ T_2 = \mathcal{P}_Y^{\mathrm{H}^2}\{ 
            \nabla h(YY^T)P_Y^\perp \eta_{\uparrow_Y}(Y^TY)^{-1}+P_Y^\perp \nabla h(YY^T)\eta_{\uparrow_Y}(Y^TY)^{-1} + 2\mathrm{skew}(\eta_{\uparrow_Y}Y^T)\nabla h(YY^T)Y(Y^TY)^{-2}  + 2\mathrm{skew}\{ \eta_{\uparrow_Y}(Y^TY)^{-1}Y^T\nabla h(YY^T) \}Y(Y^TY)^{-1}\}$, and $T_3 = 2\mathcal{P}_Y^{\mathrm{H}^3}\{ \nabla h(YY^T)\eta_{\uparrow_Y} \}$. 
\end{proposition}

\subsection{Preconditioning under Riemannian Metrics~\eqref{IngredQuotMani-Metric2} and~\eqref{IngredQuotMani-Metric3}}

Similar to the approach in Section~\ref{sec:IncRankAlgPrecond:Precond}, the preconditioners for Riemannian metrics~\eqref{IngredQuotMani-Metric2} and~\eqref{IngredQuotMani-Metric3} respectively solves the equations
\begin{equation} \label{e06}
[A(Y\xi_{\uparrow_Y}^T+\xi_{\uparrow_Y}Y^T)M+M(Y\xi_{\uparrow_Y}^T+\xi_{\uparrow_Y}Y^T)A]Y=\eta_{\uparrow_Y}(Y^TY), \hbox{ for $\xi_{\uparrow_Y} \in \mathrm{H}_Y^2$ }
\end{equation}
and
\begin{equation} \label{e07}
  2[A(Y\xi_{\uparrow_Y}^T+\xi_{\uparrow_Y}Y^T)M+M(Y\xi_{\uparrow_Y}^T+\xi_{\uparrow_Y}Y^T)A]Y=\eta_{\uparrow_Y}, \hbox{ for $\xi_{\uparrow_Y} \in \mathrm{H}_Y^3$}.
\end{equation}
The derivations are analogous to those in Section~\ref{sec:IncRankAlgPrecond:Precond} and therefore are not repeated here. The algorithms for solving~\eqref{e06} and~\eqref{e06} are respectively stated in Algorithm~\ref{Precond-Metric2-Alg} and Algorithm~\ref{Precond-Metric3-Alg}. Note that in Step~\ref{Precond-Metric3-Sum-Sxi} and~\ref{Precond-Metric3-for-Omega} of Algorithm~\ref{Precond-Metric3-Alg}, two small-scale Sylvester equations of size $p\times p$ need be solved and can be done by using \textit{lyap} function in MATLAB. 

\begin{algorithm}[htbp]
  \caption{Preconditioner under Riemannian metric~\eqref{IngredQuotMani-Metric2}} 
  \begin{algorithmic}[1] \label{Precond-Metric2-Alg}
  \REQUIRE Matrices $A$ and $M$ and horizontal vector $\eta_{\uparrow_Y}\in\mathrm{H}_Y^2$;
  \ENSURE $\xi_{\uparrow_Y}$ satisfying~\eqref{e06};
  \STATE Set $LL^T\gets Y^TMY$ (Cholesky factorization);
  \STATE Set $Q\Lambda Q^T\gets L^{-1}Y^TAYL^{-T}$ (Eigenvalues decomposition); 
  \STATE Set $\hat{V}\gets \mathrm{orthonormal}(MY)$;
  \STATE Set $v_i\gets Q^TL^{-1}Y^TA\mathcal{T}_i^{-1}\left( (I-\hat{V}\hat{V}^T)\frac{1}{2}\eta_{\uparrow_Y}(Y^TY)L^{-T}Q(:,i)\right)$;
  \STATE Set $K_i\gets 2\lambda_iI-Q^TL^{-1}Y^TA\mathcal{T}_i^{-1}\left( (I-\hat{V}\hat{V}^T)2AYL^{-T}Q\right)$;
  \STATE Set $R\gets \frac{1}{2}Q^TL^{-1}Y^T\eta_{\uparrow_Y}(Y^TY)L^{-T}Q-[v_1,\cdots,v_k]-[v_1,\cdots,v_k]^T$;
  \STATE Solve for $\tilde{S}_\xi$ by $(\mathcal{K}+\Pi\mathcal{K}\Pi)\mathrm{vec}(\tilde{S}_\xi)=\mathrm{vec}(R)$;
  \STATE Solve for $\tilde{Z}_\xi$ by   
    \begin{small}
    \begin{align}  \nonumber
  	 \tilde{Z}_\xi (:, i)\gets \mathcal{T}_i^{-1}\left((I-\hat{V}\hat{V}^T)\frac{1}{2}\eta_{\uparrow_Y}(Y^TY)L^{-1}Q(:,i)\right) - \mathcal{T}_i^{-1}\left((I-\hat{V}\hat{V}^T)2AYL^{-T}Q\right)\tilde{S}_\xi(:,i);
    \end{align}
  	\end{small}  
  \STATE Set $Z_{\xi,1}\gets \frac{1}{2}Y\left( (Y^TY)^{-1}Y^T\tilde{Z}_\xi Q^TL^{-1} + L^{-T}Q\tilde{Z}_{\xi}^TY(Y^TY)^{-1}  \right)$;
  \STATE Set $Z_{\xi,2}\gets (I - Y(Y^TY)^{-1}Y^T\tilde{Z}_{\xi}Q^TL^{-1}$;
  \STATE Set $\xi_{\uparrow_Y}\gets YL^{-T}Q\tilde{S}_\xi Q^TL^{-1} + Z_{\xi,1} + Z_{\xi,2}$.
  
 \end{algorithmic}
\end{algorithm}

\begin{algorithm}[htbp]
  \caption{Preconditioner under Riemannian metric~\eqref{IngredQuotMani-Metric3}} 
  \begin{algorithmic}[1] \label{Precond-Metric3-Alg}
  \REQUIRE Matrices $A$ and $M$ and horizontal vector $\eta_{\uparrow_Y}\in\mathrm{H}_Y^3$;
  \ENSURE $\xi_{\uparrow_Y}$ satisfying~\eqref{e07};
  \STATE Set $LL^T\gets Y^TMY$ (Cholesky factorization);
  \STATE Set $Q\Lambda Q^T \gets L^{-1}Y^TAYL^{-T}$ (Eigenvalues decomposition); 
  \STATE Set $\hat{V}\gets \mathrm{orthonormal}(MY)$;
  \STATE Set $v_i\gets Q^TL^{-1}Y^TA\mathcal{T}_i^{-1}\left( (I-\hat{V}\hat{V}^T)\frac{1}{2}\eta_{\uparrow_Y}L^{-T}Q(:,i)\right)$;
  \STATE Set $K_i\gets 2\lambda_iI-Q^TL^{-1}Y^TA\mathcal{T}_i^{-1}\left( (I-\hat{V}\hat{V}^T)2AYL^{-T}Q\right)$;
  \STATE Set $R\gets \frac{1}{2}Q^TL^{-1}Y^T\eta_{\uparrow_Y}L^{-T}Q-[v_1,\cdots,v_k]-[v_1,\cdots,v_k]^T$;
  \STATE Solve for $\tilde{S}_\xi$ by $(\mathcal{K}+\Pi\mathcal{K}\Pi)\mathrm{vec}(\tilde{S}_\xi)=\mathrm{vec}(R)$;
  \STATE Solve for $\tilde{Z}_\xi$ by 
  \begin{small}  	  
  \begin{align*} \nonumber
	  \tilde{Z}_\xi (:, i)\gets \mathcal{T}_i^{-1}\left((I-\hat{V}\hat{V}^T)\frac{1}{2}\eta_{\uparrow_Y}L^{-1}Q(:,i)\right) - \mathcal{T}_i^{-1}\left((I-\hat{V}\hat{V}^T)2AYL^{-T}Q\right)\tilde{S}_\xi(:,i);
  \end{align*}
  \end{small}
  \STATE Solve for $S_\xi$ by $S_\xi(Y^TY)^{-1}+(Y^TY)^{-1}S_\xi=2L^{-T}Q\tilde{S}_\xi Q^TL^{-1}$; \label{Precond-Metric3-Sum-Sxi}
  \STATE Solve for $\Omega$ by $\Omega Y^TY + Y^TY \Omega = Y^T\tilde{Z}_{\xi}Q^TL^{-1} - L^{-T}Q\tilde{Z}_{\xi}^TY$; \label{Precond-Metric3-for-Omega}
  \STATE Set $\xi_{\uparrow_Y}\gets Y(Y^TY)^{-1}S_\xi+\tilde{Z}_\xi Q^{T}L^{-1} - Y\Omega$. 
  
 \end{algorithmic}
\end{algorithm}

%It is notable that in Step \ref{Precond-Metric3-Sum-Sxi} and \ref{Precond-Metric3-for-Omega}, we need solve two small-scale Sylvester equations of size $p\times p$ and it is done by using $lyap$ function in MATLAB. 

%\whcomm{==============================}{}

\end{document}